\documentclass[11pt,twoside]{article}
\usepackage[a4paper,textwidth=16.5cm,textheight=24.5cm,left=2.25cm,right=2.25cm,top=3cm,headheight=15pt]{geometry}
\usepackage{amsfonts}
\usepackage{amsmath,amssymb}
\usepackage{multicol}
\usepackage{graphics}
\usepackage[numbers]{natbib}
\usepackage{stmaryrd}
\usepackage{ulem}
\usepackage{bm}
 \usepackage[colorlinks=true,linkcolor=blue,anchorcolor=blue,citecolor=blue,bookmarks=true]{hyperref}
\usepackage{color}
\usepackage{graphicx}
\usepackage{float}  
 \usepackage{graphicx}
\usepackage{subcaption}
\usepackage{accents}

\allowdisplaybreaks [2]

\newcommand{\N}{\mathbb{N}}
\newcommand{\R}{\mathbb{R}}

\newcommand{\vt}{\vartheta}

\def\a{\alpha}
\def\b{\beta}
\def\e{\varepsilon}
\def\D{\Delta}
\def\d{\delta}
\def\g{\gamma}

\def\m{\mu}
\def\n{\nu}

\def\o{\omega}
\def\O{\Omega}
\def\p{\partial}
\def\r{\rho}
\def\vr{\varrho}
\def\S{\Sigma}
\def\s{\sigma}
\def\vs{\varsigma}
\def\vphi{\varphi}
\def\ov{\overline}
\def\un{\underline}

\def\si{\sin}
\def\c{\cos}

\def\li{\lim\limits}
\def\lii{\liminf\limits}

\usepackage{comment}
\usepackage{cases}
\usepackage{indentfirst}
\usepackage{amsmath}
\usepackage{enumerate}
\usepackage{enumitem}
\usepackage{amsthm}

\usepackage{mathtools}
\newcommand{\usim}[1]{\underaccent{\mathrlap{\sim}\hphantom{#1}}{#1}}

\newtheorem{theorem}{Theorem}[section]
\newtheorem{lemma}[theorem]{Lemma}

\newtheorem{proposition}[theorem]{Proposition}
\newtheorem{claim}[theorem]{Claim}
\numberwithin{equation}{section}

\theoremstyle{definition}
\newtheorem{remark}[theorem]{Remark}
\newtheorem{definition}[theorem]{Definition}
\renewenvironment{proof}[1][Proof]
{\par\noindent\textbf{#1.}\hspace{0.5em}}
{\hfill$\Box$\par}

\allowdisplaybreaks \numberwithin{equation}{section}
\begin{document}

\author{\bf Yaqian Xu$^1$, Shigui Ruan$^2$\thanks{Email: ruan@math.miami.edu} and Zhi-Cheng Wang$^1$\thanks{Email: wangzhch@lzu.edu.cn} \\
$^1$School of Mathematics and Statistics, Lanzhou University,
 \\
Lanzhou, Gansu 730000, People's Republic
of China\\
$^2$ Department of Mathematics, University of Miami,\\
 Coral Gables, FL 33146, USA}
\title{\textbf{Asymptotic spreading speeds of a road-field reaction-diffusion predator-prey model\thanks{Research of Y. Xu and Z.-C. Wang was partially supported by  National Natural Science Foundation of China (12471164). Research of S. Ruan was partially supported by the National Science Foundation (DMS-2424605).} } }
 \date{}
\maketitle
 \begin{abstract}
This paper is devoted to studying asymptotic spreading speeds of a road-field reaction-diffusion predator-prey system, in which predators diffuse faster on the straight line $\R\times \{0\}$ than in the half plane $\R\times (0,+\infty)$. We give a sharp characterization of the spreading properties of solutions of the system with compactly supported initial values by appealing to the invasion speed of the prey when predators are absent, and the invasion speed of predators when the prey is saturated. Notably, since the line may accelerate the expansion of predators in a cone of directions, the invasion speed of predators varies with the direction. To overcome the difficulty arising from the coupling between the predator-prey interaction and the road-field interaction, we first derive pointwise comparisons between predators and the prey and then reduce the components of predators in the system to a truncated field-road system depending on the invasion speed of the prey. By using known results for the single-species field-road model, we further show the propagation properties of the truncated system. This step is very crucial. Once the propagation properties of the truncated system are established, we can obtain the expected spreading properties of solutions to the original system by applying the comparison principle and contradiction arguments, and constructing appropriate Lyapunov functionals.


  \end{abstract}


\section{Introduction}\label{introduction}
It is well-known that fast diffusion of species on roads, rivers, pipelines, and corridors can have a driving impact on their population dynamics, invasion speeds, and balance of the entire ecosystem. For example, wolves travel two to three times faster on human-made linear features, such as roads, pipelines and resource exploration lines (i.e. seismic lines), compared to moving through a natural forest \cite{DM,mckenzie+2012}. The pine processionary moth extends its geographical distribution rapidly through human-mediated transport of pupae in soil or trees, causing long-distance jumps far beyond its natural flight range \cite{robinet+2012}. Human-mediated dispersal also accelerates the spread of the yellow-legged hornet by carrying pregnant founder queens over hundreds of miles through vehicles \cite{RS}.

Another example is that {\it Aedes albopictus}, the most invasive mosquito species native to the tropical and subtropical areas of Southeast Asia, has spread to many countries and territories recently via international travel and the transport of goods \cite{bonizzoni+2013,Ha,Kr}. Once imported into a country, the mosquito would invade along roads by vehicles. For example, the spread and expansion of {\it Ae. albopictus} in the whole state of Florida was facilitated by the movement of used tires along the interstate highways \cite{LB,MM,TZ}. {\it Ae. albopictus} was observed for the first time in 2004 in Catalonia, northeastern Spain, and by 2014 the species was well-established in municipalities along the coastline of mainland Spain \cite{Co}. It was reported that {\it Aedes} species had a high infestation rate in garages trading used tires along the highways, providing a conduit for rapid dispersal across Panama \cite{Be}.

To study the effect of strong diffusion of a single species on a line, Berestycki et al.~\cite{berestycki+2013} proposed the following road-field model:
\begin{equation}\label{berestycki+2013 1.1}
    \begin{cases}
        \partial_tv-d\Delta v=f(v), &(x,y)\in\O_0,t>0,\\
		\partial_tw-D\partial_{xx}w=\nu v(x,0,t)-\mu w, &x\in\mathbb{R},t>0,\\
         -d\partial_{y}v(x,0,t)=\mu w(x,t)-\nu v(x,0,t), \qquad&x\in\mathbb{R},t>0,
    \end{cases}
\end{equation}
where $w$ and $v$ represent the densities of the same species on the {\it road} $\R\times\{0\}$ and in the {\it field} $\O_0:=\R\times \R_+$, respectively. Parameters $\m$ and $\n$ quantify the interaction strength of this population between the two spatial domains. The reaction term $f$ is of KPP type, that is
$$
f(0)=f(1)=0,\qquad 0<f(s)\leq f^{\prime}(0)s,~\forall s\in (0,1),\qquad f<0 ~\text{in}~(1,+\infty).
$$
In \cite{berestycki+2013}, the authors found that the system admits an asymptotic spreading speed in the direction of the road and the propagation is enhanced by the road if $D>2d$.
 The elevated threshold for $\frac{D}{d}$ ($2$ versus the anticipated $1$) stems from the absence of reproduction on the road. If there are additionally reaction term and drift term on the line, then system \eqref{berestycki+2013 1.1} changes into the following system
\begin{equation}\label{12}
\begin{cases}
    \partial_tv-d\Delta v=f(v),&(x,y)\in\O_0,t>0,\\
	\partial_tw-D\partial_{xx}w+q\p_x w=\nu v(x,0,t)-\mu w(x,t)+g(w),&x\in\mathbb{R},t>0,\\
     -d\partial_{y}v(x,0,t)=\mu w(x,t)-\nu v(x,0,t),&x\in\mathbb{R},t>0,
\end{cases}
\end{equation}
where $q\in\R$ and $g \in C^1([0, +\infty))$ satisfies
\[
	g(0) = 0,\qquad  \exists S > 0,~~ s.t. ~~ g(S) \leq 0,\qquad s\mapsto \frac{g(s)}{s} \text{ is nonincreasing}.
\]
In this case, the thresholds of $\frac{D}{d}$ to accelerate species propagation along $\mathbf{e}_{1}=(1,0)$ and $-\mathbf{e}_{1}=(-1,0)$ are $2-\frac{g^{\prime}(0)}{f^{\prime}(0)}-\frac{q}{\sqrt{df^{\prime}(0)}}$ and $2-\frac{g^{\prime}(0)}{f^{\prime}(0)}+\frac{q}{\sqrt{df^{\prime}(0)}}$, respectively (see \cite{berestycki+2013a}).
Note that the authors in \cite{berestycki+2013} and \cite{berestycki+2013a} only studied the spreading speeds along the directions $\pm\mathbf{e}_{1}$. Subsequently, in \cite{berestycki+2016}, they extended their investigation to all planar directions and proved that the line enhances the propagation in all directions that deviate from the normal to the road by more than a certain angle (see Lemma \ref{berestycki+2016 thm2.1}). Define the solution $(v(x,y,t),w(x,t))=\left(\vphi(x-ct,y),\psi(x-ct)\right)$ of \eqref{berestycki+2013 1.1} with
\[
	\vphi(-\infty,y)=1,~\vphi(+\infty,y)=0 \text{ locally uniformly in } y\geq 0,\qquad\psi(-\infty)=\dfrac{\n}{\m}, ~\psi(+\infty)=0, 
\]
as traveling front solution. In \cite{berestycki+2016a}, they showed that the existence of such a solution depends on whether $c$ exceeds the invasion speed along the road.
Similar results also hold for \eqref{berestycki+2013 1.1} incorporating both a drift term and a pure death term on the road, corresponding to \eqref{12} with $g(w)=\r w$ for some constant $\r<0$. This modified system admits a unique, positive, bounded and steady solution $(v_s(y),w_s)$ that is independent of $x$. In this scenario, $\vphi(\cdot)$ and $\psi(\cdot)$ satisfy $\vphi(-\infty,y)=v_s(y)$ locally uniformly in $y\geq 0$ and $\psi(-\infty)=w_s$.

Recently, this road-field model and its variants have been rigorously characterized.
By exploring the features of {\it generalized principal eigenvalues}, Giletti et al. \cite{giletti+2015} extended the results in \cite{berestycki+2013} to a more general framework, where the exchange terms $\nu$ and $\mu$ depend periodically on spatial variable $x$.
While Zhang \cite{zhang+2021} considered the case when the reaction term $f$ is spatially periodic in $x$, and established the existence of the asymptotic speed, which is also the minimal speed of pulsating waves. 
Based on the notion of generalized principal eigenvalue of \eqref{berestycki+2013 1.1} linearized at $(v,w)=(0,0)$, Berestycki et al.~\cite{berestycki+2020} investigated the  effect of a line on an ecological niche. Later they developed the properties of generalized principle eigenvalues for heterogeneous road-field systems \cite{berestycki+2020a}.
The fields in the previous articles are all half-plane, whereas the effect of the road on the propagation over different fields can be found in several studies. For example, Tellini \cite{tellini+2016}, Rossi et al. \cite{rossi+2017}, and Bogosel et al. \cite{bogosel+2021} considered the case of cylindrical domains. The fields they studied became increasingly general. The long time behavior of the species in conical fields was studied in \cite{ducasse+2018,henderson+2024a}. In particular, Henderson and Lam \cite{henderson+2024a} applied the Hamilton-Jacobi method to the road-field model and obtained very interesting results. To model the phenomenon that transportation networks promote epidemic transmission, Berestycki et al.~\cite{berestycki+2021} modified the classical SIR model into the so-called SIRT model by dividing the infected individuals into two groups---those on the road ($T$) and those in the field ($I$), as represented by the following system
\begin{equation}\label{SIRT}
\begin{cases}
            \partial_tI-d\Delta I+\alpha I=\beta SI, \qquad & (x,y)\in\O_0,t>0,\\
           \partial_tS=-\beta SI, \qquad &(x,y)\in\O_0,t>0,\\
            \partial_tT-D\partial_{xx}T=\nu I(x,0,t)-\mu T,\qquad &x\in\mathbb{R},t>0,\\
           -d\partial_yI(x,0,t)=\mu T(x,t)-\nu I(x,0,t),\qquad&x\in\mathbb{R},t>0.
        \end{cases}
\end{equation}
We also refer the readers to \cite{berestycki+2014a,berestycki+2015b,berestycki+2024} for nonlocal dispersal systems, to \cite{dietrich+2015,dietrich+2016,dietrich+2017} for more general reaction terms, and to \cite{pauthier+2015,pauthier+2016} for nonlocal exchange terms.

Roads can also expand the territorial range of predators, exposing prey to higher survival risks. For instance, seismic lines may be created during gas and oil exploration, McKenzie et al. \cite{mckenzie+2012} and Dickie et al. \cite{dickie+2016} found that wolves ({\it Canis lupus}) travelled faster along these lines. Dingoes and foxes showed a strong preference for activity along roads and tended to travel along them \cite{raiter+2018}. DeMars and Boutin \cite{DB} found that roads increased predator selection of peatlands. For more references on predator-prey interactions disrupted by roads, we refer to the review of Quiles and Barrientos \cite{QB}.
The literature cited in previous paragraphs has focused on the single species case. To the best of the authors' knowledge, no studies have examined the dynamics of road-field predator-prey systems. In \eqref{SIRT}, the susceptible and infected individuals in the field, i.e. $S$ and $I$, display dynamics analogous to prey-predator interactions. But there is no diffusion term in the $S$-equation, which allows the system to be transformed via
\[
	v(x,y,t):=\int_0^tI(x,y,s)\mathrm{d}s,\quad w(x,t):=\int_0^tT(x,s)\mathrm{d}s
\]
into
\[\left\{
	\begin{aligned}
		&\partial_t v-d\Delta v=f(v)+I_0(x,y), \qquad &&(x,y)\in\O_0,t>0,&\\
		&\partial_t w-D\partial_{xx}w=\nu v(x,0,t)-\mu w+T_0(x), \qquad &&x\in\mathbb{R},t>0,&\\
		&-d\partial_y v(x,0,t)=\mu w(x,t)-\nu v(x,0,t),\qquad &&x\in\mathbb{R},t>0.&
	\end{aligned}
\right.\]
The last system is identical to \eqref{berestycki+2013 1.1}, except for the “source” terms $I_0$ and $T_0$.

To explore the predator-prey system coupled with a second-order parabolic equation on a straight line, where fast diffusion of predators can potentially occur, we consider the following system
\begin{equation}\label{prey-predator-road}
	\left\{
		\begin{aligned}
			&\partial_{t}u-d_{1}\Delta u=u(1-u-a v),&&\qquad(x,y)\in \O_0,t>0,&\\
			&\partial_{t}v-d_{2}\Delta v=v(-1+bu-v),&&\qquad(x,y)\in \O_0,t>0,&\\
			&\partial_{t}w-Dw_{xx}=-\mu w+\nu v(x,0,t),&&\qquad x\in \mathbb{R},t>0,& \\
			&\p_y u(x,0,t)=0,&& \qquad x\in \mathbb{R},t>0,&\\
			&-d_{2}\partial_{y}v(x,0,t)=\mu w(x,t)-\nu v(x,0,t),&&\qquad x\in \mathbb{R},t>0,&
			\end{aligned}
	\right.
\end{equation}
with
\[
  u(x,y,0)=u_0(x,y),v(x,y,0)=v_0(x,y),~\forall(x,y)\in \ov{\O}_0, \qquad w(x,0)=w_0(x),~\forall x\in\R.
\]
Here, $u(x,y,t)$ and $v(x,y,t)$ represent the population density of the prey and predators respectively at time $t$ and location $(x,y)$ in the field, while $w(x,t)$ represents the population density of predators at time $t$ and location $x$ on the road. 
Parameters $d_{1}, d_{2}, D, \mu, \nu>0$ and $a>0, b>1$
are given constants. We complete this system by specifying initial conditions:
\begin{enumerate}[label=\textbf{(H\arabic*)}]
	\item \label{hy-initial} The initial datum $u_{0}(x,y),v_{0}(x,y),w_{0}(x)$ are compactly supported and continuous functions. Moreover, $0\le u_{0}\le 1,\, 0\le v_{0}\le b-1,\, 0\le w_{0}\le \frac{\nu}{\mu}(b-1)$ and $u_0\not\equiv 0,(v_0,w_0)\not\equiv(0,0)$.
\end{enumerate}

If the influence of the road is neglected, system \eqref{prey-predator-road} reduces to the classical reaction-diffusion predator-prey system
\[
	\left\{
		\begin{aligned}
			&\partial_{t}u-d_{1}\Delta u=u(1-u-a v),&&\qquad(x,y)\in \O_0,t>0,&\\
			&\partial_{t}v-d_{2}\Delta v=v(-1+bu-v),&&\qquad(x,y)\in \O_0,t>0,&\\
			&\p_y u(x,0,t)=\p_y v(x,0,t)=0,&& \qquad x\in \mathbb{R},t>0.&
			\end{aligned}
	\right.
\]
Ducrot et al.~\cite{ducrot+2019} investigated the spreading properties of diffusive predator-prey systems deeply.
Ducrot and Jin \cite{ducrot+2023a} gave an easier proof for time heterogeneous predator-prey systems by deriving pointwise comparisons between the two species. See also \cite{ducrot+2016,ducrot+2021,guo+2022} and the references therein for more results on diffuse predator-prey systems. 

The rest of this paper is organized as follows. Section \ref{preliminaries} presents some preliminaries and the main results of this paper.
In Section \ref{upper}, we investigate the spreading behavior of \eqref{prey-predator-road} under initial conditions \ref{hy-initial}.
Finally, Section \ref{positive} provides a short argument for the case where $u(x,y,0)$ admits a positive lower bound.

\section{Preliminaries and Main Results}\label{preliminaries}
In this section, we first provide some preliminaries---proving the well-posedness of system \eqref{prey-predator-road} and reviewing existing results. Then we present the main results of this paper. Throughout this paper, if $\xi=(\xi_1,\xi_2,\cdots,\xi_n), \zeta=(\zeta_1,\zeta_2,\cdots,\zeta_n)\in\R^n$, then we write $\xi\geq\zeta$ whenever $\xi_{i}\geq\zeta_{i}$ holds for $i=1,2,\cdots,n$, and $\xi>\zeta$ whenever $\xi_{i}>\zeta_{i}$ holds for $i=1,2,\cdots,n$.
\subsection{Well-Posedness}
We first prove the well-posedness of the following system
	\begin{equation}\label{drift}
		\left\{\begin{array}{ll}
			\partial_{t}u-d_1\Delta u =u(1-u-av),&\qquad(x,y)\in \O_0,t>0,\\
			\partial_{t}v-d_2\Delta v =v(-1+bu-v),&\qquad(x,y)\in \O_0,t>0,\\
			\partial_{t}w-D\p_{xx}w =-\mu w+ \n v(x,0,t),&\qquad x\in \mathbb{R},t>0,\\
			\p_y u(x,0,t)=0,& \qquad x\in \mathbb{R},t>0,\\
			-d_2\partial_{y}v(x,0,t)=\mu w-\n v(x,0,t),&\qquad x\in \mathbb{R},t>0,\\
			u(x,y,0)=u_0(x,y), v(x,y,0)=v_0(x,y),w(x,0)=w_0(x), &\qquad(x,y)\in \ov{\O}_0.
		\end{array}\right.
	\end{equation}
 The initial functions $u_0(x,y),\, v_0(x,y),\, w_0(x)$ satisfy the following assumption:
\begin{enumerate}[label=\textbf{(H\arabic*)}]
\setcounter{enumi}{1}
	\item\label{h2} The initial data $u_{0}(x,y),\, v_{0}(x,y),\, w_{0}(x)$ are uniformly continuous functions. Moreover, $u_0\not\equiv 0, (v_0,w_0)\not\equiv (0,0)$ and $(0,0,0)\le (u_{0}, v_{0}, w_{0})\le \left(1,b-1,\frac{\nu}{\mu}(b-1)\right)$.
\end{enumerate}
For notational simplicity, we denote $f_1(u,v):=u(1-u-av)$ and $f_2(u,v):=v(-1+bu-v)$.
\begin{lemma}\label{lem 23}
	System \eqref{drift} has a unique nonnegative solution $(u,v,w)$ on $\ov{\O}_0\times[0,+\infty)$, and $(u,v,w)\leq \left(1,b-1,\frac{\nu}{\mu}(b-1)\right)$ on $\ov{\O}_0\times [0,+\infty)$. In addition,
\[
	\renewcommand{\arraystretch}{1.25}
\begin{array}{c}
	u,v\in C^{2+\alpha,1+\frac{\alpha}{2}}\left(\O_0\times (0,+\infty)\right)\cap C^{1+\alpha,\frac{1+\alpha}{2}}\left(\ov{\O}_0\times(0,+\infty)\right)\cap C\left(\ov{\O}_0\times [0,+\infty)\right),\\
	w\in C^{2+\alpha,1+\frac{\alpha}{2}}(\R\times (0,+\infty))\cap C\left(\R\times [0,+\infty)\right).
\end{array}
\]

\end{lemma}
\begin{proof}
Set $\left(\overline{u}^{0},\overline{v}^{0},\overline{w}^{0}\right)=\left(1,b-1,\frac{\nu}{\mu}(b-1)\right)$ and $\left(\un{u}^0,\un{v}^0,\un{w}^0\right)=(0,0,0)$. To simplify notation, define operators $\mathcal{L}_{i}~(i=1,2,3)$ by     $ \mathcal{L}_{1}u:=\partial_{t}u-d_1\Delta u+lu$,  
$ \mathcal{L}_{2}v:=\partial_{t}v-d_2\Delta v+l v$ and $\mathcal{L}_{3}w:=\partial_{t}w-D\p_{xx}w +\mu w$,
wherein $l$ is the Lipschitz constant of both $f_1$ and $f_2$ on $\left[0,1\right]\times\left[0,b-1\right]$. 

\vspace{1.25ex}
\noindent\textit{Step 1. 
Construction of sequences $\left\{(\ov{u}^k,\ov{v}^k,\ov{w}^k)\right\}$ and $\left\{\left(\underline{u}^{k},\underline{v}^{k},\underline{w}^{k}\right)\right\}$.}
\vspace{1.25ex}

Starting from initial iterations $\left(\overline{u}^{0},\overline{v}^{0},\overline{w}^{0}\right)$ and $\left(\underline{u}^{0},\underline{v}^{0},\underline{w}^{0}\right)$, we construct two sequences $\left\{(\ov{u}^k,\ov{v}^k,\ov{w}^k)\right\}_{k}$ and $\left\{\left(\underline{u}^{k},\underline{v}^{k},\underline{w}^{k}\right)\right\}_{k}$ from the following iteration process

\begin{subequations}
	\begin{equation}\label{u-sup}
	\left\{\begin{array}{ll}
        \mathcal{L}_{1}\overline{u}^{k}=f_{1}(\overline{u}^{k-1},\underline{v}^{k-1})+l \overline{u}^{k-1},&\qquad (x,y)\in \O_0,t>0,\\
		\mathcal{L}_{1}\underline{u}^{k}=f_{1}(\underline{u}^{k-1},\overline{v}^{k-1})+l\un{u}^{k-1},&\qquad (x,y)\in \O_0,t>0,\\
		\p_y\ov{u}^{k}(x,0,t)=\p_y\un{u}^{k}(x,0,t)= 0, & \qquad x\in \mathbb{R},t>0, \\
		\overline{u}^{k}(x,y,0)=\underline{u}^{k}(x,y,0)=u_{0}(x,y), &\qquad (x,y)\in \ov{\O}_0,
	\end{array}\right.
\end{equation}
\begin{equation}\label{v-sup}
	\left\{\begin{array}{ll}
        \mathcal{L}_{2}\overline{v}^{k}=f_{2}(\overline{u}^{k-1},\overline{v}^{k-1})+l \overline{v}^{k-1},\quad & (x,y)\in \O_0,t>0, \\
		\mathcal{L}_{2}\underline{v}^{k}=f_{2}(\underline{u}^{k-1},\underline{v}^{k-1})+l\underline{v}^{k-1},\quad & (x,y)\in \O_0,t>0, \\
        \n\overline{v}^{k}(x,0,t)-d_2\partial_{y}\overline{v}^{k}(x,0,t)=\mu \overline{w}^{k-1}, \quad & x\in \mathbb{R},t>0, \\
		\n\un{v}^{k}(x,0,t)-d_2\partial_{y} \un{v}^{k}(x,0,t)=\mu\un{w}^{k-1}, \qquad & x\in \mathbb{R},t>0, \\
		\overline{v}^{k}(x,y,0)= \underline{v}^{k}(x,y,0)=v_{0}(x,y), \quad &(x,y)\in \ov{\O}_0,
	\end{array}\right.
\end{equation}
\begin{equation}\label{w-sup}
	\left\{ \begin{array}{ll}
		\mathcal{L}_{3} \overline{w}^{k}=\n\overline{v}^{k-1}(x,0,t),&\quad x\in \mathbb{R},t>0,\\
		\mathcal{L}_{3} \underline{w}^{k}=\n\underline{v}^{k-1}(x,0,t),&\quad x\in \mathbb{R},t>0,\\
		\overline{w}^{k}(x,0)=\underline{w}^{k}(x,0)=w_{0}(x),& \quad x\in \mathbb{R},
		\end{array}\right.
\end{equation}
\end{subequations}
For any $k=1,2,\cdots$, the existence of triplets $\left(\ov{u}^k,\ov{v}^k,\ov{w}^k\right)$ and $\left(\un{u}^k,\un{v}^k,\un{w}^k\right)$ can be obtained by similar arguments as those in \cite[Lemmas 7.2.1 and 7.3.3]{pao+1993}. In particular, one has $\ov{u}^k,\ov{v}^k, \un{u}^k,\un{v}^k\in C^{2,1}\left(\O_0\times(0,+\infty) \right)\cap C\left(\overline{\Omega}_0\times[0,+\infty)\right)$, $\overline{w}^k, \underline{w}^k\in C^{2,1}\left(\R \times(0,+\infty)\right)\cap C\left(\R\times[0,+\infty)\right)$. 

\vspace{1.25ex}
\noindent\textit{Step 2.
Monotonicity of $\left\{(\ov{u}^k,\ov{v}^k,\ov{w}^k)\right\}$ and $\left\{\left(\underline{u}^{k},\underline{v}^{k},\underline{w}^{k}\right)\right\}$.}
\vspace{1.25ex}

The two sequences satisfy the following monotone property:

\vspace{1.5ex}
   \noindent {\bf Claim.} {\it The two sequences $\left\{(\overline{u}^{k},\overline{v}^{k},\overline{w}^{k})\right\}_{k}$ and $\left\{(\underline{u}^{k},\underline{v}^{k},\underline{w}^{k})\right\}_{k}$ given by (\ref{u-sup})-(\ref{w-sup}) possess the monotone property
   	\begin{equation}\label{monotone sequences}
   		\begin{array}{c}
		\underline{u}^{k}\le \underline{u}^{k+1}\le \overline{u}^{k+1}\le\overline{u}^{k},\quad\underline{v}^{k}\le \underline{v}^{k+1}\le \overline{v}^{k+1}\le\overline{v}^{k}\quad \text{\rm in }\ov{\O}_0\times [0,+\infty), \\
		\underline{w}_{k}\le \underline{w}_{k+1}\le \overline{w}_{k+1}\le\overline{w}_{k}\quad \text{\rm in }\mathbb{R}\times [0,+\infty),
	\end{array}
   	\end{equation}
	where $k=0,1,2,\cdots$.}
\vspace{1.5ex}

\noindent {\it Proof of the Claim.}
Set $(\tilde{u}^{0},\tilde{v}^{0},\tilde{w}^{0})=(\overline{u}^{0}-\overline{u}^{1},\overline{v}^{0}-\overline{v}^{1},\overline{w}^{0}-\overline{w}^{1})$. For $\tilde{u}^0$ and $\tilde{w}^0$, one has
\[
\left\{\begin{array}{ll} 
	\mathcal{L}_{1}\tilde{u}^{0}= 0,\quad & (x,y)\in \O_0,t>0, \\
	\p_y\tilde{u}^{0}(x,0,t)=0,  & x\in \mathbb{R},t>0,  \\
	\tilde{u}^{0}(x,y,0)\ge 0, &  (x,y)\in \ov{\O}_0
\end{array}\right.
\quad
	\left\{\begin{array}{ll}
		\mathcal{L}_{3}\tilde{w}^{0}=\m\overline{w}^0-\n\overline{v}^0(x,0,t)= 0, \quad&x\in \mathbb{R},t>0,\\
		\tilde{w}^{0}(x,0)=\overline{w}^0(x,0)-w_{0}(x)\ge 0, & x\in \mathbb{R}. 
	\end{array}\right.
\]
	In view of the comparison principle, we have $\tilde{u}^{0}\geq 0$ and $\tilde{w}^{0}\ge 0$.
	It then follows that
$$\left\{
	\begin{array}{ll} 
        \mathcal{L}_{2}\tilde{v}^{0}=l\overline{v}^0-\ov{v}^0(-1+b\overline{u}^0-\overline{v}^0)-l\overline{v}^0= 0, & (x,y)\in \O_0,t>0,\\
		\n\tilde{v}^{0}(x,0,t)-d_2\partial_{y}\tilde{v}^{0}(x,0,t)= 0, \qquad & x\in \mathbb{R},t>0,\\
        \tilde{v}^{0}(x,y,0)\ge 0 &  (x,y)\in \ov{\O}_0.
    \end{array}
\right.$$
Applying the comparison principle yields $\overline{v}^{0}\ge\overline{v}^{1}$. Using the same argument, one gets $\underline{w}^{0}\le \underline{w}^{1}, \underline{u}^{0}\le\underline{u}^{1}$, and $\underline{v}^{0}\le \underline{v}^{1}$. Let $\left(\hat{u}^{1},\hat{v}^{1},\hat{w}^{1}\right)=\left(\overline{u}^{1}-\underline{u}^{1},\overline{v}^{1}-\underline{v}^{1},\overline{w}^{1}-\underline{w}^{1}\right)$. By the monotonicity of $f_{1}(u,v)+l u$ and $f_{2}(u,v)+l v$, one has
$$\left\{
    \begin{array}{ll}
        \mathcal{L}_{1}\hat{u}^{1}=\left[ f_{1}(\overline{u}^0,\underline{v}^0)+l\overline{u}^0 \right] -\left[ f_{1}(\underline{u}^{0},\overline{v}^0)+l\underline{u}^0 \right] \ge 0, & (x,y)\in \O_0,t>0, \\
        \p_y\hat{u}^{1}(x,0,t)=0, & x\in \mathbb{R},t>0, \\
        \hat{u}^{1}(x,y,0)=0, & (x,y)\in \ov{\O}_0,
    \end{array}
\right.$$
$$\left\{
    \begin{array}{ll}
        \mathcal{L}_{2}\hat{v}^{1}=\left[ f_{2}(\overline{u}^0,\overline{v}^0) +l  \overline{v}^0\right] -\left[ f_{2}(\underline{u}^0,\underline{v}^0)+l\underline{v}^0 \right] \ge 0, & (x,y)\in \O_0,t>0, \\
        \n\hat{v}^{1}(x,0,t)-d_2\partial_{y}\hat{v}^{1}(x,0,t)=\mu [\overline{w}^0-\underline{w}^0]\ge 0, & x\in \mathbb{R},t>0, \\
        \hat{v}^{1}(x,y,0)=0, & (x,y)\in \ov{\O}_0,
    \end{array}
\right.$$
and
$$\left\{
\begin{array}{ll}
	\mathcal{L}_{3}\hat{w}^{1}=\n\overline{v}^0> 0,\quad & x\in \mathbb{R},t>0, \\
	\hat{w}^{1}(x,0)=0, & x\in \mathbb{R}.
\end{array}
\right.$$
These imply that $(\hat{u}^{1},\hat{v}^{1},\hat{w}^{1})\ge (0,0,0)$, which leads to $\left(\underline{u}^{0},\underline{v}^{0},\underline{w}^{0}\right)\leq \left(\underline{u}^{1},\underline{v}^{1},\underline{w}^{1}\right)\leq \left(\overline{u}^{1},\overline{v}^{1},\overline{w}^{1}\right)\leq \left(\overline{u}^{0},\overline{v}^{0},\overline{w}^{0}\right)$.
Assuming by induction that $\left(\underline{u}^{k-1},\underline{v}^{k-1},\underline{w}^{k-1}\right)\leq \left(\underline{u}^{k},\underline{v}^{k},\underline{w}^{k}\right)\leq \left(\overline{u}^{k},\overline{v}^{k},\overline{w}^{k}\right)\leq \left(\overline{u}^{k-1},\overline{v}^{k-1},\overline{w}^{k-1}\right)$
for some $k=1,2,3,\cdots$ and repeating the above processes, we have (\ref{monotone sequences}) and complete the proof of the claim.

\vspace{1.25ex}
\noindent\textit{Step 3.
Convergence of $\left\{(\ov{u}^k,\ov{v}^k,\ov{w}^k)\right\}$ and $\left\{\left(\underline{u}^{k},\underline{v}^{k},\underline{w}^{k}\right)\right\}$.}
\vspace{1.25ex}

Property \eqref{monotone sequences} ensures the following pointwise limits
\[
    (\hat{u},\hat{v},\hat{w})=\li_{ k \to \infty }(\overline{u}^{k},\overline{v}^{k},\overline{w}^{k}),\quad (\check{u},\check{v},\check{w})=\li_{ k \to \infty } (\underline{u}^{k},\underline{v}^{k},\underline{w}^{k}).
\]
Then we discuss the regularity and boundedness of these limit functions. Denote $F_1(u,v):=f_1(u,v)+l u$ and $F_2(u,v):=f_2(u,v)+l v$.
Since
\[
	(0,0,0)\leq \left(\un{u}^{k-1},\, \un{v}^{k-1},\, \un{w}^{k-1}\right)\leq \left(\ov{u}^{k-1},\, \ov{v}^{k-1},\, \ov{w}^{k-1}\right)\leq \left(1,b-1,\frac{\nu}{\mu}(b-1)\right),
\]
we have that
	$F_{1}\left(\ov{u}^{k-1},\, \un{v}^{k-1}\right),~ F_{2}\left(\ov{u}^{k-1},\, \ov{v}^{k-1}\right),~ F_{1}\left(\un{u}^{k-1},\ov{v}^{k-1}\right),~ F_{2}\left(\un{u}^{k-1},\, \un{v}^{k-1}\right)$
are uniformly bounded in $L^p_{loc}\left(\ov{\O}_0\times [0,+\infty)\right)$ for $p>4$ and $k=1,2,\cdots$. By using the global $L^p$ estimate for bounded domian \cite[Theorem 1.6]{wang+2021}, we obtain the uniform boundedness of $\ov{u}^k,\, \ov{v}^k$ and $\un{u}^k,\, \un{v}^k$ in $W^{2,1}_{p}\left(Q\times[\d,T]\right)$ for any bounded $C^{2}$ subdomain $Q\subset \O_0$  and constants $0<\d<T$.
Then the imbedding theorem \cite[Theorem 7.26]{gilbarg+2001} shows that  $\ov{u}^k,\ov{v}^k$ and $\un{u}^k,\un{v}^k$ are uniformly bounded in $C^{1+\a,\frac{1+\a}{2}}\left(\ov{Q}\times[\d,T]\right)$. Hence, we have
\[
\li_{k\to\infty}\left(\ov{u}^k,\ov{v}^k\right)=\left(\hat{u},\hat{v}\right),\quad \li_{k\to\infty}\left(\un{u}^k,\un{v}^k\right)=\left(\check{u},\check{v}\right)
\]
in $C^{1+\alpha,\frac{1+\alpha}{2}}_{loc}\left(\overline{\O}_0\times (0,+\infty)\right)$, and
$F_{1}\left(\ov{u}^{k},\un{v}^{k}\right),\, F_{2}\left(\ov{u}^{k},\ov{v}^{k}\right),\, F_{1}\left(\un{u}^{k},\ov{v}^{k}\right),\, F_{2}\left(\un{u}^{k},\un{v}^{k}\right)$ are uniformly bounded in $ C^{\a,\frac{\a}{2}}\left(\ov{Q}\times[\d,T]\right)$.
Together with the Schauder estimate \cite[Theorem 1.19]{wang+2021}, this boundedness gives that
\[
  \ov{u}^k,\, \ov{v}^k,\, \un{u}^k,\, \un{v}^k  \text{ are uniformly bounded in } C_{loc}^{2+\a,1+\frac{\a}{2}}\left(\O_0\times(0,+\infty)\right),
\]
since $\ov{u}^k,\,\un{u}^k,\, \ov{v}^k$ and $\un{v}^k$ satisfy \eqref{u-sup}-\eqref{v-sup}. Analogously, $\ov{w}^k$ and $\un{w}^k$ are uniformly bounded in $C_{loc}^{2+\a,1+\frac{\a}{2}}\left(\R\times(0,+\infty)\right)$.
Therefore,
\[
    \li_{k\to\infty}\left(\ov{u}^k,\ov{v}^k,\ov{w}^k\right)=\left(\hat{u},\hat{v},\hat{w}\right),\quad \li_{k\to\infty}\left(\un{u}^k,\un{v}^k,\un{w}^k\right)=\left(\check{u},\check{v},\check{w}\right)
\]
locally uniformly in $\left(C^{2+\alpha,1+\frac{\alpha}{2}}\left(\O_0\times(0,+\infty)\right)\right)^2\times C^{2+\alpha,1+\frac{\alpha}{2}}\left(\R\times(0,+\infty)\right)$. On the other hand, since for any $k=0,1,2,\cdots$, there holds
\[
	\un{u}^k\leq \check{u}\leq \hat{u}\leq \ov{u}^k\text{ in } \ov{\O}_0\times (0,+\infty), \qquad \un{u}^k(x,y,0)= \ov{u}^k(x,y,0)=u_0(x,y) \text{ on } \ov{\O}_0,
\]
one has
\[
	\displaystyle u_0(x,y)=\lim_{t \to 0_+}\un{u}^k(t,x,y)\leq \liminf_{t \to 0_+}\check{u}(t,x,y)\leq \limsup_{t \to 0_+}\hat{u}(t,x,y)\leq \lim_{t \to 0_+}\ov{u}^k(t,x,y)=u_0(x,y).
\]
We remark that these limits are uniform in $(x,y)$. The inequality indicates that $\check{u}$ and $\hat{u}$ are continuous at $t=0$ for any $(x,y)\in\ov{\O}_0$. Consequently, we obtain
\[
	\check{u},\hat{u}\in C^{2+\alpha,1+\frac{\alpha}{2}}(\O_0\times (0,+\infty))\cap C^{1+\alpha,\frac{1+\alpha}{2}}(\ov{\O}_0\times(0,+\infty))\cap C\left(\ov{\O}_0\times [0,+\infty)\right),
\]
and hence, $	\check{u}(\cdot,0)=\hat{u}(\cdot,0)=u_0(\cdot)$  on $\ov{\O}_0$.
Likewise, one has
\[
\renewcommand{\arraystretch}{1.25}
\begin{array}{l}
	\check{v},\hat{v}\in C^{2+\alpha,1+\frac{\alpha}{2}}(\O_0\times (0,+\infty))\cap C^{1+\alpha,\frac{1+\alpha}{2}}(\ov{\O}_0\times(0,+\infty))\cap C\left(\ov{\O}_0\times [0,+\infty)\right),\\
	\check{w},\hat{w}\in C^{2+\alpha,1+\frac{\alpha}{2}}(\R\times (0,+\infty))\cap C\left(\R\times [0,+\infty)\right),\\
	\check{v}(\cdot,0)=\hat{v}(\cdot,0)=v_0(\cdot) ~{\rm on}~ \ov{\O}_0 \quad {\rm and}\quad\check{w}(\cdot)=\hat{w}(\cdot)=w_0(\cdot)~{\rm in}~\R.
\end{array}
\]

\vspace{1.25ex}
\noindent\textit{Step 4.
Existence of solutions.}
\vspace{1.25ex}

Clearly, $\check{u},\, \check{v},\, \check{w}$ and $\hat{u},\, \hat{v},\, \hat{w}$ slove the problem
\[\left\{
\begin{array}{lll}
	\partial_{t}\hat{u}-d_1\Delta \hat{u} =f_1\left(\hat{u},\check{v}\right),&\partial_{t}\check{u}-d_1\Delta \check{u} =f_1\left(\check{u},\hat{v}\right),&\quad(x,y)\in \O_0,t>0,\\
	\partial_{t}\check{v}-d_2\Delta \check{v} =f_2(\check{u},\check{v}),&\partial_{t}\check{v}-d_2\Delta \check{v} =f_2\left(\check{u},\check{v}\right),&\quad(x,y)\in \O_0,t>0,\\
	\partial_{t}\hat{w}-D\D_{x}\hat{w} =\nu\hat{v}(x,0,t)-\mu \hat{w}, & \partial_{t}\check{w}-D\D_{x}\check{w} =\nu\check{v}(x,0,t)-\mu \check{w},&\quad x\in \mathbb{R},t>0.
\end{array}
\right.\]
with $\hat{u}(x,y,0)=\check{u}(x,y,0)=u_{0}(x,y)$, $\hat{v}(x,y,0)=\check{v}(x,y,0)=v_{0}(x,y)$, $\hat{w}(x,0)=\check{w}(x,0)=w_{0}(x)$ for $(x,y)\in \ov{\O}_0$.
Moreover, \eqref{monotone sequences} yields that $(0,0,0)\leq\left(\check{u},\check{v},\check{w}\right)\leq \left(\hat{u},\hat{v},\hat{w}\right)\leq\left(\ov{u}^0,\ov{v}^0,\ov{w}^0\right)$.
We establish the existence of a solution of \eqref{drift} by showing $\left(\hat{u},\hat{v},\hat{w}\right)=\left(\check{u},\check{v},\check{w}\right)=:(u,v,w)$. Indeed, $(U,V,W):=\left(\hat{u}-\check{u},\hat{v}-\check{v},\hat{w}-\check{w}\right)\geq (0,0,0)$ satisfies
\[
	\left\{\begin{array}{lll}
		\partial_t U-d_1\D U \leq l (U+V), & \p_t V-d_2\Delta V \leq l(U+V),
		& \quad (x,y)\in \O_0,t>0,\\
		\partial_{t}W-D\p_{xx}W +\mu W=\n V(x,0,t), \quad & &\quad x\in \mathbb{R},t>0,\\
		\n V(x,0,t)-d_2\partial_{y}V(x,0,t)=\mu W(x,t), &\p_y U(x,0,t)= 0, &\quad  x\in \mathbb{R},t>0
	\end{array}\right.
\]
with
		$U(x,y,0)=0,V(x,y,0)=0, W(x,0)=0$, for $(x,y)\in \ov{\O}_0$.
This system satisfies the comparison principle (see Lemma \ref{comparison-cooperate}), which implies $(U,V,W)\leq (0,0,0)$, and thus $(U,V,W)=(0,0,0)$.


\vspace{1.25ex}
\noindent\textit{Step 5.
Uniqueness of the solution.}
\vspace{1.25ex}

Now we deal with the uniqueness issue. Let $\ov{u}$ be the solution of
\begin{equation}\label{u-sup-1}
	\left\{\begin{aligned}
		&\partial_{t}\ov{u}-d_1\Delta \ov{u}=\ov{u}(1-\ov{u}),\qquad &&(x,y)\in \O_0,t>0,&\\
		&\p_y\ov{u}(x,0,t)=0, && x\in \R,t>0&
	\end{aligned}\right.
\end{equation}
with $\ov{u}(x,y,0)=u_0(x,y)$ for any $(x,y)\in\ov{\O}_0$. For any nonnegative bounded solution $(u,v,w)$ of \eqref{drift}, there holds
$u(x,y,t)\leq \ov{u}(x,y,t)\leq 1$ for all $(x,y)\in \ov{\O}_0$ and $ t\geq 0$. 
It gives that $(v,w)$ is a subsolution of
\begin{equation}\label{v-sup-1}
	\left\{
		\begin{aligned}
			&\partial_{t}\ov{v}-d_2\Delta \ov{v}=\ov{v}(b-1-\ov{v}),&&\qquad(x,y)\in \O_0,t>0,&\\
			&\partial_{t}\ov{w}-D\ov{w}_{xx}=-\mu \ov{w}+\n \ov{v}(x,0,t),&&\qquad x\in \mathbb{R},t>0,& \\
			&-d_2\partial_{y}\ov{v}(x,0,t)=\mu \ov{w}(x,t)- \n \ov{v}(x,0,t),&&\qquad x\in \mathbb{R},t>0&
			\end{aligned}
	\right.
\end{equation}
with $\ov{v}(x,y,0)=v_0(x,y),\, \ov{w}(x,0)=w_0(x)$ for all $(x,y)\in \ov{\O}_0$. Then because $(b-1)\left(1,\frac{\nu}{\mu}\right)$ is a supersolution of \eqref{v-sup-1}, we have
$
(v,w)\leq \left(\ov{v},\ov{w}\right)\leq (b-1)\left(1,\frac{\nu}{\mu}\right)$.
Therefore, we get $(u,v,w)\leq \left(\ov{u}^0,\ov{v}^0,\ov{w}^0\right)$. By similar arguments as for the monotonicity of $\left\{(\overline{u}^{k},\overline{v}^{k},\overline{w}^{k})\right\}_{k}$ and $\left\{(\underline{u}^{k},\underline{v}^{k},\underline{w}^{k})\right\}_{k}$, we can show
\[
\un{u}^k\leq u\leq \ov{u}^k,~\un{v}^k\leq v\leq \ov{v}^k \text{ in }\ov{\O}_0\times [0,+\infty), \qquad \un{w}^k\leq w\leq \ov{w}^k\text{ in }\mathbb{R}\times [0,+\infty),
\]
for all $k=0,1,2,\cdots$.
This derives $\check{u}\leq u\leq \hat{u},~\check{v}\leq v\leq \hat{v},~\check{w}\leq w\leq \hat{w}$, and thereby
$	(u,v,w)\equiv\left(\check{u},\check{v},\check{w}\right)\equiv\left(\hat{u},\hat{v},\hat{w}\right)$.
The proof is completed.
\end{proof}

\subsection{Preliminaries}
 Let us first review the results obtained by  Berestycki et al.~\cite{berestycki+2013,berestycki+2013a,berestycki+2016}. For simplicity and generality, we take the following system as an example, drawing mainly on the work in \cite{berestycki+2016}:
\begin{equation}\label{predator sys}
	\left\{
		\begin{aligned}
			&\partial_{t}v-d\Delta v=v(\d-kv),&&\qquad(x,y)\in \O_0,t>0,&\\
			&\partial_{t}w-Dw_{xx}=-\mu w+\n v(x,0,t),&&\qquad x\in \mathbb{R},t>0,& \\
			&-d\partial_{y}v(x,0,t)=\mu w(x,t)- \n v(x,0,t),&&\qquad x\in \mathbb{R},t>0,&
			\end{aligned}
	\right.
\end{equation}
where $d,\, D,\, \d,\, k$ are positive constants, and $\O_0=\R\times (0,+\infty)$. The initial condition $(v_0,w_0)\not\equiv (0,0)$, and $v_0,\,w_0$ are nonnegative, compactly supported.
In \cite{berestycki+2016}, Berestycki et al. introduced the concept of  the {\it asymptotic expansion shape} $\mathcal{W}_{\d}$ to characterize the above spreading results. The asymptotic expansion shape is a closed set such that for any solution $(v,w)$ of \eqref{predator sys} emerging from a nontrivial and compactly supported initial datum, there hold 
\[
	\begin{array}{c}
			\displaystyle\li_{t\to+\infty}\sup_{\substack{(x,y)\in\ov{\O}_0\\{\rm dist}\left(\frac{1}{t}(x,y),\mathcal{W}_{\d}\right)>\e}}v(x,y,t)=0,\\
			\displaystyle\li_{t\to+\infty}\sup_{\substack{(x,y)\in\ov{\O}_0\\{\rm dist}\left(\frac{1}{t}(x,y),\ov{\O}_0\setminus\mathcal{W}_{\d}\right)>\e}}\left|v(x,y,t)-\frac{\delta}{k}\right|=0,
	\end{array}
	\qquad \forall \e>0.
\]
\begin{lemma}[{\cite[Theorem 2.1]{berestycki+2016}}]\label{berestycki+2016 thm2.1}
	\begin{enumerate}[label=\rm{(\roman*)}]
		\item  {\rm (Spreading in $\O_0$)} Problem \eqref{predator sys} admits an asymptotic expansion shape $\mathcal{W}_{\d}$.	
		\item  {\rm (Shape of $\mathcal{W}_{\d}$)} The set $\mathcal{W}_{\d}$ is convex and is of the form
			$$
        			\mathcal{W}_{\d}=\left\{r(\sin\vartheta,\cos\vartheta):-\frac{\pi}{2}\leq\vartheta\leq\frac{\pi}{2},~0\leq r\leq \vr_{\d}^{*}(\vartheta)\right\}.
        		$$
        		Here, $\vr_{\d}^{*}\in C^1\left(\left[-\frac{\pi}{2},\frac{\pi}{2}\right]\right)$ is even, and there is $\vartheta_0\in\left(0,\frac{\pi}{2}\right]$ such that
        		$$
        			\vr_{\d}^{*}=2\sqrt{d\d}\text{ in }[0,\vartheta_0],\quad \left(\vr_{\d}^{*}\right)^{\prime}>0\text{ in }(\vartheta_0,\pi/2].
        		$$
        		Moreover, 
        		$\mathcal{W}_{\d}$ contains the set
        		$$
        			\underline {\mathcal{W} }_{\d}: = \text{conv}\left (  \overline {B}_{2\sqrt{d\d}} \cup \left[ - \vr_{\d}^{*} \left( \frac{\pi}{2}\right) ,  \vr_{\d}^{*} \left(\frac{\pi}{2}\right) \right] \times \{ 0\} \right )
        		$$
        and the inclusion is strict if $D>2d$. Here and in what follows, $B_r:=B_r(0,0)\cap \ov{\O}_0$ for any $r>0$.
	\end{enumerate}
\end{lemma}
\begin{remark}\label{remark26}
	The spread of \eqref{predator sys} along the road was first investigated by Berestycki et al. {\rm \cite{berestycki+2013}}. It corresponds to the case  $\vt=\pm \frac{\pi}{2}$ in Lemma \ref{berestycki+2016 thm2.1}. More precisely, there hold 
        \[
			\li_{t\to+\infty}\sup_{\substack{|x|\geq ct\\y\geq 0}}(v(x,y,t),w(x,t))=(0,0),\qquad \forall c>\vr_{\d}^*\left(\frac{\pi}{2}\right),
        \]
        and 
        \[
        \li_{t\to+\infty}\inf_{\substack{|x|\leq ct\\0\leq y<l}}(v(x,y,t),w(x,t))=\frac{\delta}{k}\left(1,\dfrac{\nu}{\mu}\right),\qquad \forall c< \vr_{\d}^*\left(\frac{\pi}{2}\right),\ l>0.
        \]
\end{remark}

\begin{remark}\label{monotone}
The proof of {\rm \cite[Theorem 2.1]{berestycki+2016}} demonstrated that $\vr_{\d}^*(\vt)$ is strictly increasing in $\d$. 
Moreover, $\vr_{\d}^*(\cdot)$ is increasing on $\left[0,\frac{\pi}{2}\right]$ and is an even function. Therefore, we have
\[
	\vr_{\d}^*(\vt)\leq\vr_{\d}^*\left(\frac{\pi}{2}\right)\to 0 \text{ as } \d\to 0,\qquad \forall\vt\in\left[-\frac{\pi}{2},\frac{\pi}{2}\right].
\]
Specifically, the limit $\vr_{\d}^*(\vt)\to 0 \text{ as } \d\to 0$ is uniform in $\vt\in\left[-\frac{\pi}{2},\frac{\pi}{2}\right].$
\end{remark}

Actually, for any $\vt\in\left[-\frac{\pi}{2},\frac{\pi}{2}\right]$, $\vr_{\d}^{*}(\vt)$ is the critical value of $c$ such that the following linearized system of \eqref{predator sys} at $v=0$, namely
\begin{equation}\label{linearized}
	\left\{
		\begin{aligned}
			&\partial_{t}v-d\Delta v=\d v,&&\qquad(x,y)\in \O_0,t>0,&\\
			&\partial_{t}w-Dw_{xx}=-\mu w+\n v(x,0,t),&&\qquad x\in \mathbb{R},t>0,& \\
			&-d\partial_{y}v(x,0,t)=\mu w(x,t)- \n v(x,0,t),&&\qquad x\in \mathbb{R},t>0,&
			\end{aligned}
	\right.
\end{equation}
admits an exponential solution of the form
\begin{equation}\label{exponential}
	\left(\gamma e^{-(\alpha,\beta)\cdot((x,y)-ct\xi)},~e^{-(\alpha,\beta)\cdot((x,0)-ct\xi)}\right)
\end{equation}
 with  $\alpha,\beta\in \R,\gamma>0$ and $\xi=(\sin\vt,\cos\vt)$. Therefore, for any $\vt\in\left[-\frac{\pi}{2},\frac{\pi}{2}\right]$, $\vr^{*}_{\delta}(\vt)$ is independent of $k$. 
For any $\vt\in\left[-\frac{\pi}{2},\frac{\pi}{2}\right]$, the upper bound of propagation velocity in the direction $\xi=(\sin \vt,\cos \vt)$ can be proven by establishing the planar wave \eqref{exponential} for the linearlized system \eqref{linearized}.
However, due to the interaction between $v$ and $w$, the commonly used generalized sub- and supersolutions---the supremum of subsolutions and the infimum of supersolutions---are not applicable to system \eqref{predator sys}. Therefore, a {\it generalized subsolution} is introduced as follows:
\begin{definition}[{\cite[Definition 4.2]{berestycki+2016}}]
	A pair $(\un{v},\un{w})$ is a generalized subsolution of \eqref{predator sys}
	if $\un{v},\, \un{w}$ are continuous and satisfy the following properties:
	\begin{enumerate}[label=\rm{(\roman*)}]
		\item for any $x\in\mathbb{R},t>0$, there is a function $w$ such that $w\leq\underline{w}$ in a neighbourhood of $(x,t)$ and, at $(x,t)$ {\rm (}in the classical sense{\rm )},
		\[
			w=\underline{w},\qquad\partial_tw-D\partial_{xx}w+\mu w\leq\nu\underline{v}(x,0,t);		
		\]
          \item for any $(x,y)\in\ov{\O}_0,t>0$, there is a function $v$ such that $v\leq\underline v$ in a neighbourhood of $(x,y,t)$ and, at $(x,y,t)$,
          \[
				v=\underline{v},\qquad
				\left\{\begin{aligned}
					&\partial_tv-d\Delta v\leq v(\delta- kv), && \mathrm{~if~}y>0,\\
					&-d\partial_y v(x,0,t)+\nu v(x,0,t)\leq\mu\underline{w}(x,t), && \mathrm{~if~}y=0.
				\end{aligned}
				\right.
          \]
	\end{enumerate}
\end{definition}

\begin{lemma}\label{generalized_subsolution} For any $\vt\in\left[-\frac{\pi}{2},\frac{\pi}{2}\right]$, let $\vr^{*}_{\delta}(\vt)$ be the spreading speed of \eqref{predator sys} along the direction $(\si\vt,\c\vt)$ given by Lemma \ref{berestycki+2016 thm2.1}. Then we have
\begin{itemize}
	\item[{\rm (i)}]{\rm\cite[Lemmas 6.1 and 6.2]{berestycki+2013}} There is $\varepsilon\left(\frac{\pi}{2}\right)>0$ such that for any $c\in\left(\vr^{*}_{\delta}(\frac{\pi}{2})-\varepsilon(\frac{\pi}{2}), \vr^{*}_{\delta}(\frac{\pi}{2})\right)$, system
\begin{equation}\label{predator sys drift}
	\left\{
		\begin{aligned}
			&\partial_{t}v-d\Delta v+c\p_x v=v(\d-kv),&&\qquad(x,y)\in \O_0,t>0,&\\
			&\partial_{t}w-Dw_{xx}+c\p_x w=-\mu w+\n v(x,0,t),&&\qquad x\in \mathbb{R},t>0,& \\
			&-d\partial_{y}v(x,0,t)=\mu w(x,t)- \n v(x,0,t),&&\qquad x\in \mathbb{R},t>0&
			\end{aligned}
	\right.
\end{equation}
	admits a nonnegative, compactly supported, generalized stationary subsolution $\left(\underline{\phi},\underline{\psi}\right)\not\equiv(0,0)$.
	\item[{\rm(ii)}]{\rm \cite[Lemma 4.1]{berestycki+2016}} For any $\vt\in\left(-\frac{\pi}{2},\frac{\pi}{2}\right)$ and $\e>0$, there exists $c\in \left(\vr^{*}_{\delta}(\vt)-\varepsilon,\vr^{*}_{\delta}(\vt)\right)$ and a pair $(\underline{v},\, \underline{w})$ of nonnegative functions  with the following properties: $\un{v}(x,y,0),\, \un{w}(x,0)$ are compactly supported,
	\[
		\exists\, (\hat{x},\hat{y})\in\ov{\O}_0,~ \forall\, t\geq 0,~ \un{v}(\hat{x}+ct\si\vt,\hat{y}+ct\c\vt,t)=\un{v}(\hat{x},\hat{y},0)>0,
	\]
	 and for any $\kappa\in(0,1]$, $\kappa(\underline{v},\, \underline{w})$ is a generalized subsolution of \eqref{predator sys}.
\end{itemize}
\end{lemma}

By symmetry, we assume that $\vt\in\left[0,\frac{\pi}{2}\right]$.
It follows from  the proof of Lemma \ref{generalized_subsolution} (see \cite{berestycki+2013, berestycki+2016}) that   
the following proposition  holds by defining
 	\[
	\left(\un{\Phi}_{\delta,\vt,c}(x,y,t),\un{\Psi}_{\delta,\vt,c}(x,t)\right):=\left\{\begin{aligned}
			& \left(\un{v}(x,y,t),\un{w}(x,t)\right), && \text {if}~ 0\leq\vt<\frac{\pi}{2},& \\
			& \left(\un{\phi}(-x+ct,y),\un{\psi}(-x+ct)\right), && \text {if}~ \vt=\frac{\pi}{2},&
	\end{aligned}\right.
	\]
	where $\left(\un{v},\un{w}\right)$ and $\left(\un{\phi},\un{\psi}\right)$ are defined in Lemma \ref{generalized_subsolution}.

\begin{proposition}\label{uniform_epsilon}
	For any $\vt\in\left[-\frac{\pi}{2},\frac{\pi}{2}\right]$, $\delta>0$ and $\e>0$, there is $c\in\left(\vr^{*}_{\delta}(\vt)-\e, \vr^{*}_{\delta}(\vt)\right)$ such that system \eqref{predator sys} admits a nonnegative generalized subsolution $\left(\un{\Phi}_{\delta,\vt,c},\un{\Psi}_{\delta,\vt,c}\right)$ satisfying that:
	\begin{enumerate}[label={\rm (\alph*)}]
		\item \label{a} $\un{\Phi}_{\delta,\vt,c}(\cdot,\cdot,0)$ and $\un{\Psi}_{\delta,\vt,c}(\cdot,0)$ are compactly supported.
		\item \label{b} There exists $l_{\delta,\vt,c}\geq 0$ such that
		\[
			\begin{aligned}
				& {\rm supp}~ \un{\Phi}_{\delta,\vt,c}(\cdot,\cdot,t)={\rm supp}~ \un{\Phi}_{\delta,\vt,c}(\cdot,\cdot,0)+c(\si\vt,\c\vt)t,\\
				& {\rm supp}~ \un{\Psi}_{\delta,\vt}(\cdot,t)={\rm supp}~ \un{\Psi}_{\delta,\vt,c}(\cdot,0)+cl_{\delta,\vt,c}t,
			\end{aligned}
			\qquad \forall\, t\geq 0.
		\]
		\item \label{c} There is $(\hat{x},\hat{y})\in\ov{\O}_0$ such that
		$$
			\un{\Phi}_{\delta,\vt,c}\left(\hat{x}+ct\sin\vt,\hat{y}+ct\cos\vt,t\right)=\un{\Phi}_{\delta,\vt,c}\left({\hat{x},\hat{y},0}\right)>0,\qquad \forall t\geq 0.
		$$
		\item In particular,
		\[
			\left(\un{\Phi}_{\delta,\frac{\pi}{2},c}(-x+ct,y,t),\un{\Psi}_{\delta,\frac{\pi}{2},c}(-x+ct,t)\right) ~and ~\left(\un{\Phi}_{\delta,-\frac{\pi}{2},c}(x-ct,y,t),\un{\Psi}_{\delta,-\frac{\pi}{2},c}(x-ct,t)\right)
		\]
		are  nontrivial generalized stationary subsolution of \eqref{predator sys drift}.
	\end{enumerate}
\end{proposition}



\subsection{Main Results} 

For any $\vt\in\left[-\frac{\pi}{2},\frac{\pi}{2}\right]$, let $\vr^{*}_{b-1}(\vt)$ be the spreading speed of \eqref{predator sys} with $d=d_2$ and $\d=b-1$ along the direction $(\sin \vt, \cos \vt)$, determined in Lemma \ref{berestycki+2016 thm2.1}.
For the remainder of this paper, we denote $c^{*}:=2\sqrt{ d_1 }$, $\mathcal{V}:=\mathcal{W}_{b-1}\cap \ov{B}_{c^{*}}$, $(u^{*},v^{*},w^{*}):=\left(\frac{a+1}{ab+1},\frac{b-1}{ab+1},\frac{\n(b-1)}{\mu(ab+1)}\right)$, $\O_l:=\R\times(l,+\infty)$ for any $l\in (-\infty,+\infty)$, and $\ov{\O}_{l}={\O}_{l}=\R\times\R$ for $l=-\infty$. We now state our main results of this paper.
\begin{theorem}\label{main}
Let $(u,v,w)\equiv (u(x,y,t),v(x,y,t),w(x,t))$ be the solution of \eqref{prey-predator-road} emerging from $(u_{0},v_{0},w_{0})$. Assume that the initial data $(u_0,v_0,w_0)$ satisfies {\rm \ref{hy-initial}}. Then we have the following results:
\begin{enumerate}[label=\rm{(\roman*)}]
	\item\label{u die} For any $c>c^{*}$, there holds
		\begin{equation}\label{u die eq}
			\li_{ t \to \infty }\sup_{\substack{(x,y)\in\ov{\O}_0\\|(x,y)|\ge ct}} u(x,y,t)=0.
		\end{equation}
	\item\label{survive} For any $\e>0$, there holds
		\begin{equation}\label{plane converge}
			\li_{ t \to \infty }\sup_{\substack{(x,y)\in\ov{\O}_0\\ {\rm dist}\left(\frac{1}{t}(x,y),\ov{\O}_0\setminus \mathcal{V}\right)>\e}}|u(x,y,t)-u^{*}|+|v(x,y,t)-v^{*}|=0
		\end{equation}
		and
		\begin{equation}\label{v die}
		\li_{ t \to \infty }\sup_{\substack{(x,y)\in\ov{\O}_0\\{\rm dist}\left(\frac{1}{t}(x,y), \mathcal{V}\right)>\e}}v(x,y,t)=0.
		\end{equation}
		If $B_{c^*}\setminus \mathcal{V}\not= \emptyset$, we have
		\begin{equation}\label{u best}
			\li_{ t \to \infty }\sup_{\substack{(x,y)\in\ov{\O}_0\\{\rm dist}\left(\frac{1}{t}(x,y),\mathcal{V}\cup\left(\ov{\O}_0\setminus B_{c^*}\right)\right)>\e}}|u(x,y,t)-1|=0.
		\end{equation}
		
		\item \label{w}
		For any $0\leq c< \min\left\{c^{*},\vr^{*}_{b-1}\left(\frac{\pi}{2}\right)\right\}$, there holds
		\begin{equation}\label{w converges}
			\li_{ t \to \infty }\sup_{|x|\leq ct}|w(x,t)-w^{*}|=0.
		\end{equation}
		For any $c> \min\left\{c^{*},\vr^{*}_{b-1}\left(\frac{\pi}{2}\right)\right\}$, there holds
		\begin{equation}\label{w die}
		\li_{ t \to \infty }\sup_{|x|\geq ct}w(x,t)=0.
		\end{equation}
\end{enumerate}
\end{theorem}
\begin{figure}[htb]
    \centering
    \begin{subfigure}[b]{0.49\textwidth}
        \centering
        \includegraphics[width=\textwidth]{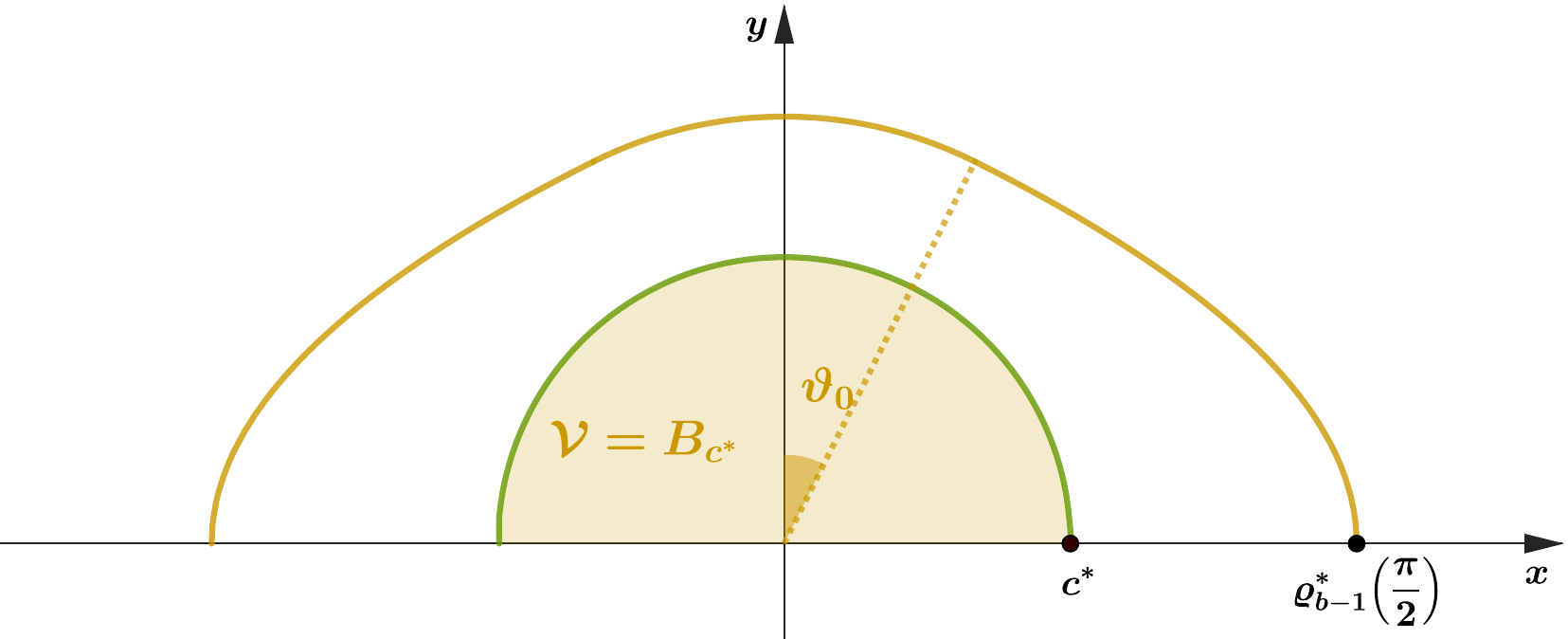}
        \caption*{\footnotesize{$c^*<\vr^*_{b-1}(0)$}}
    \end{subfigure}
    \hfill  
    \begin{subfigure}[b]{0.49\textwidth}
        \centering
        \includegraphics[width=\textwidth]{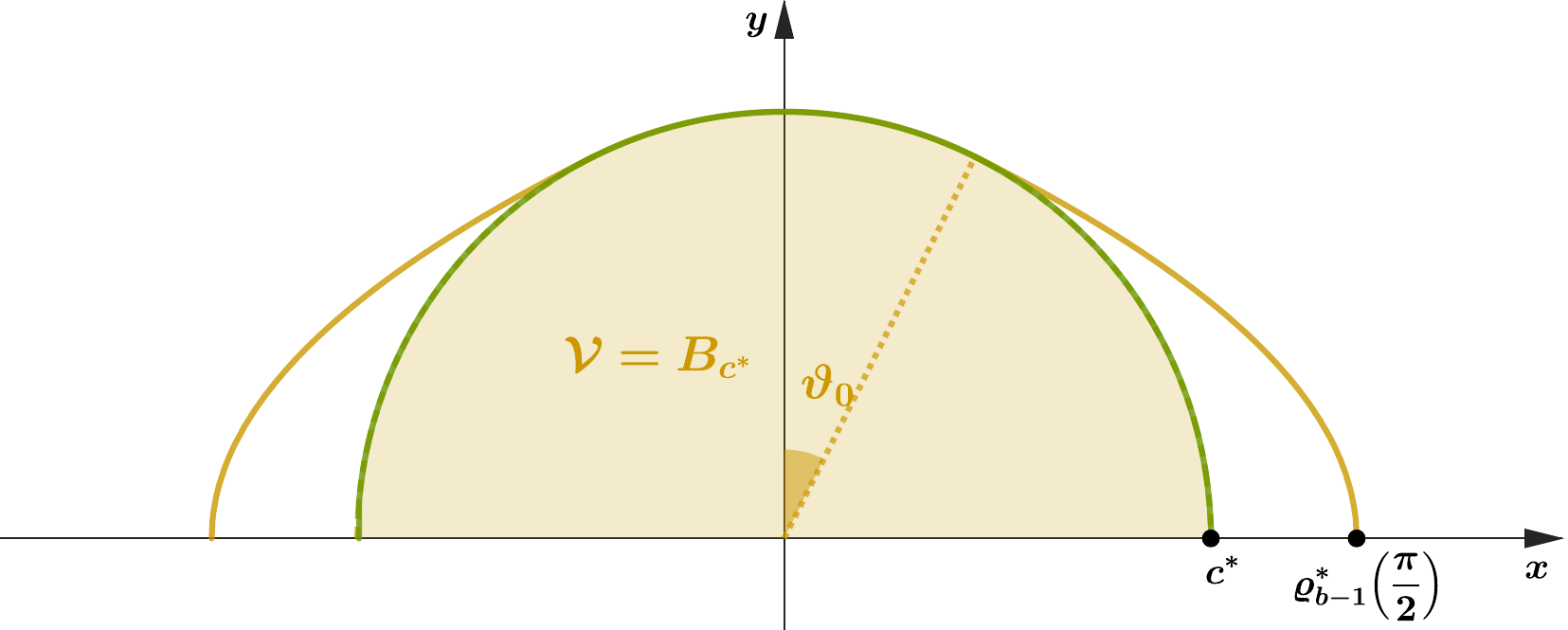}
        \caption*{\footnotesize{$c^*=\vr^*_{b-1}(0)$}}
    \end{subfigure}
    
    \vspace{\baselineskip}  
    
    \begin{subfigure}[b]{0.49\textwidth}
        \centering
        \includegraphics[width=\textwidth]{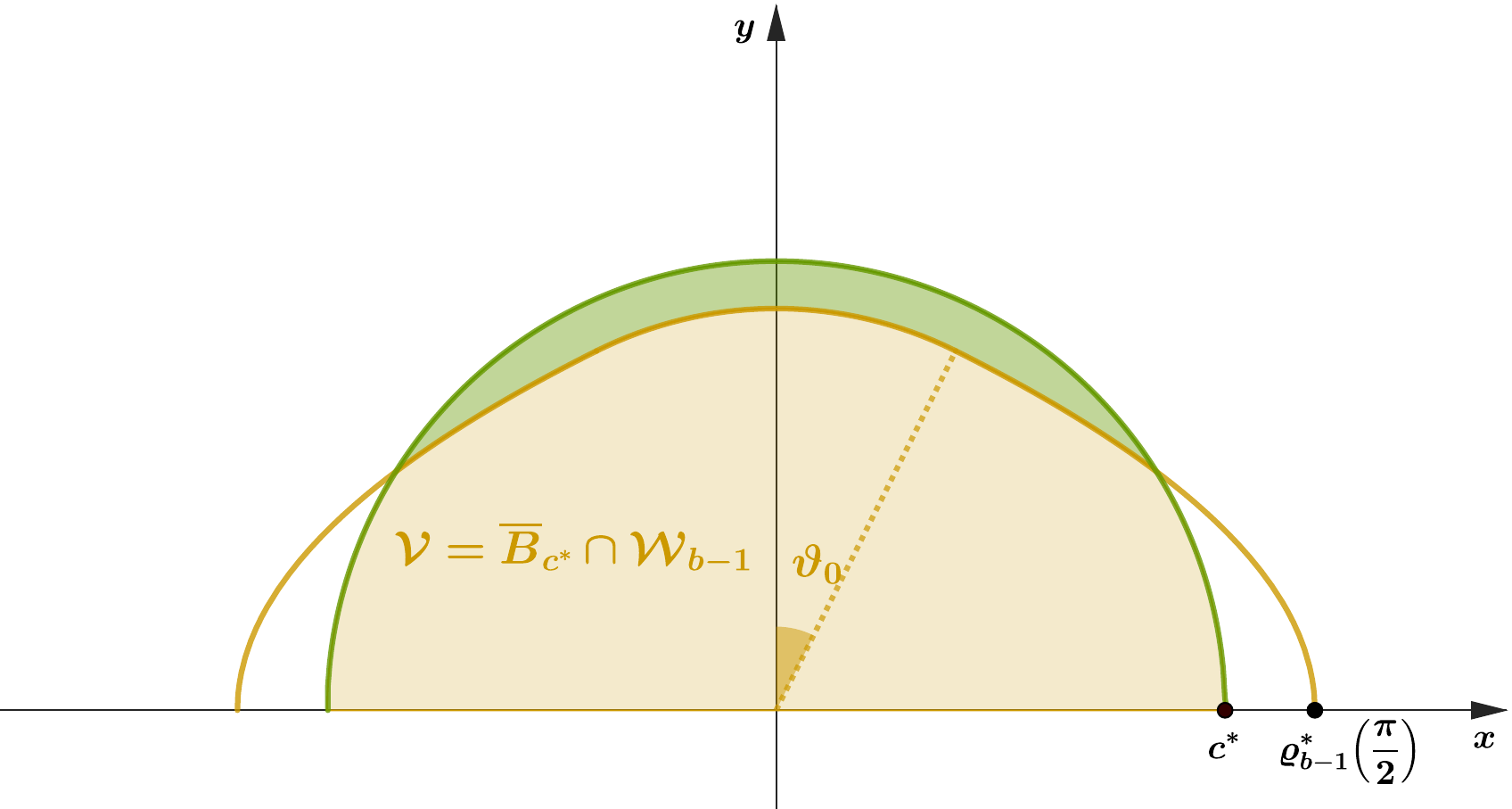}
        \caption*{\footnotesize{$\vr^*_{b-1}(0)<c^*<\vr^*_{b-1}\left(\frac{\pi}{2}\right)$}}
    \end{subfigure}
    \hfill
    \begin{subfigure}[b]{0.49\textwidth}
        \centering
        \includegraphics[width=\textwidth]{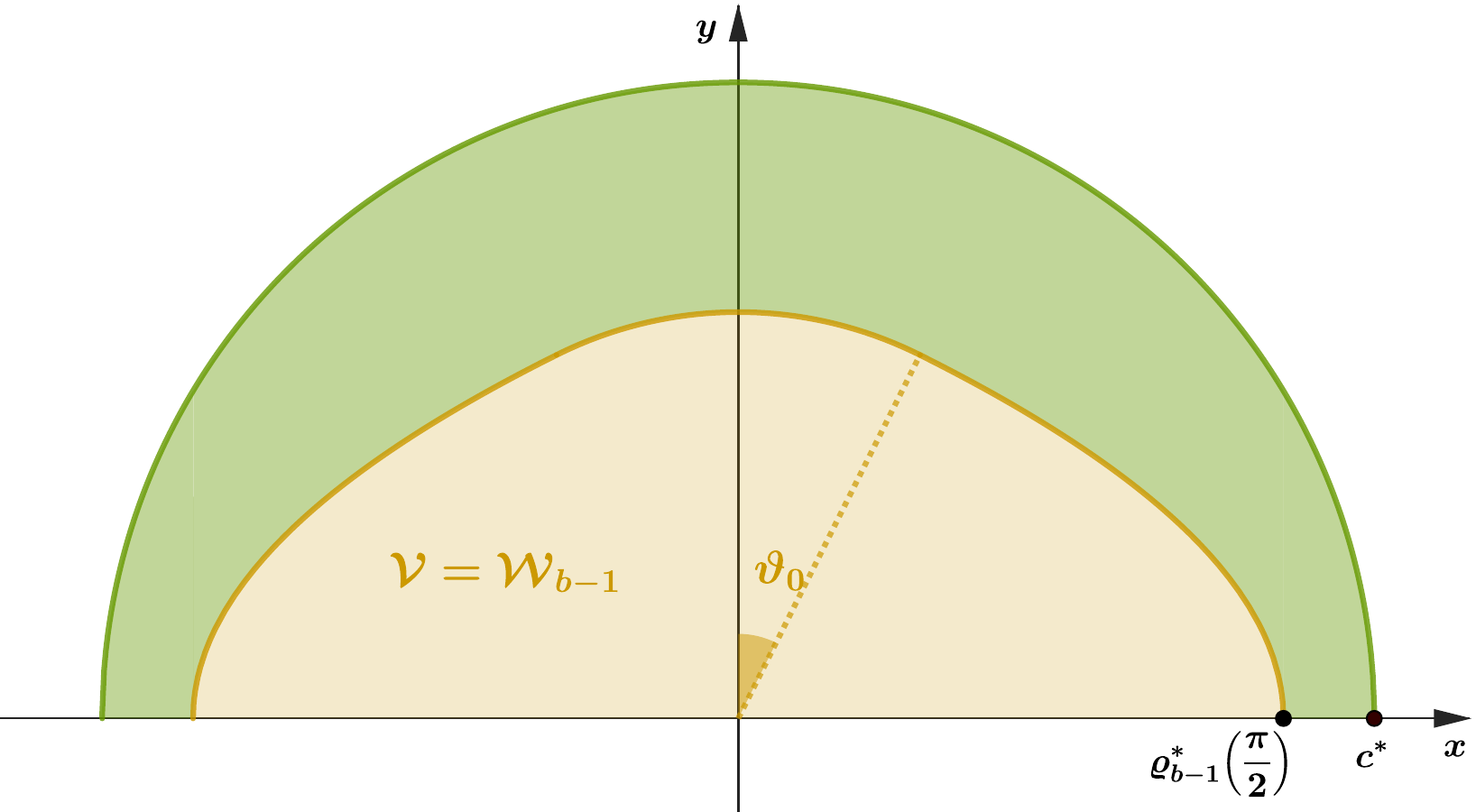}
        \caption*{\footnotesize{$c^*>\vr^*_{b-1}\left(\frac{\pi}{2}\right)$}}
    \end{subfigure}
    \caption{When the initial data satisfies \ref{hy-initial}, the {\it asymptotic expansion shape} of the prey and predators in \eqref{prey-predator-road} are $B_{c^*}$ and $\mathcal{V}$, respectively. If $B_{c^*}\setminus \mathcal{V}\not=\emptyset$ and $\frac{1}{t}(x,y)\in B_{c^*}\setminus \mathcal{V}$, $u(x,y,t)$ ultimately reaches $1$.}
     \label{fig2}
\end{figure}

Under Assumption \ref{hy-initial}, predators and the prey invade the field simultaneously (see Figure \ref{fig2}). However, if the prey initially admits a positive lower bound, the invading speed of predators is not affected by the prey.
\begin{enumerate}[label=\textbf{(H\arabic*$^\prime$)}] 
	\item\label{hy-initial-positive} $u_{0}(x,y),v_{0}(x,y)$ and $w_{0}(x)$ are uniformly continuous functions. $(v_{0}(x,y),w_{0}(x))$ is compactly supported and nontrivial. Moreover, for some fixed $\varepsilon_0>0$, we have $(\varepsilon_0,0,0)\le \left(u_{0}, v_0, w_0\right)\le \left(1,b-1, \frac{\nu}{\mu}(b-1)\right)$.
\end{enumerate}
\begin{theorem}\label{positive initial}
	Let $(u,v,w)\equiv (u(x,y,t),v(x,y,t),w(x,t))$ be the solution of \eqref{prey-predator-road} emerging from $(u_{0},v_{0},w_{0})$. Assume that the initial data $(u_0,v_0,w_0)$ satisfies {\rm \ref{hy-initial-positive}}. Then we have the following results:
	\begin{enumerate}[label=\rm{(\roman*)}]
		\item For any $\e>0$, there holds
		\begin{equation}\label{positive-coexistence}
			\li_{ t \to \infty }\sup_{\substack{(x,y)\in\ov{\O}_0\\ {\rm dist}\left(\frac{1}{t}(x,y),\ov{\O}_0\setminus \mathcal{W}_{b-1}\right)>\e}}|u(x,y,t)-u^{*}|+|v(x,y,t)-v^{*}|=0
		\end{equation}
		and
		\begin{equation}\label{positive-vdie}
			\li_{ t \to \infty }\sup_{\substack{(x,y)\in\ov{\O}_0\\{\rm dist}\left(\frac{1}{t}(x,y), \mathcal{W}_{b-1}\right)>\e}}|u(x,y,t)-1|+v(x,y,t)=0.
		\end{equation}
		\item For any $0\leq c< \vr^{*}_{b-1}\left(\frac{\pi}{2}\right)$, there holds
		\begin{equation}\label{positive-w converge}
			\li_{ t \to \infty }\sup_{|x|\leq ct}|w(x,t)-w^{*}|=0.
		\end{equation}
		For any $c>\vr^{*}_{b-1}\left(\frac{\pi}{2}\right)$, there holds
		\begin{equation}\label{positive-w die}
		\li_{ t \to \infty }\sup_{|x|\geq ct}w(x,t)=0.
		\end{equation}
	\end{enumerate}
\end{theorem}

\section{\texorpdfstring{Spreading Speeds for \eqref{prey-predator-road} with \ref{hy-initial}}{Spreading Speeds for (1.4) with (H1)}}\label{upper}
In this section, we first provide two key lemmas: one concerning the pointwise estimates between predators and the prey; and the other concerning the asymptotic expansion shape of a truncated system. Then, for system \eqref{prey-predator-road} satisfying \ref{hy-initial}, we estimate the upper bounds of spreading speeds. Finally, we provide lower estimates for these speeds.

\subsection{Important Lemmas}
\subsubsection{Pointwise comparisons between predators and the prey}
The following lemma states that the predator cannot survive without the prey.
\begin{lemma}\label{predator die without prey}
	Let $(u,v,w)$ be the classical solution of \eqref{prey-predator-road} with $(u,v,w)=(u_0,v_0,w_0)$ at $t=0$. Then, for any $\sigma>0$, there exist $M_{\sigma}>0$ and $T_{\sigma}>0$ such that for any $(u_0,v_0,w_0)$ satisfying {\rm \ref{h2}}, we have 
	\begin{equation}\label{v die without u}
		v(x,y,t) \le \sigma+ M_{\sigma}u(x,y,t),\quad \forall (x,y)\in \ov{\O}_0,\ t\ge T_{\sigma}
	\end{equation}
	and
	\begin{equation}\label{w die without u}
		w(x,t) \le \sigma+ M_{\sigma}u(x,0,t), \quad \forall (x,y)\in \ov{\O}_0,\ t\ge T_{\sigma}.
	\end{equation}
	
\end{lemma}
\begin{proof}
	We first prove \eqref{v die without u} by using a contradiction argument. Suppose that there are $\left(u_{0,n},v_{0,n},w_{0,n}\right)$ satisfying \ref{h2}, $(x_{n},y_{n},t_{n})\in \overline{\Omega}_0\times(0,+\infty)$ and $\sigma_{0}>0$ such that
$$
t_{n}\to \infty \text{ as } n \to \infty\qquad \text{and}\qquad v^{n}(x_{n},y_{n},t_{n})> \sigma _{0} + n u^{n}(x_{n},y_{n},t_{n}),~ \forall n\in \mathbb{N},
$$
where $\left(u^{n},v^{n},w^{n}\right)$ is the solution of \eqref{prey-predator-road} with $\left(u^{n},v^{n},w^{n}\right)=\left(u_{0,n},v_{0,n},w_{0,n}\right)$ at $t=0$. Denote $u_{n}(x,y,t)=u^{n}(x+x_{n},y+y_{n},t+t_{n}),~v_{n}(x,y,t)=v^{n}(x+x_{n},y+y_{n},t+t_{n}),~w_{n}(x,t)=w^{n}(x+x_{n},t+t_{n})$ for any $(x,y)\in \ov{\Omega}_{-y_n}$ and $t\geq -t_n$. In the following, we proceed by dividing the problem into two alternative cases.

\vspace{1.25ex}
\noindent\textit{Case 1.} 
$\{y_{n}\}$ admits a bounded convergent subsequence $\left\{ y_{n_k} \right\}$ with $\li_{k\to\infty}y_{n_k}=z_{0}\geq 0$.
\vspace{1.25ex}
	
Note that $0\leq u^{n}(x,y,t)\leq 1,~0\leq v^{n}(x,y,t)\leq b-1,~0\leq w^{n}(x,t)\leq \dfrac{\nu}{\mu}(b-1)$ for any $(x,y)\in\ov{\O}_0$ and $t\geq 0$ (see Lemma \ref{lem 23}). By the Arzel\`{a}-Ascoli theorem, one can extract a subsequence of $\left(u_{n_k}(x,y,t),v_{n_k}(x,y,t),w_{n_k}(x,t)\right)$ (still denoted by $\left(u_{n_k}(x,y,t),v_{n_k}(x,y,t),w_{n_k}(x,t)\right)$) such that
\[
	\lim_{k\to\infty}u_{n_k}(x,y,t)=\tilde{u}(x,y,t),\quad
	\lim_{k\to\infty}v_{n_k}(x,y,t)=\tilde{v}(x,y,t)
	\]
	 in $C^{2,1}_{loc}\left(\O_{-z_0}\times\R\right)\cap C_{loc}^{1}\left(\ov{\O}_{-z_0}\times \R\right)$
and $ \displaystyle \lim_{k\to\infty}w_{n_k}(x,t)=\tilde{w}(x,t)$ 
in $C^{2,1}_{loc}(\R\times \R)$.
Therefore, we have that $\tilde{u}$ satisfies
\[
	\left\{ \begin{aligned}
		&\partial_{t}\tilde{u}-d_1\Delta \tilde{u}=\tilde{u}(1-\tilde{u}-a\tilde{v})\geq \tilde{u}(1-\tilde{u}-a(b-1)),&&\quad (x,y)\in\O_{-z_{0}},t\in\R,&\\
		&\partial_{y}\tilde{u}(x,-z_{0},t)=0,&&\quad x\in \mathbb{R},t\in \mathbb{R},&\\
		&\tilde{u}(x,y,t)\geq 0,&&\quad (x,y)\in\ov{\O}_{-z_0},t\in\R.&
	\end{aligned} \right.
\]
In addition, there holds
	$\tilde{u}(0,0,0)\leq \li_{k\to \infty}\frac{1}{n_{k}}\left( v(x_{n_{k}},y_{n_{k}},t_{n_{k}})-\sigma_{0} \right)=0$.
Then we can conclude that $\tilde{u}\equiv 0$ on $\ov{\O}_{-z_0}\times \R$. Indeed, if $z_0>0$, we reach this conclusion directly by using the strong maximum principle. If $z_0=0$ and $\tilde{u}\not\equiv 0$, the strong maximum principle implies that $\tilde{u}(x,y,t)>0$ for any $(x,y)\in\O_{-z_0},t\in\R$. It follows that $\p_y\tilde{u}(0,-z_0,0)>0$. This is a contradiction to the boundary condition $\p_y\tilde{u}(x,-z_0,t)=0$ for all $x\in\R$ and $t\in\R$.
Thus
$(\tilde{v},\tilde{w})$ satisfies
\begin{equation}\label{limit system1}
	\left\{ \begin{aligned}
		&\partial_{t}\tilde{v}-d_2\Delta \tilde{v}=\tilde{v}(-1-\tilde{v}),&&\quad (x,y)\in\O_{-z_{0}},t\in\R,&\\
		&\partial_{t} \tilde{w}-D\tilde{w}_{x x}=-\mu \tilde{w}+\n \tilde{v}(x,-z_{0},t),&&\quad x\in \mathbb{R},t\in \mathbb{R},&\\
		&-d_2\partial_{y}\tilde{v}(x,-z_{0},t)=\mu \tilde{w}(x,t)-\n \tilde{v}(x,-z_{0},t),&&\quad x\in \mathbb{R},t\in \mathbb{R}.&
	\end{aligned} \right.
\end{equation}
For any given $t_{0}<0$, we intend to find a supersolution $(\ov{v},\ov{w})$ of the above system of the form
\[
	\begin{aligned}
		& \ov{v}= (b-1) e^{ \beta_{1}(t-t_{0}) }\left[\frac{\m}{\n+d_2\b_{2}} e^{ -\beta_{1}x-\beta_{2}(y+z_{0}) }+1\right],\vspace{1ex}\\
		& \ov{w}= (b-1) e^{ \beta_{1}(t-t_{0}) } \left(e^{ -\beta_{1}x }+\frac{\n}{\mu} \right),
	\end{aligned}
	\qquad (x,y)\in\ov{\O}_{-z_0},\, t\geq t_0,
\]
where $\beta_{1}<0$ and $\beta_{2}>-\frac{\nu}{d_2}$ are constants that will be chosen appropriately. Substituting these exponential functions into \eqref{limit system1}, it is sufficient to take $\b_1,\b_2$ satisfying
\begin{equation}\label{v die upper}
	 d_2\left(\beta_{1}^{2}+\beta_{2}^{2}\right)-\beta_{1}\leq 1,\quad 
		 D\beta_{1}^{2}-\beta_{1}\leq \frac{\m d_2\b_{2}}{\n+d_2\b_{2}},\quad
	 \beta_1<0,\quad \beta_2>- \frac{\nu}{d_2}.
\end{equation}
In the $\beta_{2}\beta_{1}$-plane, $d_2(\beta_{1}^{2}+\beta_{2}^{2})-\beta_{1}\leq 1$ is a closed disc of radius $r:=\frac{\sqrt{4d_2+1}}{2d_2}$ and centered at $C:=\left(0,\frac{1}{2d_2}\right)$,
and $D\beta_{1}^{2}-\beta_{1}\leq \frac{\m d_2\b_{2}}{\n+d_2\b_{2}}$ with $\beta_2>- \frac{\nu}{d_2}$ represents the closed region to the right of a parabolic-like curve for $\beta_{2}\geq \beta_0:=-\frac{\n}{d_2(1+4D\m)}$. Denote the disc and the region determined by \eqref{v die upper} as $\Sigma$ and $\Gamma$. Observe that $r> \frac{1}{2d_2}=|C|$ and the boundary of $\Gamma$ passes through the origin, we have $\Sigma\cap \Gamma\cap\{(\b_2,\b_1):\b_1<0,\, \beta_{2}>-\frac{\nu}{d_2}\}\not= \emptyset$, see Figure \ref{fig3}. 
\begin{figure}[!ht]
	\centering
		\centering
		\includegraphics[width=0.5\linewidth]{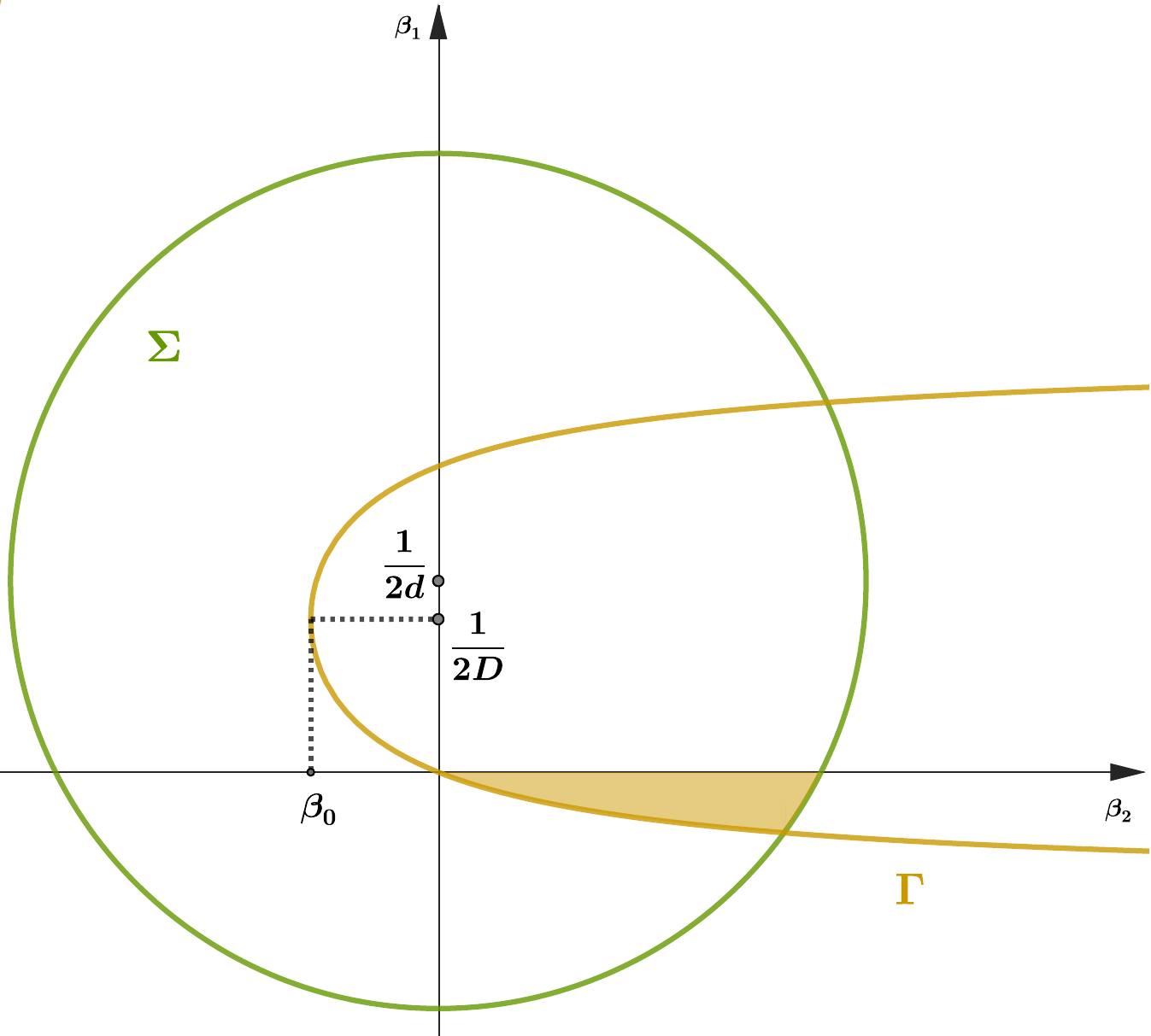}
	\caption{The shaded region corresponds to $\Sigma\cap \Gamma\cap\{(\b_2,\b_1):\b_1<0,\, \beta_{2}>-\frac{\nu}{d_2}\}$, points lying in this region satisfy \eqref{v die upper}.}
	\label{fig3}
\end{figure}

From the construction of $\left(\ov{v},\ov{w}\right)$ and Lemma \ref{lem 23}, one has $\ov{v}(x,y,t_{0})>  b-1\geq\tilde{v}(x,y,t_{0})$ for all $(x,y)\in \ov{\O}_{-z_0}$ and $\ov{w}(x,t_{0})>\frac{\n}{\mu} (b-1)\geq \tilde{w}(x,t_{0})$ for any $x\in\R$. Then the comparison principle for \eqref{limit system1} (\cite[Proposition 3.2]{berestycki+2013}) yields that $\tilde{v}(x,y,t)< \ov{v}(x,y,t)$ and $\tilde{w}(x,t)< \ov{w}(x,t)$ for all $(x,y)\in \ov{\O}_{-z_0},t\geq t_0$. In particular, $\tilde{v}(x,y,0)< \ov{v}(x,y,0)= (b-1)e^{ -\beta_{1}t_{0}}\left[\frac{\m}{\n+d_2\b_{2}} e^{ -\beta_{1}x-\beta_{2}(y+z_{0}) }+1\right]$. Letting $t_0\to-\infty$, we have
\[
\tilde{v}(x,y,0)\leq \li_{ t_{0} \to -\infty } \ov{v}(x,y,0)=0\qquad \text{ locally uniformly in } (x,y)\in\ov{\O}_{-z_0},
\]
which contradicts
\begin{equation}\label{218}
	\tilde{v}(0,0,0)= \li_{ k \to \infty } v^{n_k}(x_{n_{k}},y_{n_{k}},t_{n_{k}})\ge \sigma_{0}>0.
\end{equation}

	\vspace{1.25ex}
	\noindent\textit{Case 2.} 
$\li_{n\to \infty}y_{n}=+\infty$.
	\vspace{1.25ex}
	
Up to a subsequence, we have a triplet $\left(\tilde{u}(x,y,t),\tilde{v}(x,y,t),\tilde{w}(x,t)\right)$ such that $\li_{n\to \infty}u_{n}(x,y,t)= \tilde{u}(x,y,t)$, $\li_{n\to \infty}v_{n}(x,y,t)=\tilde{v}(x,y,t)$ uniformly in $C^{2,1}_{loc}(\R^2\times\R)$ and $\li_{n\to \infty}w_{n}(x,t)=\tilde{w}(x,t)$ uniformly in $C^{2,1}_{loc}(\R\times\R)$. Similarly to {\it Case 1}, we have $\tilde{u}\equiv 0$. Hence $\tilde{v}$ is an entire solution of the equation
\begin{equation}\label{limit system}
\partial_{t}\tilde{v}-d_2\Delta \tilde{v}=\tilde{v}(-1-\tilde{v}), \qquad (x,y)\in\R^{2},t\in \mathbb{R}.
\end{equation}
For any $t_0<0$, we take $\ov{v}(t)= (b-1)e^{-(t-t_{0})}$. Then $\ov{v}(t_0)= b-1\geq\tilde{v}(x,y,t_0)$, and thus $\ov{v}(t)\geq\tilde{v}(x,y,t)$ for all $(x,y,t)\in \R^2\times[t_0,+\infty)$. Since $\li_{t_0\to -\infty}\ov{v}(t)= 0$ locally uniformly in $t\in\R$, we have $\tilde{v}(0,0,0)\leq\displaystyle \lim_{t_0\to-\infty}\ov{v}(0)=0$ and get a contradiction with \eqref{218} again.

Next, we prove \eqref{w die without u}. Suppose by contradiction again that there exist $\sigma_{0}\in\left(0,\frac{\nu}{\mu}(b-1)\right)$, $\left(u_{0,n},v_{0,n},w_{0,n}\right)$ satisfying \ref{h2} and $x_{n}, t_{n}$ such that
\begin{equation}\label{37}
	\lim_{n \to \infty}t_{n}=+\infty,\qquad w^{n}(x_{n},t_{n})> \sigma_{0}+n u^{n}(x_{n},0,t_{n}), \quad \forall n\in \mathbb{N}.
\end{equation}
Then $\displaystyle \lim_{n \to \infty}u^{n}(x_n,0,t_n)\leq \lim_{n \to \infty}\frac{1}{n}\left(\frac{\nu}{\mu}(b-1)-\s_0\right)=0$. Therefore, applying \eqref{v die without u} yields that for any $\e>0$, there exists $M_{\e}>0$ such that
	$\displaystyle \limsup_{n \to \infty}v^{n}(x_n,0,t_n)\leq \e+M_{\e}\lim_{n \to \infty}u^{n}(x_n,0,t_n)=\e$.
It implies that 
\begin{equation}\label{38}
	\displaystyle \lim_{n \to \infty}v^{n}(x_n,0,t_n)=0
\end{equation}
by the arbitrariness of $\e>0$.
	By Arzel\`{a}-Ascoli theorem again, we have that there exist functions $\tilde u,\tilde v\in C^{2,1}\left(\O_0\times\R\right)\cap C^1\left(\ov{\O}_0\times\R\right)$, $\tilde w\in C^{2,1}(\R\times \R)$ and subsequences of $\{x_n\}$, $\{t_n\}$ and $\left\{(u_{0,n},v_{0,n},w_{0,n})\right\}$ (still denoted by $\{x_n\}$, $\{t_n\}$ and $\left\{(u_{0,n},v_{0,n},w_{0,n})\right\}$) such that
\begin{equation*}
	\li_{n\to\infty} u^{n}(x+x_{n},y,t+t_{n})=\tilde{u}(x,y,t),\qquad
	\li_{n\to\infty} v^{n}(x+x_{n},y,t+t_{n})=\tilde{v}(x,y,t)
\end{equation*}
	in $C_{loc}^{2,1}(\O_0\times\R)\cap C^1_{loc}\left(\ov{\O}_0\times\R\right)$, and
$ \li_{n\to\infty} w^{n}(x+x_{n},t+t_{n}) = \tilde{w}(x,t)$
 in $C_{loc}^{2,1}(\R\times\R)$. Hence, $(\tilde{u},\tilde{v},\tilde{w})$ is an entire solution of \eqref{prey-predator-road}. On the other hand, \eqref{38} implies that $\tilde{v}(0,0,0)=0$ and thereby
\begin{equation}\label{trivial}
	\tilde{v}(x,y,t)\equiv 0,\qquad \forall (x,y)\in\ov{\O}_0,t\in\R.
\end{equation}
If \eqref{trivial} is not valid, then the maximum principle yields $\tilde{v}(x,y,t)>0,~\forall (x,y,t)\in\O_0\times \R$. We are led to $\p_y\tilde{v}(0,0,0)>0$, which contradicts $\p_{y}\tilde{v}(0,0,0)=-\frac{\mu}{d_2}\tilde{w}(0,0)\leq 0$. Therefore, $\tilde{w}$ satisfies
\[
	\partial_{t}\tilde{w}-D \partial_{x x}\tilde{w}=-\mu \tilde{w},\qquad \forall (x,t)\in \mathbb{R}^2.
\]
For any $t_{0}>0$, $\ov{w}(t)= \frac{\n}{\mu}(b-1)e^{ -\mu(t+t_{0}) }$ is a supersolution of the above equation. Note that $0\leq \tilde w\leq\frac{\nu}{\mu}(b-1)=\ov{w}(-t_0)$, thus from the comparison principle, one has
$\tilde{w}(x,t)\le \overline{w}(t)$  for all $x\in \mathbb{R}$ and $t\ge -t_{0}$.
Letting $t_0\to +\infty$, we get $\tilde{w}(x,0)\equiv 0$.
While by \eqref{37}, we have $\tilde{w}(0,0)\geq\sigma_{0}>0$. There is a contradiction and thus we have completed the proof.
\end{proof}
\subsubsection{A truncated system}\label{truncated_sys}
Take $d>0,\, D>0,\, M_1>1,\, M_2>0,\, T_0> 0$ and $\vs_{*}>0$. For any $0\leq\rho<M_1-1$ and $\vt\in\left[-\frac{\pi}{2},\frac{\pi}{2}\right]$, the spreading velocity of the system
\begin{equation}\label{61}
	\left\{
		\begin{aligned}
			&\p_tv-d\D v=\left[M_1-1-\rho-(1+M_2)v\right]v,&& \quad(x,y)\in\O_{0},t>T_0, &\\
			&\p_t w-D \p_{xx}w=-\m w+\n v(x,0,t),&&\quad x\in\R,t>T_0, &\\
			& -d\p_y v(x,0,t)=\m w-\n v(x,0,t),&&\quad x\in\R,t>T_0&
			\end{aligned}
		\right.
\end{equation}
with nonnegative, bounded, compactly supported and nontrivial initial data $(v(x,y,T_0),w(x,T_0))$ in direction $(\sin\vt,\cos\vt)$ is determined by its linearization at $(0,0)$. We denote this speed by $\vs_{\rho}(\vt)$.
By Lemma \ref{berestycki+2016 thm2.1} and Remark \ref{monotone}, we konw that $\vs_{\rho}(\cdot)\geq \vs_{\rho}(0)=2\sqrt{d(M_1-1-\rho)}$ on $\left[-\frac{\pi}{2},\frac{\pi}{2}\right]$ and that for any $\vt\in\left[-\frac{\pi}{2},\frac{\pi}{2}\right]$, $\vs_{\rho}(\vt)$ is decreasing in $\rho$. For $t>T_0$ and $0<m<\vs_{*}T_0$, call
\[
	\o_{t}:=\left\{(x,y)\in \ov{\O}_{0}:
		(x,y)\cdot(\sin\vt,\cos\vt)\leq \vs_{*}t,\quad\forall \vt\in\left[-\frac{\pi}{2},\frac{\pi}{2}\right]
	\right\}
\]
and
\[
	\tilde{\o}_{t}:=\left\{(x,y)\in\o_{t}:
		(x,y)\cdot(\sin\vt,\cos\vt)\leq \vs_{*}t-m,\quad\forall \vt\in\left[-\frac{\pi}{2},\frac{\pi}{2}\right]
	\right\}.
\]
Define $\chi:\O_{0}\times(T_0,+\infty)\mapsto [0,1]$ as
	$
		\chi(x,y,t)=\psi\left(|(x,y)|-\vs_{*}t+m\right),
	$
	where $\psi:\R\to[0,1]$ is a smooth function satisfying
	\[
		\psi(s)=\left\{
			\begin{aligned}
			 	&1,&&s\leq 0,&\\
				&0,&&s\geq m,&
			\end{aligned}
		\right.
		\qquad \text{and} \qquad
		\psi^{\prime}(\cdot)\leq 0.
	\]
 Then $\chi(x,y,t)$  is a smooth function satisfying the following properties:
 \begin{itemize}
	\item For any $t>T_0$,
	\[
		\chi(x,y,t)=\left\{
			\begin{aligned}
				 & 1, && (x,y)\in\tilde{\o}_{t},\\
				 & 0, &&(x,y)\in\O_{0}\setminus\o_{t}.
			\end{aligned}
		\right.
	\]
	\item For all $(x,y)\in\O_{0}$, $\chi(x,y,\cdot)$ is nondecreasing in $(T_0,+\infty)$.
	\item For all $y>0, t>T_0$, $\chi(\cdot,y,t)$ is symmetric about the vertical axis.
 \end{itemize}

Define $c(\vt)=\min\left\{ \vs_{*},\vs_{0}(\vt) \right\}$. In this subsection, we present a brief proof that $c(\vt)$ is a lower bound for the spreading speed in the direction $(\sin\vt,\cos\vt)$ of the following truncated system
\begin{equation}\label{62}
\left\{
	\begin{aligned}
		&\p_tv-d\D v=\big[-1+(M_1-M_2 v)\chi-v\big]v,&& \quad(x,y)\in\O_{0},t>T_0, &\\
		&\p_t w-D \p_{xx}w=-\m w+\n v(x,0,t),&&\quad x\in\R,t>T_0, &\\
		& -d\p_y v(x,0,t)=\m w-\n v(x,0,t),&&\quad x\in\R,t>T_0,&
	\end{aligned}
\right.
\end{equation}
under the following assumption:
\begin{enumerate}[label=\textbf{(A\arabic*)}]
	\item\label{a1} $(v(x,y,T_0),w(x,T_0))$ is nonnegative, bounded and nontrivial.
\end{enumerate}

Fix $p,\, q,\, s\in\mathbb{R}$, consider the following system:
	\begin{equation}\label{sub-drift}
		\left\{
			\begin{aligned}
				&\p_tv-d\D v+p\p_x v=\big[-1+(M_1-M_2 v)\chi(x-st,y,t)-v\big]v,&& (x,y)\in\O_{0},t>T_0, &\\
				&\p_t w-D \p_{xx}w+q\p_x w=-\m w+\n v(x,0,t),&&x\in\R,t>T_0, &\\
				& -d\p_y v(x,0,t)=\m w(x,t)-\n v(x,0,t),&&x\in\R,t>T_0,&
			\end{aligned}
		\right.
	\end{equation}
 
\begin{definition}\label{truncated-generalized-sub}
	A pair $(\un{v},\un{w})$ {\rm (resp. $(\ov{v},\ov{w})$)} is a generalized subsolution {\rm (resp. supersolution)} of \eqref{sub-drift}
	if $\un{v},\, \un{w}$ {\rm (resp. $\ov{v},\, \ov{w}$)} are continuous and satisfy the following properties:
	\begin{enumerate}[label=\rm{(\roman*)}]
		\item for any $x\in\mathbb{R},\, t>T_0$, there is a function $w$ such that $w\leq\underline{w}$ {\rm (resp. $w\geq\overline{w}$)} in a neighbourhood of $(x,t)$ and, at $(x,t)$ {\rm (}in the classical sense{\rm )},
		\[
			w=\underline{w}\, (w=\ov{w}),\qquad\partial_tw-D\partial_{xx}w+q\p_x w+\mu w\leq\nu\underline{v}(x,0,t)\,(\geq\nu\overline{v}(x,0,t));		
		\]
          \item for any $(x,y)\in\ov{\O}_{0},t>T_0$, there is a function $v$ such that $v\leq\underline v$ {\rm (resp. $v\geq\ov v$)} in a neighbourhood of $(x,y,t)$ and, at $(x,y,t)$,
          \[
		  	\begin{aligned}
		 		&v=\underline{v}\, ({\rm resp.}~v=\overline{v}),\\
				&\left\{\begin{aligned}
					&\partial_tv-d\Delta v+p\p_x v\leq (\geq)\big[-1+(M_1-M_2 v)\chi(x-st,y,t)-v\big]v, && \mathrm{~if~}y>0,\\
					&-d\partial_y v(x,0,t)+\nu v(x,0,t)\leq\mu\underline{w}\, (\geq\mu\overline{w}), && \mathrm{~if~}y=0.
				\end{aligned}
				\right.
		  	\end{aligned}
          \]
	\end{enumerate}
\end{definition}

The comparison principle, well-posedness and generalized comparison principle for \eqref{sub-drift} hold. Namely, we have the following three lemmas.

\begin{lemma}\label{lem62}
	Let $(\un{v},\un{w})$ and $(\ov{v},\ov{w})$ be respectively a generalized subsolution and a generalized supersolution  of \eqref{sub-drift} satisfying $(\un{v},\un{w})\left|_{t=T_0}\leq (\ov{v},\ov{w})\right|_{t=T_0}$. In addition, $(\un{v},\un{w})\geq (0,0)$, $(\ov{v},\ov{w})\geq (0,0)$ and $(\un{v},\un{w})$ is bounded from above. Then $\left(\ov{v},\ov{w}\right)\geq\left(\un{v},\un{w}\right)$ for all $t>T_0$.
\end{lemma}
\begin{lemma}\label{lem33}
	Let nonnegative pairs $\left(\ov{v},\ov{w}\right)$ and $\left(\un{v},\un{w}\right)$ be respectively a supersolution and a subsolution of \eqref{sub-drift} satisfying $\left(\ov{v},\ov{w}\right)\geq \left(\un{v},\un{w}\right)$ at $t=T_0$. Assume further that $(\un{v},\un{w})$ is bounded from above. Then
	\begin{equation}\label{comparison-trun}
		\left(\ov{v},\ov{w}\right)\geq \left(\un{v},\un{w}\right), \qquad \forall\, (x,y)\in\ov{\O}_{0},\, t\geq T_0.
	\end{equation}
	And either $(\un{v},\un{w})<(\ov{v},\ov{w})$ for all $t>T_0$ or there exists $T>T_0$ such that $(\un{v},\un{w})=(\ov{v},\ov{w})$ for $t\in[T_0,T]$.
\end{lemma}
\begin{lemma}\label{lem61}
	Assume that {\rm\ref{a1}} holds. Then problem \eqref{sub-drift} admits a unique nonnegative solution. Additionally, the solution satisfies
\[
	\begin{aligned}
	 	& v(x,y,t)\leq \max\left\{\sup_{\ov{\O}_{0}}v(x,y,T_0),\, \frac{\m}{\n}\sup_{\R}w(x,T_0),\, \frac{M_1-1}{M_2+1}\right\}=:\ov{V},\quad &&\forall (x,y)\in\ov{\O}_{0},t\geq T_0&\\
		& w(x,t)\leq \frac{\n}{\m}\ov{V}=:\ov{W}, \quad &&\forall x\in\mathbb{R},t\geq T_0.&
	\end{aligned}
\]
\end{lemma}
The three lemmas can be proved by  arguments similar to those for Proposition 4.3 in \cite{berestycki+2016} and Propositions 3.1 and 3.2 in \cite{berestycki+2013}. Here we omit the proofs.  
In the following, we consider the propagation dynamics of \eqref{62}.
\begin{theorem}\label{thm63}
	Let $(v,w)$ be the solution of \eqref{62} under {\rm\ref{a1}}. Then for any  $0<c<c(\frac{\pi}{2})$ 
	and $l>0$, there holds
	\[
		\limsup_{t\to\infty}\sup_{\substack{|x|\leq c(t-T_0)\\0\leq y<l}}(v(x,y,t),w(x,t))\leq\frac{M_1-1}{M_2+1}\left(1,\frac{\n}{\m}\right).
	\]
\end{theorem}
\begin{proof} Fix  $0<c<c(\frac{\pi}{2})$ and $l>0$. We prove the theorem by a contradiction argument.
Suppose on the contrary that there exist $\epsilon>0$, $\{(x_n,y_n)\}_{n\in\N}$ and $\{t_n\}_{n\in\N}$, where $|x_n|\leq c(t_n-T_0)$ and $0\leq y_n\leq l$ for all $n\in\N$ and $t_n\to +\infty$ as $n\to \infty$,  such that either 
	\[
		 v(x_n,y_n,t_n)\geq \dfrac{M_1-1}{M_2+1} +\epsilon,\quad \forall \, n\in\N,
\]
or 
	\[
		 w(x_n,y_n,t_n)\geq \dfrac{\nu(M_1-1)}{\mu(M_2+1)} +\epsilon, \quad \forall \, n\in\N.
	\]
Without loss of generality, we assume that the former holds. In addition, we assume that $y_n\to \bar{y}$ as $n\to \infty$. For this end, we denote
	\[
		v_n(x,y,t):=v\left(x+x_n,y,t+t_n\right),\quad w_n(x,t):=w\left(x+x_n,t+t_n\right),\quad \forall\, (x,y)\in\ov{\O}_{0},\, t\geq T_0-t_n.
	\]
Then we have that (up to extracting a subsequence) $(v_n,w_n)$ converges to $(v_{\infty},w_{\infty})$ locally uniformly  for some $v_{\infty}\in C^{2+\alpha,1+\frac{\alpha}{2}}(\O_{0}\times \R)\cap C^{1+\alpha,\frac{1+\alpha}{2}}(\ov{\O}_{0}\times\R)$ and $w_{\infty}\in C^{2+\alpha,1+\frac{\alpha}{2}}(\R\times \R)$. In particular, we have $v_{\infty}(0,\bar{y},0)\geq \frac{M_1-1}{M_2+1} +\epsilon$. Due to $0<c<c(\frac{\pi}{2})\leq \vs^{*}$, there holds that $\chi\left(x+x_n,y,t+t_n\right)\to 1$ as $n\to \infty$ locally uniformly in $(x,y,t)\in\Omega_{0}\times \mathbb{R}$. Therefore, we have $(v_{\infty},w_{\infty})$ is a solution of
	\begin{equation}\label{313}
		\left\{
		\begin{aligned}
			&\p_tv-d\D v=\left[M_1-1-(1+M_2)v\right]v,&& \quad(x,y)\in\O_{0},t\in\R, &\\
			&\p_t w-D \p_{xx}w=-\m w+\n v(x,0,t),&&\quad x\in\R,t\in\R, &\\
			& -d\p_y v(x,0,t)=\m w(x,t)-\n v(x,0,t),&&\quad x\in\R,t\in\R.&
			\end{aligned}
		\right.
	\end{equation}
		On the other hand, $(0,0)\leq (v_{\infty},w_{\infty})\leq \left(\ov{V},\ov{W}\right)$, and $\left(\ov{V},\ov{W}\right)$ is a supersolution of \eqref{313}. Take $t_0\in\R$ and let $(\ov{v},\ov{w})$ be the solution of
	\begin{equation}\label{314}
		\left\{
		\begin{aligned}
			&\p_tv-d\D v=\left[M_1-1-(1+M_2)v\right]v,&& \quad(x,y)\in\O_{0},t\geq t_0, &\\
			&\p_t w-D \p_{xx}w=-\m w+\n v(x,0,t),&&\quad x\in\R,t\geq t_0, &\\
			& -d\p_y v(x,0,t)=\m w(x,t)-\n v(x,0,t),&&\quad x\in\R,t\geq t_0&
			\end{aligned}
		\right.
	\end{equation}
	with $\left(\ov{v}(x,y,t_0),\ov{w}(x,t_0)\right)=\left(\ov{V},\ov{W}\right)$ on ${\ov{\O}_{0}}$. Then $(\ov{v},\ov{w})$ is nonincreasing in $t\geq t_0$ and converges to the unique nonnegative, nontrivial, bounded solution $\frac{M_1-1}{M_2+1}\left(1,\frac{\n}{\m}\right)$ of system \eqref{314} (see \cite[Proposition 4.1]{berestycki+2013}).
Thus
	$\left(v_{\infty}(x,y,t),w_{\infty}(x,t)\right)\leq \left(\ov{v}(x,y,t+t_0),\ov{w}(x,t+t_0)\right)$, for all$(x,y)\in\ov{\O}_{0},t\geq 0$.
	Letting $t_0\to+\infty$ and taking $t=0$, we obtain 
	\[
		\left(v_{\infty}(x,y,0),w_{\infty}(x,0)\right)\leq\frac{M_1-1}{M_2+1}\left(1,\frac{\n}{\m}\right),\qquad \forall (x,y)\in\ov{\O}_{0}.
	\]
This contradicts the inequality $v_{\infty}(0,\bar{y},0)\geq \frac{M_1-1}{M_2+1} +\epsilon$. Thus the proof is completed.
\end{proof}
\begin{lemma}\label{a2}
	For any $\vt\in\left[-\frac{\pi}{2},\frac{\pi}{2}\right]$ and $\e\in(0,c(\vt))$, there exist $c:=c_{\vt}\in( c(\vt)-\e, c(\vt))$ such that system \eqref{62} admits a generalized subsolution $\left(\un{\phi}_{\vt,c}(x,y,t),\un{\psi}_{\vt,c}(x,t)\right)$ in $\ov{\Omega}_{0}\times [T_{\vt,c},+\infty)$ for some $T_{\vt,c}\geq T_0$. 
	Additionally, $c_{\vt}$, $T_{\vt,c}$ and $\left(\un{\phi}_{\vt,c}(x,y,t),\un{\psi}_{\vt,c}(x,t)\right)$ satisfy the following properties:
	\begin{itemize}
		\item $c_{\vt}=c_{-\vt}, T_{\vt,c}=T_{-\vt,c}$.
		\item ${\rm supp}~\un{\phi}_{\vt,c}(\cdot,\cdot,T_{\vt,c})$ and ${\rm supp}~\un{\psi}_{\vt,c}(\cdot,T_{\vt,c})$ are bounded.
		\item There exists $(\usim{x},\usim{y})\in\ov{\O}_{0}$ such that $			\un{\phi}_{\vt,c}(\usim{x}+ct\sin\vt,\usim{y}+ct\cos\vt,t)=\un{\phi}_{\vt,c}(\usim{x},\usim{y},T_{\vt,c})>0$  for all $t\geq T_{\vt,c}$. 
	\end{itemize}
	Furthermore, both
		$\left(\un{\phi}_{\frac{\pi}{2},c}(-x+ct,y,t),\un{\psi}_{\frac{\pi}{2},c}(-x+ct,t)\right)$ and
$\left(\un{\phi}_{-\frac{\pi}{2},c}(x-ct,y,t),\un{\psi}_{-\frac{\pi}{2},c}(x-ct,t)\right)$
are nonnegative,  compactly supported, nontrivial generalized stationary subsolution of \eqref{sub-drift} with $p=q=s=c$ in $\ov{\O}_0\times [T_{\frac{\pi}{2},c},+\infty)$.
\end{lemma}
\begin{proof}
	Fix $\vt\in\left[-\frac{\pi}{2},\frac{\pi}{2}\right]$ and $\e>0$. It follows from Proposition \ref{uniform_epsilon} that for any $0\leq\rho<M_1-1$, there is  $c\in\left(\vs_{\rho}(\vt)-\e,\vs_{\rho}(\vt)\right)$ such that system \eqref{61} admits a generalized subsolution $(\phi_{\rho,\vt,c}(x,y,t),\psi_{\rho,\vt,c}(x,t))$, which satisfies the following properties
	\begin{itemize}
		\item $\phi_{\rho,\vt,c}(\cdot,\cdot,T_0)$ and $\psi_{\rho,\vt,c}(x,t)(\cdot,T_0)$ are compactly supported.
		\item There exists $l_{\rho,\vartheta,c}>0$, which is uniformly bounded in $\vt\in\left[-\frac{\pi}{2},\frac{\pi}{2}\right]$, such that
		\[
			\begin{aligned}
				& {\rm supp}~ \phi_{\rho,\vt,c}(\cdot,\cdot,t)={\rm supp}~ \phi_{\rho,\vt,c}(\cdot,\cdot,T_0)+c(\si\vt,\c\vt)(t-T_0),\\
				& {\rm supp}~ \psi_{\rho,\vt,c}(x,t)(\cdot,t)={\rm supp}~ \psi_{\rho,\vt,c}(x,t)(\cdot,T_0)+cl_{\delta,\vt,c}(t-T_0),
			\end{aligned}
			\qquad \forall\, t\geq T_0>0.
		\]
		\item There exists $(\hat{x},\hat{y})$ such that
		$\phi_{\rho,\vt,c}(\hat{x}+ct\sin\vt,\hat{y}+ct\cos\vt,t)=\phi_{\rho,\vt,c}(\hat{x},\hat{y},T_0)>0$  for $t\geq T_0$.

		\item $(\phi_{\rho,\frac{\pi}{2},c}(-x+ct,y,t),\psi_{\rho,\frac{\pi}{2},c}(-x+ct,t))$ and $(\phi_{\rho,-\frac{\pi}{2},c}(x-ct,y,t),\psi_{\rho,-\frac{\pi}{2},c}(x-ct,t))$ are nonnegative, compactly supported, nontrivial generalized stationary subsolution of the following system
	\[\left\{
		\begin{aligned}
		&\p_tv-d\D v+c\partial_x v=\big[-1+(M_1-M_2 v)\chi-v\big]v,&& \quad(x,y)\in\O_{0},t>T_0, &\\
		&\p_t w-D \p_{xx}w+c\partial_x w=-\m w+\n v(x,0,t),&&\quad x\in\R,t>T_0, &\\
		& -d\p_y v(x,0,t)=\m w-\n v(x,0,t),&&\quad x\in\R,t>T_0,&
		\end{aligned}
	\right.\]
	\end{itemize}
	We consider the following two complementary cases.
	
	\vspace{1.5ex}
	\noindent\textit{Case 1.} $\vs_0(\vartheta)> \vs_*$.

	In this case, there is $\sigma:=\sigma(\vartheta)>0$ such that $\vs_{\sigma}(\vartheta)=\vs_*=c(\vartheta)$. Define
	\[
		\left(\un{\phi}_{\vartheta,c}(x,y,t),\un{\psi}_{\vartheta,c}(x,t)\right):=(\phi_{\sigma,\vartheta,c}(x,y,t),\psi_{\sigma,\vartheta,c}(x,t)).
	\]
	Since ${\rm supp}~\phi_{\sigma,\vartheta,c}(\cdot,\cdot,T_0)$ is bounded, then for any $(x,y)\in {\rm supp}~\un{\phi}_{\vartheta,c}(\cdot,\cdot,t)$, there is $(\xi,\zeta)\in {\rm supp}~\un{\phi}_{\vartheta,c}(\cdot,\cdot,T_0)$ and $T_{\vartheta,c}\geq T_0$ such that
	\begin{equation}\label{truncated_support}
		(x,y)\cdot(\sin\vt,\cos\vt)\leq |(\xi,\zeta)|+c(t-T_0)< \vs_{*}t-m, \qquad \forall\, t\geq T_{\vartheta,c}.
	\end{equation}
	
	\vspace{1.5ex}
	\noindent\textit{Case 2.} $\vs_0(\vartheta)\leq \vs_*$.
	 
	In this case, we take
	$		\left(\un{\phi}_{\vartheta,c}(x,y,t),\un{\psi}_{\vartheta,c}(x,t)\right):=(\phi_{0,\vartheta,c}(x,y,t),\psi_{0,\vartheta,c}(x,t))$.
	For any $(x,y)\in {\rm supp}~\un{\phi}_{\vartheta,c}(\cdot,\cdot,t)$, there is $(\xi,\zeta)\in {\rm supp}~\un{\phi}_{\vartheta,c}(\cdot,\cdot,T_0)$ and $T_{\vartheta,c}\geq T_0$ such that \eqref{truncated_support} holds.
	Therefore, by taking
$\left(\usim {x},\usim{y}\right):=(\hat{x},\hat{y})-c(T_{\vartheta,c}-T_0)(\sin\vartheta,\cos\vartheta)$,
we can verify that $\left(\un{\phi}_{\vartheta,c}(x,y,t),\un{\psi}_{\vartheta,c}(x,t)\right)$ satisfies the properties in Lemma \ref{a2}.

	Finally, we provide the explanation that $\left(\un{\phi}_{\vartheta,c},\un{\psi}_{\vartheta,c}\right)$ is a generalized subsolution of \eqref{62} in $\ov{\Omega}_{0}\times [T_{\vartheta,c},+\infty)$. Note that $\left(\tilde{\phi}_{\frac{\pi}{2},c}(x,y),\tilde{\psi}_{\frac{\pi}{2},c}(x)\right)$ and $\left(\tilde{\phi}_{-\frac{\pi}{2},c}(x,y),\tilde{\psi}_{-\frac{\pi}{2},c}(x)\right)$ are generalized subsolutions of \eqref{sub-drift} with the desired properties, provided this conclusion holds. Since \eqref{truncated_support}, we have ${\rm supp}~\un{\phi}_{\vartheta,c}(\cdot,\cdot,t)\subset {\rm Int}~ \tilde{\omega}_{t}$ for any $t\geq T_{\vartheta,c}$, where $\tilde{\omega}_{t}$ is a set such that $\chi(x,y,t)=1$, $\forall (x,y)\in\tilde{\omega}_{t}$.
Let us first note that $\left(\un{\phi}_{\vartheta,c},\un{\psi}_{\vartheta,c}\right)$ is a generalized subsolution of \eqref{62} with $\chi(x,y,t)$ replaced by $1$ in $\ov{\O}_0\times [T_{\vartheta,c},+\infty)$. Therefore, we need only verify that $\un{\phi}_{\vartheta,c}$ is a generalized subsolution of
\[
	\left\{
	\begin{aligned}
		&\p_tv-d\D v=\big[-1+(M_1-M_2 v)\chi-v\big]v,&& \quad(x,y)\in\O_{0},t>T_{\vartheta,c}, &\\
		& -d\p_y v(x,0,t)=\m \un{\psi}_{\vartheta,c}-\n v(x,0,t),&&\quad x\in\R,t>T_{\vartheta,c},&
	\end{aligned}
\right.
\]
If $t> T_{\vartheta,c}, (x,y)\in\tilde{\omega}_{t}$, there is $\phi_{\vartheta,c}$ satisfying $\phi_{\vartheta,c}\leq \un{\phi}_{\vartheta,c}$ in a neighborhood of $(x,y,t)$, and at $(x,y,t)$,
\begin{equation}\label{317}
		\begin{aligned}
		& \phi_{\vartheta,c}= \un{\phi}_{\vartheta,c},&&& \\
		&
			\p_t \phi_{\vartheta,c}-d_2\D \phi_{\vartheta,c}
			\leq 
			\big[-1+(M_1-M_2\phi_{\vartheta,c})\chi-\phi_{\vartheta,c}\big]\phi_{\vartheta,c},
		&& \quad {\rm if}~y> 0,&\\
		& -d_2\p_y\phi_{\vartheta,c}(x,0,t)\leq \m\un{\psi}_{\vartheta,c}(x,t)-\n\phi_{\vartheta,c}(x,0,t),&&\quad {\rm if}~y= 0.&
	\end{aligned}
\end{equation}
If $t> T_{\vartheta,c}$ and $(x,y)\in\ov{\O}_{0}\setminus \tilde{\omega}_{t}$, we take $\phi_{\vartheta,c}\equiv 0$. Then $\phi_{\vartheta,c}\leq \un{\phi}_{\vartheta,c}$ and \eqref{317} holds at $(x,y,t)$
Therefore, $\left(\un{\phi}_{\vartheta,c},\un{\psi}_{\vartheta,c}\right)$ is a generalized subsolution of \eqref{62}. This completes the proof.
\end{proof}

\begin{lemma}
	For any $\e>0$, let $c\in\left(c(\frac{\pi}{2})-\e,c(\frac{\pi}{2})\right)$ be the value obtained in Lemma \ref{a2} such that \eqref{62} admits generalized subsolutions $\left(\un{\phi}_{\frac{\pi}{2},c}(x,y,t),\un{\psi}_{\frac{\pi}{2},c}(x,t)\right)$ and $\left(\un{\phi}_{-\frac{\pi}{2},c}(x,y,t),\un{\psi}_{-\frac{\pi}{2},c}(x,t)\right)$. Denote by $\left(v_{\gamma,c},w_{\gamma,c}\right)$
	and $\left(\tilde{v}_{\gamma,c},\tilde{w}_{\gamma,c}\right)$ the solutions of \eqref{62} on $\ov{\O}_0\times \left[T_{\frac{\pi}{2},c},+\infty\right)$, starting from
	$\gamma\left(\underline{\phi}_{\frac{\pi}{2},c}(x,y,T_{\frac{\pi}{2},c}),\underline{\psi}_{\frac{\pi}{2},c}(x,T_{\frac{\pi}{2},c})\right)$ and
	$\gamma\left(\underline{\phi}_{-\frac{\pi}{2},c}(x,y,T_{\frac{\pi}{2},c}),\underline{\psi}_{-\frac{\pi}{2},c}(x,T_{\frac{\pi}{2},c})\right)$, respectively. Then for any $0<\gamma\leq 1$, one has
	\begin{equation}\label{316}
		\liminf_{t \to \infty}(v_{\gamma,c}(-x+ct,y,t),w_{\gamma,c}(-x+ct,t))\geq \frac{M_1-1}{M_2+1}\left(1,\frac{\n}{\m}\right)
	\end{equation}
	and
	\begin{equation}\label{lem382}
		\liminf_{t \to \infty}(\tilde{v}_{\gamma,c}(x-ct,y,t),\tilde{w}_{\gamma,c}(x-ct,t))\geq \frac{M_1-1}{M_2+1}\left(1,\frac{\n}{\m}\right)
	\end{equation}
	locally uniformly in $(x,y)\in\ov{\O}_{0}$.
\end{lemma}

\begin{proof} 
	We prove \eqref{316}, and \eqref{lem382} can be proven similarly. Fix $\e>0$ and $0<\gamma\leq 1$. It is sufficient to show that for any $l>0$, there holds
\[
	\liminf_{t\to+\infty}(v_{\gamma,c}(-x+ct,y,t),w_{\gamma,c}(-x+ct,t))\geq\dfrac{M_1-1}{M_2+1}\left(1,\dfrac{\n}{\m}\right)\quad 	\text{\rm  uniformly in }  (x,y)\in [-l,l]\times[0,l].
	\] 
Fix $l>0$. Suppose on the contrary that there exist $\varepsilon_0>0$ small enough, $\{(x_n,y_n)\}_{n\in\N}\subset [-l,l]\times[0,l]$, $\{t_n\}_{n\in\N}$ with $t_n\to +\infty$ as $n\to \infty$ such that either 
	\begin{equation}\label{www1}
		 v_{\gamma,c}(-x_n+ct_n,y_n,t_n)\leq \dfrac{M_1-1}{M_2+1} -\varepsilon_0,\quad \forall \ n\in\N,
\end{equation}
or 
	\[
		 w_{\gamma,c}(-x_n+ct_n,t_n)\leq \dfrac{\nu(M_1-1)}{\mu(M_2+1)} -\varepsilon_0, \quad \forall \ n\in\N.
	\]
Without loss of generality, we assume that the former holds. Moreover, we assume that $x_n\to \bar{x}$ and $y_n\to \bar{y}$ as $n\to \infty$.	We define
	\[
		v^n_{\gamma}(x,y,t):=v_{\gamma,c}\left(-x+c(t+t_n),y,t+t_n\right),\quad w^n_{\gamma}(x,y,t):=w_{\gamma,c}\left(-x+c(t+t_n),t+t_n\right)
	\]
	for all $(x,y)\in\ov{\O}_{0},t\geq T_{\frac{\pi}{2},c}-t_n$.
	Then there exist a subsequence of $\left\{ t_n \right\}$ (still denoted by $\left\{ t_n \right\}$), and functions $v^{\infty}_{\gamma}\in C^{2+\alpha,1+\frac{\alpha}{2}}(\O_{0}\times \R)\cap C^{1+\alpha,\frac{1+\alpha}{2}}(\ov{\O}_{0}\times\R)$, $w^{\infty}_{\gamma}\in C^{2+\alpha,1+\frac{\alpha}{2}}(\R\times \R)$ such that
	\[
		\lim_{n \to \infty}\left(v^n_{\gamma},w^n_{\gamma}\right)(x,y,t)=\left(v^{\infty}_{\gamma},w^{\infty}_{\gamma}\right)(x,y,t)
	\]
	locally uniformly in $(x,y,t)\in\ov{\O}_{0}\times \mathbb{R}$. Since for any $(x,y)\in\O_{0}$ and $t\in\R$, there is $N$ large enough such that
	\[
		\left(x-c\left(t+t_n\right),y\right)\cdot(\sin\vt,\cos\vt)
		\leq  \left|\left(x-c\left(t+t_n\right),y\right)\right|
		\leq  c\left(\frac{\pi}{2}\right)(t+t_n)-m\leq \vs_{*}(t+t_n)-m,
	\]
	for any $\vt\in\left[-\frac{\pi}{2},\frac{\pi}{2}\right]$ and $n\geq N$,
	we have
	\[
		\lim_{n \to \infty}\chi\left(x-c\left(t+t_n\right),y,t+t_n\right)= 1 \qquad\text{locally uniformly in}~(x,y,t)\in\Omega_{0}\times \R. 
	\]
	Thus the limit function pair $\left(v^{\infty}_{\gamma}(x,y,t),w^{\infty}_{\gamma}(x,t)\right)$ satisfies 	
	\begin{equation}\label{sys319}
		\left\{
			\begin{aligned}
				&\p_t v-d\D v+c\p_x v=\big[M_1-1-(1+M_2)v\big]v,&& \quad(x,y)\in\O_{0},t\in\R,&\\
				&\p_t w-D \p_{xx}w+c\p_x w=-\m w+\n v(x,0,t),&&\quad x\in\R,t\in\R,&\\
				& -d\p_y v(x,0,t)=\m w(x,t)-\n v(x,0,t),&&\quad x\in\R,t\in\R.&
			\end{aligned}
		\right.
	\end{equation}
	Note that $(v_{\gamma,c}(x,y,t),w_{\gamma,c}(x,t))\geq \g\left(\underline{\phi}_{\frac{\pi}{2},c}(x,y,t),\underline{\psi}_{\frac{\pi}{2},c}(x,t)\right)$ for $(x,y)\in\ov{\O}_{0},t\geq T_{\frac{\pi}{2},c}$, we have
	\[
		\left(v^{\infty}_{\gamma}(x,y,t),w^{\infty}_{\gamma}(x,t)\right)\geq \g\left(\tilde{\phi}_{\frac{\pi}{2},c}(x,y),\tilde{\psi}_{\frac{\pi}{2},c}(x)\right),\qquad \forall (x,y)\in\ov{\O}_{0},t\in\R.
	\]
	Take $t_0\in\R$ and let $(\un{v},\un{w})$ be the solution of  system \eqref{sys319} on $\ov{\O}_{0}\times (t_0,+\infty)$ with initial data $(\un{v},\un{w})|_{t=t_0}=\gamma\left(\un{\phi}_{\frac{\pi}{2},c}(x,y),\un{\psi}_{\frac{\pi}{2},c}(x)\right)$. Then, we see that
	\[
		\left(v^{\infty}_{\g}(x,y,t),w^{\infty}_{\g}(x,t)\right)\geq (\un{v}(x,y,t+t_0),\un{w}(x,t+t_0)),\qquad \forall \left(x,y\right)\in\ov{\O}_{0},t\geq 0,
	\]
	$(\un{v},\un{w})$ is strictly increasing in $t$, and $(\un{v},\un{w})$ converges to  a solution $\left(\un{V},\un{W}\right)$ of the system
	\[
		\left\{
			\begin{aligned}
				&-d\D v+c\p_x v=\big[M_1-1-(1+M_2)v\big]v,&& \quad(x,y)\in\O_{0},&\\
				&-D \p_{xx}w+c\p_x w=-\m w+\n v(x,0),&&\quad x\in\R,&\\
				& -d\p_y v(x,0)=\m w(x)-\n v(x,0),&&\quad x\in\R.&
			\end{aligned}
		\right.
	\]
	Following the same proof process as in \cite[Theorem 1.1]{berestycki+2013}, we have $\left(\un{V},\un{W}\right)\equiv \frac{M_1-1}{M_2+1}\left(1,\frac{\n}{\m}\right)$. 
	Therefore,
	\[
		\left(v^{\infty}_{\g}(x,y,0),w^{\infty}_{\g}(x,0)\right) 
			 \geq \lim_{t_0 \to +\infty}(\un{v}(x,y,t_0),\un{w}(x,t_0))
			 = \dfrac{M_1-1}{M_2+1}\left(1,\frac{\n}{\m}\right)
\]
	locally uniformly in $(x,y)\in\ov{\O}_{0}$. On the other hand, it follows from \eqref{www1} that $v^{\infty}_{\g}(\bar{x},\bar{y},0)\leq \dfrac{M_1-1}{M_2+1}-\varepsilon_0$. There is a contradiction. This completes the proof.
	\end{proof}	

\begin{theorem}\label{thm64}
	Suppose that {\rm \ref{a1}} holds. Let $(v,w)$ be the solution of \eqref{62}. Then for any
	$0\leq c<c(\frac{\pi}{2})$ 
	and $l>0$, we have
	\[
		\liminf_{t\to\infty}\inf_{\substack{|x|\leq c(t-T_0)\\0\leq y<l}}(v(x,y,t),w(x,t))\geq\frac{M_1-1}{M_2+1}\left(1,\frac{\n}{\m}\right).
	\]
\end{theorem}
\begin{proof}
	Fix $0\leq c<c(\frac{\pi}{2})$ and take $r\in\left(c,c(\frac{\pi}{2})\right)$ be the value such that \eqref{62} admits generalized subsolutions $\left(\un{\phi}_{\frac{\pi}{2},r}(x,y,t),\un{\psi}_{\frac{\pi}{2},r}(x,t)\right)$ and $\left(\un{\phi}_{-\frac{\pi}{2},r}(x,y,t),\un{\psi}_{-\frac{\pi}{2},r}(x,t)\right)$ on $\ov{\O}_0\times \left[T_{1},+\infty\right)$, where $T_1:=T_{\frac{\pi}{2},r}$.
	For any $0<\g\leq 1$, let \(\left(v_{\gamma},w_{\gamma}\right)\) and \(\left(\tilde{v}_{\gamma},\tilde{w}_{\gamma}\right)\) be solutions of \eqref{62} on \(\overline{\Omega}_0\times \left[T_{1},+\infty\right)\), with initial data
\(\gamma\left(\underline{\phi}_{\frac{\pi}{2},r}(x,y,T_{1}),\underline{\psi}_{\frac{\pi}{2},r}(x,T_{1})\right)\)
and
\(\gamma\left(\underline{\phi}_{-\frac{\pi}{2},r}(x,y,T_{1}),\underline{\psi}_{-\frac{\pi}{2},r}(x,T_{1})\right)\), respectively. Then

Choose $0<\gamma_1\leq 1$ sufficiently small such that
	\[
		\begin{aligned}
			&\left(v_{\gamma}(-x+r t,y,t),w_{\gamma}(-x+r t,t)\right)\big|_{t=T_{1}}\\
			=&\gamma\left(\un{\phi}_{\frac{\pi}{2},r}(-x+r T_{1},y,T_{1}),\un{\psi}_{\frac{\pi}{2},r}(-x+r T_{1},T_{1})\right)\\
			\leq&\left(\min\left\{ v(x,y,T_{1}+1),\frac{M_1-1}{M_2+1} \right\},\min\left\{w(x,T_{1}+1),\frac{\nu(M_1-1)}{\mu(M_2+1)}\right\}\right), \quad \forall\, 0<\gamma\leq \gamma_1.
		\end{aligned}
	\]
	Then by Lemma \ref{lem61}, we have $v_{\gamma_1}(-x+r t,y,t)\leq \frac{M_1-1}{M_2+1}<\frac{M_1}{M_2}$ in $\ov{\O}_{0}\times [T_{1},+\infty)$, which implies that for any $(x,y)\in\ov{\O}_{0}$ and $t>T_{1}+1$, $v_{\gamma_1}(x+r(t-1),y,t-1)$ fulfills
	\[
		\big[-1+(M_1-M_2 v)\chi(-x+rt,y,t)-v\big]v\geq \big[-1+(M_1-M_2 v)\chi(-x+r(t-1),y,t-1)-v\big]v
	\]
	since $\chi(x,y,\cdot)$ is nondecreasing in $(T_0,+\infty)\supset \left(T_{1},+\infty\right)$.
	Hence using the comparison principle (namely, Lemma \ref{lem33}), we have
	\begin{equation}\label{318}
		(v(-x+rt,y,t),w(-x+rt,t))\geq \left(v_{\gamma_1}(-x+r(t-1),y,t-1),w_{\gamma_1}(-x+r(t-1),t-1)\right)
	\end{equation}
	for all $t\geq T_{1}+1$. On the other hand, from \eqref{316}, one can find $T>T_{1}$ large enough while $\g_2>0$ small enough such that
	\begin{equation}\label{325}
			\left(v_{\g_1}(-x+rt,y,t),w_{\g_1}(-x+rt,t)\right)
			\geq \gamma_2\left(\left\|\un{\phi}_{-\frac{\pi}{2},r}(\cdot,\cdot,T_{1})\right\|_{L^{\infty}(\Omega_0)},\left\|\un{\psi}_{-\frac{\pi}{2},r}(\cdot,T_{1})\right\|_{L^{\infty}(\mathbb{R})}\right)
	\end{equation}
	for all $t\geq T$. Hence,
	fix $\widetilde{T}>\max\left\{ 2T,T_{1} \right\}$ and $r\left(2T-T_{1}\right)\leq \xi\leq r \widetilde{T}$, there is $\tau\in \left[T,\frac{\widetilde{T}+T_{1}}{2}\right]$ such that $\xi=r\left(2\tau-T_{1}\right)$. Using \eqref{318} and \eqref{325}, there holds
	\[
		\begin{aligned}
		 	&\left.\left(v(-x+rt,y,t+T_{1}),w(-x+rt,t+T_{1})\right)\right|_{t=\tau}\\
			\geq & \left(v_{\gamma_1}(-x-r T_{1}+r(\tau+T_{1}-1),y,\tau+T_{1}-1),w_{\g_1}(-x-rT_{1}+r(\tau+T_{1}-1),\tau+T_{1}-1)\right)\\
			\geq & \gamma_2\left(\un{\phi}_{-\frac{\pi}{2},r}(x-2\tau,y,T_{1}),\un{\psi}_{-\frac{\pi}{2},r}(x-2\tau,T_{1})\right)\\
			= & \left.\left(\tilde{v}_{\gamma_2}(x-r(t+\tau),y,t-\tau+T_{1}),\tilde{w}_{\gamma_2}(x-r(t+\tau),t-\tau+T_{1})\right)\right|_{t=\tau}.
		\end{aligned}
	\]

	Since $\psi(\cdot)$ is nonincreasing in $\mathbb{R}$, there holds
\[
  \begin{aligned}
    \chi(-x+r(t+\tau),y,t-\tau+T_{1}) 
    = & \psi\left(|(-x+r(t+\tau),y)|-\vs_{*}(t-\tau+T_{1})+m\right)\\
    \leq & \psi\left(|(-x+rt,y)|-\vs_{*}(t+T_{1})+m+\left(\vs_{*}-r\right)\tau\right)\\
    \leq & \psi\left(|(-x+rt,y)|-\vs_{*}(t+T_{1})+m\right)\\
    = & \chi\left(-x+rt,y,t+T_{1}\right).
  \end{aligned}
\]
In addition, $(0,0)\leq \left(\tilde{v}_{\gamma_2}(x,y,t),\tilde{w}_{\gamma_2}(x,t)\right)\leq \frac{M_1-1}{M_2+1}\left(1,\frac{\nu}{\mu}\right)$.
Therefore,
\[
   \left(\tilde{v}(x,y,t),\tilde{w}(x,t)\right):=\left(\tilde{v}_{\gamma_2}(x-r(t+\tau),y,t-\tau+T_{1}),\tilde{w}_{\gamma_2}(x-r(t+\tau),t-\tau+T_{1})\right)
\]
satisfies
  \begin{align*}
    & \partial_t \tilde{v}-d\Delta\tilde{v} +r\partial_{x}\tilde{v}\\
    = & \left[-1+(M_1-M_2\tilde{v})\chi(x-r(t+\tau),y,t-\tau+T_{1})-\tilde{v}\right]\tilde{v}\\
    \leq & \left[-1+(M_1-M_2\tilde{v})\chi(-x+rt,y,t+T_{1})-\tilde{v}\right]\tilde{v},\qquad \forall (x,y)\in\O_0,t>\tau,
  \end{align*}
  and
  \[
  	\begin{aligned}
		& \partial_t\tilde{w}-D\partial_{xx}\tilde{w}+r\tilde{w}= -\mu \tilde{w}+\nu\tilde{v}(x,0,t),\qquad && \forall x\in\mathbb{R},t>\tau,&\\
  		& -d\partial_y\tilde{v}(x,0,t)=\mu \tilde{w}-\nu\tilde{v}(x,0,t), &&\forall x\in\mathbb{R},t>\tau.&
	\end{aligned}
  \]
Observe that $\left(v(-x+rt,y,t+T_{1}),w(-x+rt,t+T_{1})\right)$ fulfills
  \[
	\left\{\begin{aligned}
		& \partial_t v-d\Delta v+r\partial_x v= \left[-1+(M_1-M_2 v)\chi(x,y,t+T_{1})-v\right]v,\qquad && \forall (x,y)\in\O_0,t>\tau,\\
		& \partial_t w-D\partial_{xx}w+r\partial_x w= -\mu w+\nu v(x,0,t),\qquad && \forall x\in\mathbb{R},t>\tau,&\\
  		& -d\partial_yv(x,0,t)=\mu w-\nu v(x,0,t), &&\forall x\in\mathbb{R},t>\tau.&
	\end{aligned}\right.
  \]
Therefore, the comparison principle implies that
	\[
		\left(v(-x+rt,y,t+T_{1}),w(-x+rt,t+T_{1})\right)\geq \left( \tilde{v}(x,y,t),\tilde{w}(x,t)\right),\qquad \forall (x,y)\in\ov{\O}_0, t\geq \tau.
	\]
	In particular, for any $y\geq 0$, one has
	\[
		\begin{aligned}
			& \left(v\left(-\xi+r\widetilde{T},y,\widetilde{T}+T_{1}\right),w\left(-\xi+r\widetilde{T},\widetilde{T}+T_{1}\right)\right)\\
			\geq &  \left( \tilde{v}(\xi,y,\widetilde{T}),\tilde{w}(\xi,\widetilde{T})\right)\\
			= & \left(\tilde{v}_{\gamma_2}\left(-r\lambda(\widetilde{T},\tau),y,\lambda(\widetilde{T},\tau)\right),\tilde{w}_{\gamma_2}\left(-r\lambda(\widetilde{T},\tau),\lambda(\widetilde{T},\tau)\right)\right),
		\end{aligned}
	\]
	where $\lambda\left(\widetilde{T},\tau\right):=\widetilde{T}-\tau+T_{1}$. Fix $l>0$. Then, combined with \eqref{316}, this inequality ensures that
	\[
	\begin{aligned}
		& \liminf_{t\to\infty}\inf_{\substack{0\leq x\leq c(t-T_0)\\0\leq y<l}}(v(x,y,t),w(x,t))\\
		\geq & \liminf_{\widetilde{T} \to \infty}\inf_{\substack{r\left(2T-T_{1}\right)\leq \xi\leq r \widetilde{T}\\0\leq y<l}}\left(v(-\xi+r\widetilde{T},y,\widetilde{T}+T_{1}),w(-\xi+r\widetilde{T},\widetilde{T}+T_{1})\right)\\
		\geq & \liminf_{\widetilde{T} \to \infty}\inf_{\substack{T\leq\tau\leq \left(\widetilde{T}+T_{1}\right)/2\\0\leq y<l}}\left(\tilde{v}_{\gamma_2}\left(-r\lambda\left(\widetilde{T},\tau\right),y,\lambda\left(\widetilde{T},\tau\right)\right),\tilde{w}_{\gamma_2}\left(-r\lambda\left(\widetilde{T},\tau\right),\lambda\left(\widetilde{T},\tau\right)\right)\right)\\
		\geq & \frac{M_1-1}{M_2+1}\left(1,\frac{\n}{\m}\right),
	 \end{aligned}
	\]
	By handling the negative values of $x$ via reflection with respect to $\left\{ 0 \right\}\times \mathbb{R}$, and then using the arbitrariness of $c\in \left[0,c(\frac{\pi}{2})\right)$, we obtain the desired result. The proof is completed.
\end{proof}

\vspace{1.5ex}
It follows from Theorems \ref{thm63} and \ref{thm64} that for any $0\leq c<c(\frac{\pi}{2})$ and $l>0$, there holds
\begin{equation}\label{65}
	\lim_{t\to\infty}\sup_{\substack{|x|\leq c(t-T_0)\\0\leq y<l}}\left|(v(x,y,t),w(x,t))-\frac{M_1-1}{M_2+1}\left(1,\frac{\n}{\m}\right)\right|=0.
\end{equation}
\begin{theorem}\label{thm65}
	Suppose that {\rm \ref{a1}} holds. Let $(v,w)$ be the solution of \eqref{62}. Then for any $\vt\in\left[-\frac{\pi}{2},\frac{\pi}{2}\right]$, the spreading speed of \eqref{62} along $(\sin \vt, \cos\vt)$ is not smaller that $c(\vt)$, namely, for any $\e>0$, there holds
	\[
		\lim_{t \to \infty}\sup_{\substack{(x,y)\in\ov{\O}_{0}\\{\rm dist}\left(\frac{1}{t}(x,y),\ov{\O}_{0}\setminus \mathcal{W}\right)>\e}}\left|v(x,y,t)-\frac{M_1-1}{M_2+1}\right|=0.
	\]
 Here $\mathcal{W}:=\left\{r(\sin\vt,\cos\vt): \vt\in\left[-\frac{\pi}{2},\frac{\pi}{2}\right], 0\leq r \leq c(\vt) \right\}$.
\end{theorem}
\begin{proof}
We divide the proof into two steps.

 {\bf Step 1.} We claim: {\it For any $\e\in\left(0,2\sqrt{d(M_1-1)}\right), \vt\in\left[-\frac{\pi}{2},\frac{\pi}{2}\right]$, there exist a point $\left(\tilde{x},\tilde{y}\right)\in\ov{\O}_{0}$ and an open bounded set $\mathcal{A}$ in the relative topology of $\ov{\O}_{0}$ such that
\[
	\mathcal{A}\supset  \left\{ r(\sin\vt,\cos\vt):0\leq r \leq c(\vt)-\e\right\}, \qquad \inf_{\substack{t\geq T_0+1\\ (x,y)\in t\mathcal{A}}}v(x+\tilde{x},y+\tilde{y},t)=:h_{\vt,\e}>0.
\]
}

Fix $\vt\in\left[-\frac{\pi}{2},\frac{\pi}{2}\right]$, $\e\in(0,c(\vt))$. From Lemma \ref{a2}, one can choose $c\in (c(\vt)-\e,c(\vt))$ be a constant such that \eqref{62} admits a generalized subsolution $\left(\un{\phi}_{\vt,c}(x,y,t),\, \un{\psi}_{\vt,c}(x,t)\right)$. Let $T_{\vt,c}\geq T_0$ be the constant defined in Lemma \ref{a2}. For the case $\vt\not=\pm\frac{\pi}{2}$, define $\d:=\frac{c-c(\vt)+\e}{2c}\in\left(0,\frac{1}{2}\right)$. Up to multiplying $\left(\un{\phi}_{\vt,c},\un{\psi}_{\vt,c}\right)$ by a small constant $\kappa>0$, we can assume that
\[
\left(\sup_{\ov{\O}_{0}}\un{\phi}_{\vt,c}(\cdot,\cdot,T_{\vt,c}),\, \sup_{\mathbb{R}}\un{\psi}_{\vt,c}(\cdot,T_{\vt,c})\right)<\frac{M_1-1}{M_2+1}\left(1,\frac{\n}{\m}\right).\]
Meanwhile, from \eqref{65}, we have that there exists $\tau\geq \max \left\{\frac{T_{\vt,c}+1}{\d},\frac{2c\left(\frac{\pi}{2}\right)T_{\vt,c}}{\e\d}\right\}$ large enough such that for all $|c^{\prime}|\leq c\left(\vt\right)-\frac{\e}{2}$ and  $\lambda\in(\delta,1]$, there holds
\[
	\left(v(x+c^{\prime}\lambda t,y,\lambda t),w(x+c^{\prime}\lambda t,\lambda t)\right)>\left(\un{\phi}_{\vt,c}(x,y,T_{\vt,c}),\un{\psi}_{\vt,c}(x,T_{\vt,c})\right),\qquad \forall (x,y)\in\ov{\O}_{0}, t\geq \tau.
\]
Recall that $\chi(x,y,t)=\psi\left(|(x,y)|-\vs_{*}t+m\right)$, and $\psi(\cdot)$ is nonincreasing. We have $\chi(x,y,s+T_{\vt,c})\leq \chi\left(x+c^{\prime}\lambda t,y,\lambda t+s\right)$ for all $t\geq \tau, s\geq 0$ by noticing that
\[
	\begin{aligned}
		&\left|(x+c^{\prime}\lambda t,y)\right|-\vs_{*}\left(\lambda t+s\right)+m \\
		\leq & |(x,y)|-\vs_{*}(s+T_{\vt,c})+m+|c^{\prime}|\lambda t-\vs_{*}(\lambda t-T_{\vt,c})\\
		\leq & |(x,y)|-\vs_{*}(s+T_{\vt,c})+m,\qquad \forall\, t\geq \tau,\, s\geq 0.
	\end{aligned}
\]
In addition, Lemma \ref{lem62} gives that
\[
	\left(\un{\phi}_{\vt,c}(x,y,s+T_{\vt,c}),\un{\psi}_{\vt,c}(x,s+T_{\vt,c})\right)\leq \frac{M_1-1}{M_2+1}\left(1,\frac{\nu}{\mu}\right)< \frac{M_1}{M_2}\left(1,\frac{\nu}{\mu}\right), \quad \forall (x,y)\in\ov{\O}_{0}, s\geq 0.
\]
Therefore,
\[
	\begin{aligned}
		& \un{\phi}_{\vt,c}(x,y,s+T_{\vt,c})\left[ -1+(M_1-M_2\un{\phi}_{\vt,c}(x,y,s+T_{\vt,c}))\chi(x,y,s+T_{\vt,c})-\un{\phi}_{\vt,c}(x,y,s+T_{\vt,c}) \right]\\
		\leq & \un{\phi}_{\vt,c}(x,y,s+T_{\vt,c})\left[ -1+(M_1-M_2\un{\phi}_{\vt,c}(x,y,s+T_{\vt,c}))\chi(x+c^{\prime}\lambda t,y,\lambda t+s)-\un{\phi}_{\vt,c}(x,y,s+T_{\vt,c}) \right].
	\end{aligned}
\]
Applying the generalized comparison principle (that is Lemma \ref{lem62}) again, we have that for any $t\geq \tau$, there holds
\[
	v(x+c^{\prime}\lambda t,y,\lambda t+s)\geq \un{\phi}_{\vt,c}(x,y,s+T_{\vt,c}),\qquad \forall (x,y)\in\ov{\O}_{0},\ s\geq 0.
\]
It implies that
\begin{equation}\label{321}
	\begin{aligned}
		&v\left(\check{x}+c[(1-\lambda)t+T_{\vt,c}]\sin\vt+c^{\prime}\lambda t,\check{y}+c[(1-\lambda)t+T_{\vt,c}]\cos\vt,t\right)\\
		\geq&\un{\phi}_{\vt,c}\left(\check{x}+c[(1-\lambda)t+T_{\vt,c}]\sin\vt,\check{y}+c[(1-\lambda)t+T_{\vt,c}]\cos\vt,(1-\lambda)t+T_{\vt,c}\right)\\
		=&\un{\phi}_{\vt,c}(\check{x},\check{y},T_{\vt,c})>0,\qquad \forall t\geq \tau
	\end{aligned}
\end{equation}
by taking $(x,y)=(\check{x},\check{y})+c[(1-\lambda)t+T_{\vt,c}](\sin\vt,\cos\vt)\in \ov{\O}_{0}$ and $s=(1-\lambda)t$. Hence, letting
\begin{equation}\label{expansion_subdomain}
	\begin{aligned}
		&\left(\tilde{x},\tilde{y}\right)=\left(\check{x},\check{y}\right)+cT_{\vt,c}(\sin\vt,\cos\vt),\\
		&\mathcal{A} :=\left\{\left(c(1-\lambda)\sin\vt+c^{\prime}\lambda,c(1-\lambda)\cos\vt\right):\d<\lambda\leq 1,\left|c^{\prime}\right|\leq c\left(\vt\right)-\frac{\e}{2}\right\},
	\end{aligned}
\end{equation}
\eqref{321} yields that                            
\begin{equation}\label{322}
	\inf_{\substack{t\geq\tau\\(x,y)\in t\mathcal{A}}}v(x+\tilde{x},y+\tilde{y},t)>0.
\end{equation}
Furthermore, by taking $c^{\prime}=0$ and restricting $\lambda$ to $\left[2\d,1\right]$, one obtains
\[
	\mathcal{A} \supset \left\{\left(c(1-\lambda)\sin\vt,c(1-\lambda)\cos\vt\right):2\d\leq\lambda \leq 1\right\}=\left\{r(\sin\vt,\cos\vt):0\leq r\leq c(\vt)-\e\right\}.
\]

When $\vt=\pm\frac{\pi}{2}$, $\mathcal{A}$ defined in \eqref{expansion_subdomain} satisfies
\[
	\mathcal{A}\supset\left[-c\left(\frac{\pi}{2}\right)+\frac{\e}{2},c\left(\frac{\pi}{2}\right)-\frac{\e}{2}\right]\times \{0\}\supset\left\{r(\sin\vt,\cos\vt):0\leq r\leq c(\vt)-\e\right\}.
\]
Combined with \eqref{65}, it guarantees that \eqref{322} still holds.
We complete the proof of Step 1 by noting that $v>0$ on compact subsets of $\ov{\O}_{0}\times [T_0+1,\tau]$.

\vspace{0.5em}
{\bf Step 2.} We finish the proof of Theorem \ref{thm65}.

Fix $\e_0>0$. Take sequences $\left\{(x_n,y_n)\right\}\subset \ov{\O}_{0}$ and $\{t_n\}$ satisfying
\[
	\lim_{n \to \infty}t_n=+\infty,\qquad \text{and}\qquad  {\rm dist}\left(\frac{1}{t_n}(x_n,y_n),\ov{\O}_{0}\setminus \mathcal{W}\right)>\e_0, \quad \forall n\in\N.
\]
Up to a subsequence, $\displaystyle \lim_{n \to \infty}v(x_n,y_n,t_n)=\eta$ for some $\eta\in\left[0,\frac{M_1-1}{M_2+1}\right]$. In what follows, we devote ourselves to proving $\eta=\frac{M_1-1}{M_2+1}$.

Suppose that $\{y_n\}$ admits a bounded subsequence $\left\{y_{n_k}\right\}$. Since
\[
\begin{aligned}
	\e_0&\leq \liminf_{k \to \infty}{\rm dist}\left(\frac{1}{t_{n_k}}(x_{n_k},y_{n_k}),\ov{\O}_{0}\setminus \mathcal{W}\right)\\
	&\leq \liminf_{k \to \infty}\left[ {\rm dist}\left(\frac{x_{n_k}}{t_{n_k}},\R\setminus\left[-c\left(\tfrac{\pi}{2}\right),c\left(\tfrac{\pi}{2}\right)\right]\right)+\frac{y_{n_k}}{t_{n_k}}\right]\\
	&= \liminf_{k \to \infty}\left[{\rm dist}\left(\frac{x_{n_k}}{t_{n_k}},\R\setminus\left[-c\left(\tfrac{\pi}{2}\right),c\left(\tfrac{\pi}{2}\right)\right]\right)\right],
\end{aligned}
\]
we have $\left|x_{n_k}\right|\leq c\left(\frac{\pi}{2}\right)-\frac{\e_0}{2}t_{n_k}$ for $k$ large enough. It then follows from \eqref{65} that $\eta=\frac{M_1-1}{M_2+1}$.

Suppose that $ y_{n}\to +\infty$ as $n\to\infty$. We denote $\frac{1}{t_n}(x_n,y_n)=r_n(\sin\vt_{n},\cos\vt_{n})$ for some $r_n\in[0,c(\vt_{n})-\e_0]$. Recall that $\vs_0(\cdot)$ is of class $C^1\left(\left[-\frac{\pi}{2},\frac{\pi}{2}\right]\right)$, thus $c(\cdot)$ is Lipschitz continuous. Therefore, without of loss generality, we assume that there exists $\vt_0\in\left[-\frac{\pi}{2},\frac{\pi}{2}\right]$ and $r_0\in[0,c(\vt_0)-\e_0]\subset [0,\vs_{*}-\e_0]$ such that \[
	\lim_{n \to \infty}\vt_n=\vt_0,\quad \lim_{n \to \infty}r_n=r_0.
\]
Let $v_n(x,y,t):=v(x+x_n,y+y_n,t+t_n)$, then up to a subsequence, $v_n$ converges to a function $v_{\infty}\in C^{2+\alpha,1+\frac{\alpha}{2}}(\R^{2}\times \R)$. Notice that
\[
	\lim_{n \to \infty}\frac{1}{t+t_n}(x+x_n,y+y_n)=r_0(\sin\vt_0,\cos\vt_0) \quad \text{locally uniformly in } (x,y,t)\in\mathbb{R}^{3}.
\]
For any $(x,y,t)\in\mathbb{R}^{3}$, there is $N_1>0$ large enough such that
\[
		 (x+x_n,y+y_n)\cdot(\sin\vt,\cos\vt)
		\leq  \left(r_0+\frac{\e}{2}\right)(t+t_n)
		\leq \vs_{*}(t+t_n)-m,\quad \forall n\geq N_1, \vt\in\left[-\frac{\pi}{2},\frac{\pi}{2}\right].
\]
It gives that $v_{\infty}$ is a solution of the following equation
\[
	\p_t v_{\infty}-d\D v_{\infty}=\big[-1+M_1-(M_2 +1)v_{\infty}\big]v_{\infty}, \quad \forall(x,y,t)\in\R^{3}.
\]
On the other hand, we see that
\[
	\lim_{n \to \infty}\frac{1}{t+t_n}(x+x_n-\tilde{x},y+y_n-\tilde{y})=r_0(\sin\vt_0,\cos\vt_0)\in\mathcal{A}\quad \text{locally uniformly in }(x,y,t)\in\mathbb{R}^{3},
\]
and for any $(x,y,t)\in\mathbb{R}^{3}$, there is $N_2=N(x,y,t)>0$ large enough such that $y+y_n-\tilde{y}\geq 0, t+t_n\geq T_0+1$ for $n\geq N_2$. Here we would like to notice that $\mathcal{A}$ and $(\tilde{x},\tilde{y})$ were defined in Step 1 for $\e_0$ and $\vartheta_0$. Thus we have $N_3\geq N_2$ such that
\[
	(x+x_n-\tilde{x},y+y_n-\tilde{y})\in(t+t_n)\mathcal{A},\quad\forall n\geq N_3
\]
because $\mathcal{A}$ is open in $\ov{\O}_{0}$. Hence
\[
	v_n(x,y,t)=v(x+x_n,y+y_n,t+t_n)\geq \inf_{\substack{\tau\geq T_0+1\\ (\xi,\zeta)\in \tau\mathcal{A}}}v(\xi+\tilde{x},\zeta+\tilde{y},\tau)=h_{\vt_0,\e_0}>0.
\]
It infers that $v_{\infty}\geq h_{\vt_0,\e_0}$ in $(x,y,t)\in\R^3$. By comparison with the solution of
\[
	\left\{\begin{aligned}
		& \un{v}^{\prime}=\un{v}\left[M_1-1-(M_2+1)\un{v}\right],\quad t>0,\\
		&\un{v}(0)=h_{\vt_0,\e_0},
	\end{aligned}\right.
\]
we have $v_{\infty}(x,y,t+\tau)\geq\un{v}(t)$ for all $(x,y)\in\mathbb{R}^{2},t\geq 0,\tau\in\mathbb{R}$. Thus
\[
	\eta=v_{\infty}(0,0,0)\geq \lim_{\tau \to \infty}\un{v}(\tau)=\frac{M_1-1}{M_2+1}.
\]
The proof of this theorem is therefore completed.
\end{proof}

\subsection{Upper Bounds of Spreading Speeds}
Under the standing assumption \ref{hy-initial} in this subsection,  we prove that (i) $c^{*}=2\sqrt{d_1}$ is an upper bound for the spreading speed of the prey (i.e., \eqref{u die eq}), and (ii) for any $\vt\in\left[-\frac{\pi}{2},\frac{\pi}{2}\right]$, $\min\left\{c^{*}, \vr^*_{b-1}(\vt)\right\}$ is an upper bound for the predator in the direction $\xi=(\sin\vt,\cos\vt)$ (i.e., \eqref{v die} and \eqref{w die}). Here, we recall that $\vr^*_{b-1}(\vt)$ is defined in Lemma \ref{berestycki+2016 thm2.1} with $d=d_2$ and $\d=b-1$. In the remainder of this paper, unless otherwise stated, we denote by $(u,v,w)$ the solution of \eqref{prey-predator-road} with initial value $\left(u_{0},v_{0},w_{0}\right)$.
We first show that, under Hypothesis \ref{hy-initial}, the prey cannot spread faster than $c^{*}$, that is \eqref{u die eq}.

\vspace{0.5em}
\begin{proof}[Proof of Theorem \ref{main} \ref{u die}]
	Assume $c>c^*$. Let $\ov{u}(x,y,t)$ be the solution of \eqref{u-sup-1}
	with the same initial value $u_{0}$ as given in Theorem \ref{main}. Then $u$ is a subsolution of \eqref{u-sup-1}, whose spreading speed is $c^{*}$ (see \cite{aronson+1978}). Hence it follows from the comparison principle that
	\[
		\li_{ t \to \infty }\sup_{\substack{(x,y)\in \ov{\O}_0\\ |(x,y)|\ge ct}} u(x,y,t)\le \li_{ t \to \infty }\sup_{\substack{(x,y)\in \ov{\O}_0\\|(x,y)|\ge ct}} \ov{u}(x,y,t)=0.
	\]
	This completes the proof.
	\end{proof}
	\vspace{1.5ex}

Then we demonstrate that for any $\vt\in[-\frac{\pi}{2},\frac{\pi}{2}]$, the spreading speed of predators along the direction $(\sin\vt,\cos\vt)$ is not bigger than $\min\{c^{*},\vr_{b-1}^*(\vt)\}$, namely, we prove \eqref{v die} and \eqref{w die}.

\vspace{0.5em}
\begin{proof}[Proofs of \eqref{v die} and \eqref{w die}]
We first show that predators cannot spread faster than the population modeled by \eqref{predator sys} with $d=d_2$ and $\d=b-1$, namely
\begin{equation}\label{39}
	\left\{
		\begin{aligned}
			&\partial_{t}\ov{v}-d_2\Delta \ov{v}=v(b-1-\ov{v}),&&\qquad(x,y)\in \O_0,t>0,&\\
			&\partial_{t}\ov{w}-D\ov{w}_{xx}=-\mu \ov{w}+\n \ov{v}(x,0,t),&&\qquad x\in \mathbb{R},t>0,& \\
			&-d_2\partial_{y}\ov{v}(x,0,t)=\mu \ov{w}(x,t)- \n \ov{v}(x,0,t),&&\qquad x\in \mathbb{R},t>0.&
			\end{aligned}
	\right.
\end{equation}
 Let $(\ov{v}(x,y,t),\, \ov{w}(x,t))$ be the solution of \eqref{39}
corresponding to the initial value $\ov{v}(x,y,0)=v_0(x,y),\, \ov{w}(x,0)=w_0(x)$ for $(x,y)\in \ov{\O}_0$.
Notice that $(v,w)$ is a subsolution of \eqref{39} due to the fact $u\leq 1$. Then by Lemma \ref{berestycki+2016 thm2.1}, Remark \ref{remark26}, and the comparison principle for \eqref{39} (that is \cite[Proposition 3.2]{berestycki+2013}), we have
	\begin{equation}\label{v small best}
		\li_{ t \to \infty }\sup_{\substack{(x,y)\in\ov{\O}_0\\{\rm dist}\left(\frac{1}{t}(x,y),\mathcal{W}_{b-1}\right)>\e}}v(x,y,t)\leq \li_{ t \to \infty }\sup_{\substack{(x,y)\in\ov{\O}_0\\{\rm dist}\left(\frac{1}{t}(x,y),\mathcal{W}_{b-1}\right)>\e}} \ov{v}(x,y,t)=0, \qquad \forall\e>0
	\end{equation}
	and
	\begin{equation}\label{w small best}
	\li_{ t \to \infty} \sup_{|x|\geq ct}w(x,t)\\
		\leq \li_{ t \to \infty }\sup_{|x|\geq ct}\ov{w}(x,t)\\
		=0, \qquad \forall c>\vr_{b-1}^{*}\left(\frac{\pi}{2}\right).
	\end{equation}
	
		Next, we show that $v$ and $w$ cannot spread outside the range of the prey. 
	By Lemma \ref{predator die without prey}, for any $\s>0$, there are $T_{\frac{\s}{2}}>0$ and $M_{\frac{\s}{2}}>0$ such that 
\[
	v(x,y,t) \le \frac{\s}{2}+ M_{\frac{\s}{2}}u(x,y,t),\quad
	w(x,t) \le \frac{\s}{2}+ M_{\frac{\s}{2}}u(x,0,t), \qquad \forall (x,y)\in \ov{\O}_0, t\geq T_{\frac{\s}{2}}.
\]
On the other hand, \eqref{u die eq} infers that for any $c>c^{*}$,
	there is $\widetilde{T}_{\s}>0$ such that
	\begin{equation*}
		\sup_{\substack{(x,y)\in \ov{\O}_0,|(x,y)|\ge ct\\ t\geq \widetilde{T}_{\s}}} u(x,y,t)\le \frac{\s}{2M_{\frac{\s}{2}}}.
	\end{equation*}
	Therefore, for any $c>c^{*}$ and $\s>0$, one has
\[
	\begin{aligned}
		& \displaystyle\limsup_{ t \to \infty } \sup_{\substack{(x,y)\in\ov{\O}_0\\|(x,y)|\ge ct}}v(x,y,t)\leq \sup_{\substack{(x,y)\in\ov{\O}_0,|(x,y)|\ge ct\\ t\geq T^{\prime}}}v(x,y,t)\leq \s,\\
		& \displaystyle\limsup_{ t \to \infty } \sup_{|x|\ge ct}w(x,t)\leq \sup_{\substack{|x|\ge ct\\ t\geq T^{\prime}}}w(x,t)\leq \s,
	\end{aligned}
\]
where $T^{\prime}=\max\left\{T_{\frac{\s}{2}},\widetilde{T}_{\s}\right\}$. The arbitrariness of the choice of $\s>0$ ensures that there hold
\begin{equation}\label{311}
	\displaystyle\lim_{ t \to \infty } \sup_{\substack{(x,y)\in\ov{\O}_0\\|(x,y)|\ge ct}}v(x,y,t)= 0\qquad \text{and}\qquad\displaystyle\lim_{ t \to \infty } \sup_{|x|\ge ct}w(x,t)= 0.
\end{equation}
Note that for any $t>0$ and $\varepsilon>0$, there is
\[
\left\{(x,y)\in\overline{\Omega}_0: {\rm dist}\left(\dfrac{1}{t}(x,y),B_{c^*}\right)>\varepsilon\right\}=\left\{(x,y)\in\ov{\O}_0:|(x,y)|>(c^*+\varepsilon)t\right\}.
\]
 As a result, \eqref{311} together with \eqref{v small best} and \eqref{w small best} gives \eqref{v die} and \eqref{w die}. This completes the proof.
\end{proof}

\subsection{Lower Estimates on the Spreading Speeds}\label{lower}

In this subsection, we show that under \ref{hy-initial}, the spreading speed of the prey is exactly $c^{*}$ and that of predators along $(\sin\vt,\cos\vt)$ is $\min\left\{c^{*},\vr_{b-1}^{*}(\vt)\right\}$ for any $\vt\in\left[-\frac{\pi}{2},\frac{\pi}{2}\right]$.
\subsubsection{Persistence of the two species}

We first prove the persistence of the prey.
\begin{theorem} \label{persistence}
	 For any $0\leq c<c^{*}$, there exists $\eta=\eta(c)>0$ such that for any $(u_0,v_0,w_0)$ satisfying {\rm \ref{hy-initial}}, we have
	 \[
		\lii_{ t \to \infty } \inf_{\substack{(x,y)\in\ov{\O}_0\\|(x,y)|\le ct} } u(x,y,t) \geq \eta.
	\]
\end{theorem}
\begin{proof}
	Fix $c<c^{*}$ and choose $0<\delta< \frac{1}{a}$ such that $c<2\sqrt{d_1(1-a\d)}$. It follows from \eqref{v die without u} that there is $T_{\delta}>0$ and $M_{\delta}>0$ such that $v(x,y,t+T_{\d})\leq \delta+M_{\delta} u(x,y,t+T_{\d})$ for all $(x,y)\in\ov{\O}_0$ and $t\geq 0$. Thus $u(x,y,t+T_{\d})$ is a supersolution of the following system:
	\begin{equation}\label{lemma41-1}
\left\{
\begin{aligned}
	&\partial_{t} \underline{u}-d_1 \Delta \underline{u}=\left[ 1-a\delta -(1+aM_{\delta})\underline{u}\right]\underline{u},\qquad &&(x,y)\in \O_0, t> 0,&\\
	&\p_y\un{u}(x,0,t)=0, &&x\in\R,t>0.
	\end{aligned}
\right.
\end{equation}
 On the other hand, because $u_{0}$ is nontrivial, one has $u(x,y,t)>0$ for any $(x,y)\in\ov{\O}_0$ and $t>0$. Thus we can supplement \eqref{lemma41-1} with the initial data $\un{u}(x,y,0)$ satisfying $\un{u}(\cdot,\cdot,0)\not\equiv 0$ on $\ov{\O}_0$ and
\[
	0\leq\un{u}(x,y,0)\leq u(x,y,T_{\d}),\qquad \forall (x,y)\in \ov{\O}_0.
\]
Then the comparison principle shows that
$u(x,y,t+T_{\d})\ge \underline{u}(x,y,t)$ for $(x,y)\in \ov{\O}_0$ and $t\ge 0$.     
Therefore, one has
\[
\begin{aligned}
	&\lii_{ t \to \infty } \inf_{\substack{(x,y)\in\ov{\O}_0\\|(x,y)|\le ct}} u(x,y,t)\\
	\ge &\lii_{ t \to \infty } \inf_{\substack{(x,y)\in\ov{\O}_0\\|(x,y)|\le ct } }\underline{u}(x,y,t-T_{\delta})\\
	\geq &\li_{ t \to \infty } \inf_{\substack{(x,y)\in\ov{\O}_0\\|(x,y)|< 2\sqrt{d_1(1-a\d)}t }} \underline{u}(x,y,t)= \frac{1-a\delta}{1+aM_{\delta}}=:\eta.
\end{aligned}
\]
Here $\eta$ depends only on $c$. This completes the proof.
\end{proof}
\begin{remark}
	Let $(u_l,v_l,w_l)$ be the solution of \eqref{prey-predator-road} with $\O_0$ replaced by $\O_l$, and with the initial condition {\rm\ref{hy-initial}}. By transforming in $y$, {\rm Theorem \ref{persistence}} implies that for any $l\in(-\infty,+\infty)$ and $0<c<c^{*}$, one has
	 \begin{equation}\label{42}
		\lii_{ t \to \infty } \inf_{\substack{(x,y)\in\ov{\O}_l\\|(x,y)|\le ct }} u_l(x,y,t) > 0.
	 \end{equation}
	 When $l=-\infty$, \eqref{42} follows directly from {\rm\cite[Proposition 3.3]{ducrot+2023a}}.
\end{remark}

Now we show \eqref{u best}.

\begin{proof}[Proof of \eqref{u best}]Assume $B_{c^*}\setminus \mathcal{V}\not= \emptyset$ and fix $\e>0$. To prove this convergence result, we proceed by contradiction again. Suppose that there exist $\d>0$, $t_n$ and $(x_{n},y_{n})\in\ov{\O}_0$ with $t_{n}\to \infty$ as $n\to \infty$ and  ${\rm dist}\left(\frac{1}{t_n}(x_n,y_n),\mathcal{V}\cup(\O_0\setminus B_{c^*})\right)>\e$ such that
\begin{equation*}
	u(x_n,y_n,t_{n})\le 1-\delta, \qquad \forall n\in\N.
\end{equation*}
Similar to the proof of Lemma \ref{predator die without prey}, we denote $(u_n(x,y,t),v_{n}(x,y,t),w_{n}(x,t)):=(u(x+x_n,y+y_n,t+t_n),v(x+x_n,y+y_n,t+t_n),w(x+x_n,t+t_n))$ for $(x,y)\in\ov{\O}_{-y_n},t\geq -t_n$.

Suppose that there is a subsequence of $\{y_n\}$ tends to some $z_0\geq 0$ as $n\to\infty$, then there exists a subsequence of $(u_n(x,y,t),\,v_n(x,y,t),\,w_n(x,t))$ (still denoted by $(u_n(x,y,t)$, $v_n(x,y,t)$, $w_n(x,t))$) and a triplet $(\tilde{u}(x,y,t),\, \tilde{v}(x,y,t),\, \tilde{w}(x,t))$ such that
\[
	\lim_{n \to \infty}u_n(x,y,t)=\tilde{u}(x,y,t),\qquad \lim_{n \to \infty}v_n(x,y,t)=\tilde{v}(x,y,t)
\]
in $C^{2,1}_{loc}(\O_{-z_0}\times\R)\cap C^{1}_{loc}\left(\ov{\O}_{-z_0}\times\R\right)$
and
$\lim_{n \to \infty}w_n(x,t)=\tilde{w}(x,t)$
 in $C^{2,1}_{loc}(\R\times\R)$.
In particular, $\left(\tilde u,\tilde v,\tilde w\right)$ is a solution of
 \[
 \left\{
 \begin{aligned}
 &\p_t\tilde{u}-d_1\D\tilde{u}=\tilde{u}(1-\tilde{u}-a\tilde{v}),&& \quad (x,y)\in\O_{-z_0},t\in\R,&\\
 &\p_t\tilde{v}-d_2\D\tilde{v}=\tilde{v}(-1+b\tilde{u}-\tilde{v}),&& \quad (x,y)\in\O_{-z_0},t\in\R,&\\
 &\p_t \tilde{w}-D\p_{xx}\tilde{w}=-\m\tilde{w}(x,t)+\n\tilde{v}(x,-z_0,t),&&\quad (x,t)\in\R^2,&\\
 &\p_y \tilde{u}(x,-z_0,t)=0,&&\quad (x,t)\in\R^2,&\\
 &-d_2\p_y\tilde{v}(x,-z_0,t)=\m\tilde{w}(x,t)-\n\tilde{v}(x,-z_0,t), &&\quad (x,t)\in\R^2.&
 \end{aligned}
 \right.
 \]
Recalling \eqref{v die}, we have
$	\tilde{v}(0,0,0)= \li_{ n \to \infty } v(x_{n},y_{n},t_{n})=0$. 
If $z_0=0$, by an argument similar to that for \eqref{trivial}, we obtain
\begin{equation}\label{339}
	\tilde{v}(x,y,t)\equiv 0,\qquad \forall (x,y)\in \O_{-z_0}, t\in \mathbb{R}.
\end{equation}
If $z_0>0$, \eqref{339} follows immediately from the strong maximum principle.
On the other hand, it is obvious that $(x_n,y_n,t_n)$ satisfies
\[
	{\rm dist}\left(\dfrac{1}{t_n}(x_n,y_n), \Omega_0\setminus B_{c^*}\right)>\varepsilon,
\]
since ${\rm dist}\left(\dfrac{1}{t_n}(x_n,y_n), \mathcal{V}\cup\left(\ov{\Omega}_0\setminus B_{c^*}\right)\right)>\varepsilon$. Thus, we have $\left|\dfrac{1}{t_n}(x_n,y_n)\right|\leq c^*-\varepsilon$. 


\begin{claim}\label{limit u}
	There holds
\[
	\inf_{\substack{(x,y)\in \ov{\O}_{-z_0}\\ t\in \mathbb{R}}}\tilde{u}(x,y,t)>0.
\]
\end{claim}
\vspace{1.5ex}
{\it Proof of Claim \ref{limit u}.}
For any $(x,y)\in \ov{\O}_{-z_0}$, $t\in\R$, there exists $N\in\N$ such that
$$|(x,y)|\le \frac{\e}{2} t_{n}+ \left(c^{*}-\frac{\e}{2}\right)t,\qquad \forall n\geq N,$$
where $\e$ is given as above.
By the choice of $(x_n,y_n)$, we have
\begin{equation*}
		\left|(x+x_{n},y+y_{n})\right|\leq \left|(x,y)\right|+\left|(x_{n},y_{n})\right|
		\leq \frac{\e}{2} t_{n}+ \left(c^{*}-\frac{\e}{2}\right)t+(c^*-\e)t_{n}
		= \left(c^*-\frac{\e}{2}\right)(t+t_{n}),
\end{equation*}
for all $n\geq N$.
Thus using \eqref{42}, we obtain that for any $(x,y)\in \ov{\O}_{-z_0}$ and $t\in \mathbb{R}$, there is
\[
	\tilde{u}(x,y,t)=\li_{ n \to \infty } u(x+x_{n},y+y_{n},t+t_{n})
	\ge \lii_{ \tau \to \infty } \inf_{\substack{(\xi,\zeta)\in \ov{\O}_{0}\\|(\xi,\zeta)|\le \left(c^*-\frac{\e}{2}\right)\tau }} u(\xi,\zeta,\tau)
	=: \eta>0
\]
since $\li_{n\to\infty} t_n=+\infty$.
Therefore,
$\displaystyle \inf_{\substack{(x,y)\in \ov{\O}_{-z_0}\\ t\in \mathbb{R}}}\tilde{u}(x,y,t)\ge \eta>0$.
The proof of this claim is thus completed.

Denote $\underline{u}(t)$ by the solution of
\begin{equation}\label{ode}
	\underline{u}^{\prime}= \underline{u}(t)(1-\underline{u}(t)),\qquad t>0
\end{equation}
with
$\underline{u}(0)=\inf_{\substack{(x,y)\in \ov{\O}_{-z_0}\\ t\in \mathbb{R}}}\tilde{u}(x,y,t):=\tilde{\eta}>0$. 
Notice that $\tilde{u}$ satisfies the following equations
\begin{equation}\label{44}
\begin{cases}
	\partial_{t}\tilde{u}-d_1 \Delta \tilde{u}=\tilde{u}(1-\tilde{u}), \qquad & \forall (x,y)\in \O_{-z_0},t\in\R,\\
	\p_y \tilde u(x,-z_0,t)=0,\qquad & \forall x\in\R,t\in\R.
	\end{cases}
\end{equation}
The comparison principle implies that
$\tilde{u}(x,y,t+\tau)\ge \underline{u}(t)$  for $\tau\in\mathbb{R}$, $t\geq 0$ and $(x,y)\in \overline{\Omega}_{-z_0}$.
Since $\li_{ t \to \infty }\underline{u}(t)=1$, we arrive at
$\tilde{u}(x,y,0)\ge\li_{ \tau \to -\infty } \underline{u}(-\tau)=1$  for  $ (x,y)\in \overline{\Omega}_{-z_0}$.
 However,
$\tilde{u}(0,0,0)=\li_{n\to\infty}u(x_n,y_n,t_{n})\le 1-\delta<1$. 
We get a contradiction and thus
$\li_{n\to\infty}y_n=+\infty$. Then the limit pair $(\tilde{u}(x,y,t),\tilde{v}(x,y,t))$ of $(u_n(x,y,t),v_n(x,y,t))$ fulfills
 \[
 \left\{
 \begin{aligned}
 &\p_t\tilde{u}-d_1\D\tilde{u}=\tilde{u}(1-\tilde{u}-a\tilde{v}),&& \qquad (x,y,t)\in\R^3,&\\
 &\p_t\tilde{v}-d_2\D\tilde{v}=\tilde{v}(-1+b\tilde{u}-\tilde{v}),&& \qquad (x,y,t)\in\R^3.&
 \end{aligned}
 \right.
 \]
The classical strong maximum principle and $\tilde v(0,0,0)=\displaystyle \lim_{n\to\infty}v(x_n,y_n,t_n)=0$ (see \eqref{v die}) yield that $\tilde v(x,y,t)\equiv 0$ for all $(x,y,t)\in \R^3$. Thus $\tilde{u}$ satisfies
\begin{equation}\label{46}
	\begin{cases}
	\partial_{t}\tilde{u}-d_1 \Delta \tilde{u}=\tilde{u}(1-\tilde{u}), \qquad & \forall (x,y,t)\in \mathbb{R}^{3},\\
	\tilde{u}(x,y,0)>0, \qquad & \forall (x,y)\in\R^2.
	\end{cases}
\end{equation}
Applying \eqref{42} with $l=-\infty$ and repeating the proof of Claim \ref{limit u}, we have
	$\inf_{\substack{(x,y)\in\mathbb{R}^{2}\\t\in \mathbb{R}}}\tilde{u}(x,y,t)>0$.
Then by comparing with \eqref{ode}, one can get a contradiction with $\tilde{u}(0,0,0)\leq 1-\d$ again. Therefore, we conclude \eqref{u best}. This completes the proof.
\end{proof}

\vspace{0.5em}
Next, we provide another pointwise comparison between the prey and predators. It indicates that as the predator population becomes negligible, the prey density reaches the environmental carrying capacity.
\begin{lemma}\label{lem43}
Let $c\in[0,c^*)$. For each $\s>0$, there exist $M_{\s}>0$ and $T_{\s}>0$ such that for any $y_0\in\R$ and $(u_0,v_0,w_0)\in C\left(\ov{\O}_{y_0}\right)\times C\left(\ov{\O}_{y_0}\right)\times C(\R)$ satisfying {\rm \ref{hy-initial}}, one has
\begin{equation}\label{1-u}
	1-u(x,y,t)\leq \s+M_{\s}v(x,y,t),\qquad \forall t\geq T_{\s},(x,y)\in\ov{\O}_{y_0}~\text{\rm with}~|(x,y-y_0)|\leq ct.
\end{equation}
Here $(u,v,w)$ is the solution of \eqref{prey-predator-road} with the road $\R\times \left\{ y_0 \right\}$ and the initial data $(u_0,v_0,w_0)$, namely,
\[
	\left\{
		\begin{aligned}
			&\partial_{t}u-d_{1}\Delta u=u(1-u-a v),&&\qquad(x,y)\in \O_{y_0},t>0,&\\
			&\partial_{t}v-d_{2}\Delta v=v(-1+bu-v),&&\qquad(x,y)\in \O_{y_0},t>0,&\\
			&\partial_{t}w-Dw_{xx}=-\mu w+\nu v(x,y_0,t),&&\qquad x\in \mathbb{R},t>0,& \\
			&\p_y u(x,y_0,t)=0,&& \qquad x\in \mathbb{R},t>0,&\\
			&-d_{2}\partial_{y}v(x,y_0,t)=\mu w(x,t)-\nu v(x,y_0,t),&&\qquad x\in \mathbb{R},t>0,&\\
			& u(x,y,0)=u_0(x,y), v(x,y,0)=v_0(x,y),w(x,0)=w_0(x),&& \qquad (x,y)\in\ov{\O}_{y_0}.&
			\end{aligned}
	\right.
\]
\end{lemma}
\begin{proof}
	Fix $c\in[0,c^*)$. Suppose by contradiction that there exist sequences $\{ y_{0,n} \}$, $\{t_n\}$, $\{ u_{0,n}\}\subset C\left(\ov{\O}_{y_0}\right)$, $\{v_{0,n}\}\subset C\left(\ov{\O}_{y_0}\right)$, $\{w_{0,n} \}\subset C(\R)$, $\{(x_n,y_n)\}\subset\ov{\O}_{y_{0,n}}$, and a constant $\s_0>0$ such that
	\[
		(u_{0,n},v_{0,n},w_{0,n})\text{ satisfies } \ref{hy-initial},
		\qquad\li_{n\to\infty}t_n=\infty,\qquad |(x_n,y_n-y_{0,n})|\leq ct_n,\ \  \forall n\in\N,
	\]
	and
	\begin{equation}\label{48}
		1-u^{n}(x_n,y_n,t_n)>\s_0+n v^{n}(x_n,y_n,t_n),\qquad \forall n\in\N.
	\end{equation}
	Here $(u^n,v^n,w^n)$ is the solution of \eqref{prey-predator-road} with the road $\R\times \left\{ y_{0,n} \right\}$ and intial data $(u^n,v^n,w^n)\big|_{t=0}=(u_{0,n},v_{0,n},w_{0,n})$.
	For any $x\in\R,y\geq y_{0,n}-y_n,t\geq -t_n$, define $u_n(x,y,t):=u^{n}(x+x_n,y+y_n,t+t_n),v_n(x,y,t):=v^{n}(x+x_n,y+y_n,t+t_n)$ and $w_n(x,t):=w^{n}(x+x_n,t+t_n)$. Then since $y_{0,n}-y_n\leq 0$, there exists $\left(\tilde{u},\tilde{v},\tilde{w}\right)$ and $z_0\in[0,+\infty]$ such that a subsequence $\left\{ (u_{n_k},v_{n_k},w_{n_k}) \right\}$ fufills
	\[
		\begin{aligned}
			&\lim_{k \to \infty}u_{n_k}=\tilde{u},\quad \lim_{k \to \infty}v_{n_k}=\tilde{v},\quad && {\rm in}~ C_{loc}^{2,1}\left(\O_{-z_0}\times\R\right)~{\rm and~also~in}~C^{1}_{loc}\left(\ov{\O}_{-z_0}\times\R\right)~{\rm if}~z_0\not=+\infty,&\\
			&\lim_{k \to \infty}w_{n_k}=\tilde{w},\quad && {\rm in}~ C_{loc}^{2,1}\left(\R^2\right).&
	\end{aligned}
	\]
Since \eqref{48} implies $\tilde{v}(0,0,0)=0$, we have
\begin{equation}\label{trivial-whole}
	\tilde{v}\equiv 0\qquad {\rm on}~ \ov{\O}_{-z_0}\times \R.
\end{equation}
Indeed, if $z_0<+\infty$, the proof of \eqref{trivial-whole} is the same as \eqref{339}. If $z_0=+\infty$, then $\tilde{v}$ satisfies
\[
	\p_t\tilde{v}-d_2\D\tilde{v}=\tilde{v}\left(-1+b\tilde{u}-\tilde{v}\right)\geq \tilde{v}\left[-1-(b-1)\right] = -b \tilde{v} \qquad {\rm in}~\R^3
\]
since $(0,0)\leq(\tilde{u},\tilde{v})\leq (1,b-1)$. By the strong maximum principle, we get \eqref{trivial-whole}.
As a result, $\tilde{u}(x,y,t)$ satisfies \eqref{44} or \eqref{46}.
On the other hand, applying Theorem \ref{persistence} to
$(u^n(x,y+y_{0,n},t),v^n(x,y+y_{0,n},t),w^n(x,y+y_{0,n},t))$,
we obtain that there exists $\eta>0$ depending on $c$ such that 
$\tilde{u}(x,y,t)\geq \eta>0$ for all $(x,y,t)\in \ov{\O}_{-z_0}\times \R$.
Hence $\tilde{u}\equiv 1$ by the same arguments as those in the proof of \eqref{u best}. It contradicts \eqref{48}. Thus we have \eqref{1-u} and complete the proof.
\end{proof}

\begin{theorem}\label{uniform}
	Assume that {\rm \ref{hy-initial}} holds. Then for any $\e>0$, there is $\varpi>0$ such that
\[
	\lii_{t\to\infty}\inf_{\substack{(x,y)\in\ov{\O}_0\\ {\rm dist}\left(\frac{1}{t}(x,y),\ov{\O}_0\setminus\mathcal{V}\right)>\e}}v(x,y,t)\geq\varpi.
\]
\end{theorem}
\begin{proof}
	Fix $0<\e<\min \left\{ c^{*},2\sqrt{d_2(b-1)}\right\}$ and take $0<\sigma<1-\frac{1}{b}$ sufficiently small such that 
	\[
		\vr_{b(1-\sigma)-1}^*(\vt)\geq \vr_{b-1}^*(\vt)-\frac{\e}{2},\qquad \forall \vt\in\left[-\frac{\pi}{2},\frac{\pi}{2}\right].
	\]
	Using Lemma \ref{lem43}, we can find $M_{\s}$ and $T_{\s}$ such that
$u\geq 1-\left(\s+M_{\s}v\right)$ for $t\geq T_{\s}$ and $(x,y)\in\omega_t$,
 where
\[
	\omega_{t}:=\left\{
		 (x,y)\in \ov{\O}_{0}: (x,y)\cdot(\sin\vt,\cos\vt)\leq (c^{*}-\e/2)t,\quad \forall \vt\in\left[-\frac{\pi}{2},\frac{\pi}{2}\right]
	\right\}.
\]
Furthermore, take $0<m<\sqrt{d_2[b(1-\s)-1]}T_{\sigma}$, and define the set
\[
	\tilde{\omega}_{t}:=\left\{
		 (x,y)\in\o_{t}: (x,y)\cdot(\sin\vt,\cos\vt)\leq (c^{*}-\e/2)t-m,\quad \forall \vt\in\left[-\frac{\pi}{2},\frac{\pi}{2}\right]\right\},
\]
as well as the function $\chi:\O_{0}\times [T_{\s},+\infty)\mapsto [0,1]$,
\[
	\chi(x,y,t)=\psi\left(|(x,y)|-(c^{*}-\e/2)t+m\right).
\]
Here $\psi:\mathbb{R}\mapsto [0,1]$ is the smooth function defined in Section \ref{truncated_sys}. Therefore,
	$u(x,y,t)\geq \max\left\{0,\left[1-\sigma-M_{\sigma}v(x,y,t)\right]\chi(x,y,t)\right\}$
 for any $(x,y)\in\ov{\O}_{0},~t\geq T_{\sigma}$. Hence $(v,w)$ is a supersolution of the following system
 \begin{equation}\label{350}
 	\left\{
		\begin{aligned}
			&\p_t\un{v}-d_2\D \un{v}=\big[-1+b(1-\s-M_{\s}\un{v})\chi-\un{v}\big]\un{v},&& (x,y)\in\O_{0},t>T_{\sigma}, &\\
			&\p_t\un{w}-D \p_{xx}\un{w}=-\m\un{w}+\n\un{v}(x,0,t),&& x\in\R,t>T_{\sigma}, &\\
			& -d_2\p_y\un{v}(x,0,t)=\m\un{w}(x,t)-\n\un{v}(x,0,t),&& x\in\R,t>T_{\sigma},&\\
			&\underline{v}(x,y,T_{\sigma})=v(x,y,T_{\sigma}),&&(x,y)\in\O_{0},&\\
			&\underline{w}(x,T_{\sigma})=w(x,T_{\sigma}),&&x\in\R.&
		\end{aligned}
	\right.
\end{equation}
Notice that $(v,w)$ fulfills
\[
	\left\{
		\begin{aligned}
			&\partial_{t}v-d_{2}\Delta v=v(-1+bu-v)\geq (-1-v)v,&&\qquad(x,y)\in \O_0,t>0,&\\
			&\partial_{t}w-Dw_{xx}=-\mu w+\nu v(x,0,t),&&\qquad x\in \mathbb{R},t>0,& \\
			&\p_y u(x,0,t)=0,&& \qquad x\in \mathbb{R},t>0,&\\
			&-d_{2}\partial_{y}v(x,0,t)=\mu w(x,t)-\nu v(x,0,t),&&\qquad x\in \mathbb{R},t>0.&
			\end{aligned}
	\right.
\]
Hence $(0,0)<(v,w)\leq \left(b-1,\frac{\nu}{\mu}(b-1)\right)$ for any $(x,y,t)\in\ov{\O}_0\times (0,+\infty)$, which ensures that the initial datum $(\underline{v}(x,y,T_{\sigma}),\underline{w}(x,T_{\sigma}))$ of \eqref{350} satisfies \ref{a1}. Therefore, one can apply Theorem \ref{thm65} to \eqref{350} to yield that
	\[
		\lim_{t \to \infty}\sup_{\substack{(x,y)\in\ov{\O}_{0}\\{\rm dist}\left(\frac{1}{t}(x,y),\ov{\O}_{0}\setminus \mathcal{U}\right)>\frac{\e}{2}}}\left|\underline{v}(x,y,t)-\frac{b(1-\sigma)-1}{bM_{\sigma}+1}\right|=0
	\]
 with
 \[
 	\begin{aligned}
		\mathcal{U}&:=\left\{r(\sin\vt,\cos\vt): \vt\in\left[-\frac{\pi}{2},\frac{\pi}{2}\right], 0\leq r \leq \min\left\{c^{*}-\frac{\varepsilon}{2},\vr_{b(1-\sigma)-1}^*(\vt)\right\}\right\}\\
		& \supset \left\{r(\sin\vt,\cos\vt): \vt\in\left[-\frac{\pi}{2},\frac{\pi}{2}\right], 0\leq r \leq \min\left\{c^{*},\vr_{b-1}^*(\vt)\right\}-\frac{\varepsilon}{2}\right\}.
	\end{aligned}
\]

Consequently,
\[
	\liminf_{t \to \infty}\inf_{\substack{(x,y)\in\ov{\O}_0\\ {\rm dist}\left(\frac{1}{t}(x,y),\ov{\O}_0\setminus\mathcal{V}\right)>\e}}v(x,y,t)\geq\frac{b(1-\sigma)-1}{bM_{\sigma}+1}=:\varpi.
\]
We complete the proof of this theorem.
\end{proof}

\vspace{0.5em}
Finally, by comparing $v(x,0,t)$ and $w(x,t)$, we prove the persistence of $w$. 
\begin{lemma}\label{lem313}
	For any $\sigma>0$, there exists $\widetilde{M}_{\sigma}>0$ and $\widetilde{T}_{\sigma}>0$ such that for any $(u_0,v_0,w_0)$ satisfying {\rm \ref{h2}}, we have
	\[
		v(x,0,t)\leq \sigma+\widetilde{M}_{\sigma}w(x,t),\qquad \forall x\in\mathbb{R},t\geq \widetilde{T}_{\sigma}.
	\]
\end{lemma}
\begin{proof} Let $\sigma_0>0$.
	Proceed by contradiction and assume that there exist $\left\{ x_n \right\}$, $\left\{ t_n \right\}$ and $(u_{0,n},v_{0,n},w_{0,n})$ such that
	\[
		\lim_{n \to \infty}t_n=+\infty,\quad(u_{0,n},v_{0,n},w_{0,n}) \text{ satisfies \ref{h2}}
	\] and
	\[
		v^{n}(x_n,0,t_n)\geq \sigma_0 +n w^{n}(x_n,t_n).
	\]
	Then one can extract a subsequence so that
	\[
		\lim_{n \to \infty}u^{n}(x+x_n,y,t+t_n)= \tilde{u}(x,y,t),\quad \lim_{n \to \infty}v^{n}(x+x_n,y,t+t_n)= \tilde{v}(x,y,t)
	\]
	in $C_{loc}^{2,1}\left(\O_0\times \mathbb{R}\right)\cap C_{loc}^{1}\left(\ov{\O}_0\times \mathbb{R}\right)$ and
	$\lim_{n \to \infty}w^{n}(x+x_n,t+t_n)= \tilde{w}(x,t)$
 	in $C_{loc}^{2,1}\left(\mathbb{R}\times \mathbb{R}\right)$.
	Here $\tilde{w}$ satisfies
	\[
		\left\{\begin{aligned}
			&\partial_t\tilde{w}-D\partial_{xx}\tilde{w}=\nu v(x,0,t)-\mu \tilde{w}\geq -\mu \tilde{w}, &&x\in\mathbb{R},t\in\mathbb{R},& \\
			& \tilde{w}\geq 0, &&x\in\mathbb{R},t\in\mathbb{R},& \\
			& \tilde{w}(0,0)= 0.
		\end{aligned}\right.
	\]
	Applying the strong maximum principle infers that $\tilde{w}\equiv 0$ in $\mathbb{R}^{2}$. Therefore, $\tilde{v}(x,0,t)=0$ for any $(x,t)\in\mathbb{R}^{2}$, it contradicts $\tilde{v}(0,0,0)\geq \sigma_0$. Thus we complete the proof.
\end{proof}
\begin{theorem}\label{thm317}
	Assume that {\rm \ref{hy-initial}} holds. Then for any $\e>0$, there is $\widetilde{\varpi}>0$ such that
	\[
		\liminf_{t \to \infty}\inf_{|x|\leq ct}w(x,t)\geq \widetilde{\varpi}, \qquad \forall\, 0\leq c\leq \min\left\{c^{*},\vr^{*}_{b-1}\left(\tfrac{\pi}{2}\right)\right\}-\e.
	\]
\end{theorem}
\begin{proof}
	Fix $\e>0$. From Lemma \ref{uniform}, we have $\varpi>0$ such that
	\[
		\liminf_{t \to \infty} \inf_{|x|\leq ct}v(x,0,t)\geq \varpi,\qquad \forall\, 0\leq c\leq \min\left\{c^{*},\vr^{*}_{b-1}\left(\tfrac{\pi}{2}\right)\right\}-\e.
	\]
	Then, taking $0<\sigma<\varpi$, we obtain
	\[
		\liminf_{t \to \infty}\inf_{|x|\leq ct}w(x,t)\geq\liminf_{t \to \infty} \inf_{|x|\leq ct}\frac{v(x,0,t)-\sigma}{\widetilde{M}_{\sigma}}\geq \frac{1}{\widetilde{M}_{\sigma}}(\varpi-\sigma),
	\]
	where $\widetilde{M}_{\sigma}$ is defined in Lemma \ref{lem313}. The proof is completed by taking $\widetilde{\varpi}:=\frac{\varpi-\sigma}{\widetilde{M}_{\sigma}}$.
\end{proof}

\subsubsection{Convergence of solutions}\label{converge}
By constructing appropriate Lyapunov functionals, Ducrot et al. \cite[Lemma 4.2]{ducrot+2021} and Guo and Shimojo \cite[Theorem 1.1]{guo+2021} established Liouville type theorems, which can be used to explore the asymptotic behavior of  solutions for some reaction-diffusion systems. Motivated by their results, we establish the following lemma to show the asymptotic behavior of $(u,v,w)$ in $\mathcal{V}$, namely \eqref{plane converge} and \eqref{w converges}. The difficulty to prove this lemma comes from that  the components  $v$ and $w$ satisfy the different equations.  

\begin{lemma}\label{lem converge}
	For any entire solution $(u,v,w)$ of \eqref{prey-predator-road} satisfying
	$$
		\varepsilon\leq u,\, v,\, w\leq M\qquad \text{\rm for some positive constants } \varepsilon \text{ \rm and } M,
	$$
	one has $(u(x,y,t),v(x,y,t),w(x,t))\equiv (u^{*},v^{*},w^{*})$ for all $(x,y,t)\in\ov{\O}_0\times \R$.
\end{lemma}

\begin{proof} For $u>0$, $v>0$ and $w>0$, define
\[
	F(u):=\frac{b}{a}u^{*}\left( \frac{u}{u^{*}}-\ln \frac{u}{u^{*}}-1 \right),~ G(v):=v^{*}\left( \frac{v}{v^{*}}-\ln \frac{v}{v^{*}}-1 \right),~ H(w):= w^{*}\left(\frac{w}{w^{*}}-\ln\frac{w}{w^{*}}-1\right).
\]
For any given $R>0$, define
\[
	\vphi_R(z)=e^{-\frac{\sqrt{1+z^2}}{R}},\qquad
	\psi_{R}(z)=
		\begin{cases}
			\cos z,\quad & z\in[0,z_R],\\
			\vphi_R(z),\quad & z\in\left(z_R,+\infty\right),
		\end{cases}
\]
where $z_R\in\left(0,\frac{\pi}{2}\right)$ is the minimal point such that $\c z=\vphi_{R}(z)$ in $z\in\left(0,\frac{\pi}{2}\right)$.
For all $t\in\R$, define 
\[
	\begin{aligned}
	\mathcal{F}_{R}(t):=\int_{0}^{+\infty}\int_{-\infty}^{+\infty}\varphi_R\left( x\right)\psi_{R}\left(\frac{y}{R}\right)\left(F(u)+\frac{G(v)}{R^{2}}\right)dxdy
	+\int_{-\infty}^{+\infty}\varphi_R\left( x\right)\frac{H(w)}{R^{2}}dx.
	\end{aligned}
\]
Then one can check the following properties:
\begin{itemize}
    \item  $F(\cdot)$, $G(\cdot)$ and $H(\cdot)$ are strictly convex functions, and there is $K>0$ such that $0\leq F(\cdot), G(\cdot), H(\cdot)\leq \frac{K}{6 \max \{d_1,d_2, 1\}}$  in $\left[\varepsilon,M\right]$.
	\item $z_R$ is decreasing in $R$ and tends to $0$ as $R$ goes to $+\infty$.
	\item There is $R_0\geq 1$, such that 
	\begin{equation}\label{351}
		\frac{\max\left\{ 4d_1,4d_2,2D,1 \right\}}{R}\leq \sqrt{z_{R}}< 2\e, \qquad \forall\, R\geq R_0.
	\end{equation}
	Thus, it follows that
		\[
		\frac{\max\left\{ 2d_1,2d_2,D \right\}}{R^{2}}\mathcal{F}_{R}(t)\leq \frac{\sqrt{z_{R}}}{2}\mathcal{F}_{R}(t),~ \frac{\sqrt{z_{R}}}{R}\leq z_R, \qquad \forall R\geq R_0,
	\]
	and
	$$
		\begin{aligned}
			& F^{\prime}(u)u(1-u-av)+\frac{G^{\prime}(v)}{R^{2}}v(-1+bu-v)\\
			= & -\frac{b}{a}(u-u^{*})^{2}-\frac{1}{R^{2}}(v-v^{*})^{2}\\
			\leq & -\sqrt{z_{R}}\left(F(u)+\frac{G(v)}{R^{2}}\right), 		\qquad \forall u,v\in[\e,M],R\geq R_0.
		\end{aligned}
	$$
	\item $\vphi_{R}^{\prime\prime}(z)\le \frac{\vphi_{R}(z)}{R^2}$, $\psi_{R}^{\prime}\left(z_R^{-}\right)=-\si z_{R}>-z_R, \psi_{R}^{\prime}\left(z_R^{+}\right)=-\frac{z_R}{R\sqrt{1+z_R^2}}\psi_{R}(z_R)<0$.
	\item $\displaystyle\int_{-\infty}^{+\infty}\varphi_{R}(x)dx=2\int_{0}^{+\infty}\varphi_{R}(x)dx\leq 2\int_{0}^{+\infty}e^{-\frac{x}{R}}dx=2R$.
\end{itemize}
These properties are easily to verify, except for \eqref{351}, whose verification is deferred to Remark \ref{remark319}.
Let $R\geq R_0$. Using integration by parts and the above properties, we obtain
\begin{align*}
	\int_{0}^{+\infty}\int_{-\infty}^{+\infty}\varphi_{R}\left( x \right)\psi_{R}\left(\frac{y}{R}\right)F^{\prime}(u)u_{xx}dxdy
	=&\int_{0}^{+\infty}\int_{-\infty}^{+\infty}-\varphi_{R}^{\prime}\left( x \right)\psi_{R}\left(\frac{y}{R}\right)\frac{\p F(u(x,y,t))}{\p x}\\
	&-\varphi_{R}\left( x \right)\psi_{R}\left(\frac{y}{R}\right)F^{\prime\prime}(u)u_{x}^{2}dxdy\\
	\leq & \int_{0}^{+\infty}\int_{-\infty}^{+\infty}\varphi_{R}^{\prime\prime}\left( x \right)\psi_{R}\left(\frac{y}{R}\right) F(u)dxdy\\
	\leq & \int_{0}^{+\infty}\int_{-\infty}^{+\infty}\frac{1}{R^2}\varphi_{R}(x)\psi_{R}\left(\frac{y}{R}\right)F(u)dxdy
\end{align*}
and
\begin{align*}
	\int_{0}^{+\infty}\int_{-\infty}^{+\infty}\varphi_{R}(x)\psi_{R}\left(\frac{y}{R}\right)F^{\prime}(u)u_{yy}dxdy
	\leq & \int_{-\infty}^{+\infty}\int_{0}^{Rz_R}-\frac{1}{R}\varphi_{R}(x)\psi_{R}^{\prime}\left(\frac{y}{R}\right)\frac{\p F}{\p y}dydx\\
	&+\int_{-\infty}^{+\infty}\int_{Rz_R}^{+\infty}-\frac{1}{R}\varphi_{R}(x)\psi_{R}^{\prime}\left(\frac{y}{R}\right)\frac{\p F}{\p y}dydx\\
	< & \int_{-\infty}^{+\infty}\int_{Rz_R}^{+\infty}\frac{1}{R^4}\varphi_{R}( x )\psi_R\left(\frac{y}{R}\right)F(u)dydx+\frac{K}{6Rd_1}z_R\cdot 2R\\
	\leq & \int_{0}^{+\infty}\int_{-\infty}^{+\infty}\frac{1}{R^2}\varphi_{R}( x )\psi_R\left(\frac{y}{R}\right)F(u)dxdy+\frac{K}{3d_1}z_R.
\end{align*}
Similarly, one has
\[
\int_{0}^{+\infty}\int_{-\infty}^{+\infty}\varphi_{R}(x)\psi_{R}\left(\frac{y}{R}\right)G^{\prime}(v)v_{xx}dxdy\leq  \int_{0}^{+\infty}\int_{-\infty}^{+\infty}\frac{1}{R^2}\varphi_{R}(x)\psi_{R}\left(\frac{y}{R}\right)G(v)dxdy
\]
and
\[
    \begin{aligned}
        \int_{0}^{+\infty}\int_{-\infty}^{+\infty}\varphi_{R}(x)\psi_{R}\left(\frac{y}{R}\right)G^{\prime}(v)v_{yy}dxdy
        < & \int_{0}^{+\infty}\int_{-\infty}^{+\infty}\frac{1}{R^2}\varphi_{R}(x)\psi_{R}\left(\frac{y}{R}\right)G(v)dxdy+\frac{K}{3d_2}z_R \\ &-\frac{1}{d_2}\int_{-\infty}^{+\infty}\varphi_{R}(x)\left(1-\tfrac{v^{*}}{v(x,0,t)}\right)\left[\nu v(x,0,t)-\mu w\right]dx.
    \end{aligned}
\]
Moreover,
\[
   \int_{-\infty}^{+\infty}\varphi_{R}(x)H^{\prime}(w)w_{xx}dx\leq \int_{-\infty}^{+\infty}\frac{1}{R^2}\varphi_{R}(x)H(w)dx.
\]
Therefore, for any $R\geq R_0$, we obtain
\begin{align*}
	&\frac{d}{dt}\mathcal{F}_{R}(t)+\sqrt{z_{R}}\mathcal{F}_{R}(t)\\
	\leq & \int_{0}^{+\infty}\int_{-\infty}^{+\infty}\varphi_{R}(x)\psi_{R}\left(\frac{y}{R}\right) \left[d_1F^{\prime}(u)\Delta u+\frac{d_2}{R^{2}}G^{\prime}(v)\Delta v\right]\\
	& \qquad+ \int_{-\infty}^{+\infty}\frac{\varphi_R(x)}{R^{2}}H^{\prime}(w)\left[Dw_{xx}+\nu v(x,0,t)-\mu w(x,t)\right]dx+\sqrt{z_{R}}\int_{-\infty}^{+\infty}\varphi_R\left( x\right)\frac{H(w(x,t))}{R^{2}}dx\\
	\leq & \int_{0}^{+\infty}\int_{-\infty}^{+\infty}\frac{2}{R^{2}} \varphi_R( x)\psi_{R}\left(\frac{y}{R}\right) \left[d_1F(u)+d_2\frac{G(v)}{R^{2}}\right]dxdy+\int_{-\infty}^{+\infty}\frac{D}{R^{2}} \varphi_R( x )\frac{H(w)}{R^{2}}dx +\frac{2K}{3}z_{R}\\
	&\qquad -\int_{-\infty}^{+\infty}\frac{\varphi_{R}(x)}{R^{2}}\frac{v^{*}}{\mu v(x,0,t)w(x,t)}\left[\nu v(x,0,t)-\mu w(x,t)\right]^{2}dx+\frac{K}{3}\frac{\sqrt{z_{R}}}{R}\\
	\leq & \int_{0}^{+\infty}\int_{-\infty}^{+\infty}\frac{2}{R^{2}} \varphi_R( x)\psi_{R}\left(\frac{y}{R}\right) \left[d_1F(u)+d_2\frac{G(v)}{R^{2}}\right]dxdy+\int_{-\infty}^{+\infty}\frac{D}{R^{2}} \varphi_R( x )\frac{H(w)}{R^{2}}dx +Kz_{R}\\
	\leq & \frac{\sqrt{z_{R}}}{2}\mathcal{F}_{R}(t)+Kz_{R}.
\end{align*}
Then one can find $R$ sufficiently large such that
$ \frac{d}{dt}\mathcal{F}_{R}\leq -\frac{\sqrt{z_{R}}}{2}\mathcal{F}_{R}+ Kz_R$.
This yields that $\left(\mathcal{F}_R(t)-2K\sqrt{z_{R}}\right)e^{\frac{\sqrt{z_{R}}}{2}t}$ is nonincreasing in $t\in\R$. Hence,
\[
    \li_{R\to\infty}\mathcal{F}_{R}(t)=\lim_{R \to \infty}2K\sqrt{z_{R}}=0,\qquad \forall \, t\in\R.
\]
due to $\li_{t\to-\infty}\left(\mathcal{F}_R(t)-2K\sqrt{z_{R}}\right)e^{\frac{\sqrt{z_{R}}}{2}t}=0$ for all $R\geq R_0$. By dominated convergence theorem,
\[
	\begin{aligned}
		\lim_{R \to \infty}\mathcal{F}_R(t)& \geq\lim_{R \to \infty}\int_{0}^{+\infty}\int_{-\infty}^{+\infty}\varphi_R\left( x\right)\psi_{R}\left(\frac{y}{R}\right)F(u(x,y,t))dxdy\\
		& \geq \lim_{R \to \infty}\int_{0}^{+\infty}\int_{-\infty}^{+\infty}\varphi_R\left( x\right)\varphi_{R}\left(\frac{y}{R}\right)F(u(x,y,t))dxdy\\
		& =\int_{0}^{+\infty}\int_{-\infty}^{+\infty}F(u(x,y,t))dxdy.
	\end{aligned}
\]
Consequently,
$F(u(x,y,t))=0$ for any $x\in \R, y>0$ and $t\in\R$, thus $u=u^{*}$ in $\O_0\times \R$. This, combined with the first equation of \eqref{prey-predator-road}, gives that $v=v^{*}$ in $\O_0\times \R$. Since $u,\,v$ are continuous up to the boundary of $\O_0$, we have $u\equiv u^{*}$ and $v\equiv v^{*}$ in $\ov{\O}_0\times \R$. Finally, this lemma follows by using the last equation of \eqref{prey-predator-road}, which implies that $w\equiv w^{*}$ on $\mathbb{R}\times \mathbb{R}$.
\end{proof}
\begin{remark}\label{remark319}
	We provide a brief explanation of the existence of $R_0$ for which \eqref{351} holds. Define $l:=\left(\max\left\{4d_1,4d_2,2D,1\right\}\right)^{2}$ and $f(x):=\cos x-e^{-\frac{x}{\sqrt{l}}}$. Then $f(0)=0$ and there exists $x_0>0$ such that $f^{\prime}(x)=-\sin x+\frac{1}{\sqrt{l}}e^{-\frac{x}{\sqrt{l}}}\geq 0$ for $0\leq x\leq x_0$. Therefore, we have 
$\cos x\geq e^{-\frac{x}{\sqrt{l}}}\geq e^{-\frac{\sqrt{x+x^3}}{\sqrt{l}}}$  for  $x\in[0,x_0]$. 
Thus, we have $R_0\geq 1$ such that
$\cos \frac{l}{R^{2}}>e^{-\sqrt{\frac{1}{R^2}+\frac{l^2}{R^6}}}=\varphi\left(\frac{l}{R^2}\right)$ for all $R\geq R_0$.
It implies that $z_R\geq \frac{l}{R^{2}}$ for $R\geq R_0$ and thus \eqref{351} holds.
\end{remark}

\vspace{1.5ex}
\begin{proof}[Proofs of \eqref{plane converge} and \eqref{w converges}] Fix $0<\e<\min\left\{c^*,2\sqrt{d_2(b-1)}\right\}$. We again proceed by contradiction.  Assume that there exist $\d>0,(x_n,y_n)\in\ov{\O}_0$, and $t_n\geq 0$ such that
\[
	{\rm dist}\left(\frac{1}{t_n}(x_n,y_n),\ov{\O}_0\setminus\mathcal{V}\right)>\e,\quad \li_{n\to\infty}t_n=+\infty,
\]
and
\begin{equation}\label{converge contra}
|u(x_n,y_n,t_n)-u^*|+|v(x_n,y_n,t_n)-v^*|\geq \d.
\end{equation}
Then there exists a constant $0<k<1$ such that
\[
	{\rm dist}\left(\frac{1}{kt_n}(x_n,y_n),\ov{\O}_0\setminus\mathcal{V}\right)>\frac{\e}{2},\qquad \forall n\in\N.
\]
Actually, we can take $k=\frac{2(c_0-\e)}{2c_0-\e}$ with $c_0:=\min\left\{c^*,\vr^*_{b-1}\left(\frac{\pi}{2}\right)\right\}$.
Define the function sequences $u_n(x,y,t):=u(x+x_n,y+y_n,t+t_n), v_n(x,y,t):=v(x+x_n,y+y_n,t+t_n), w_n(x,t):=w(x+x_n,t+t_n)$ for $(x,y)\in\ov{\O}_{-y_n},t\geq (k-1)t_n$. 
Then, there exists $l\in[-\infty,0]$ such that after passing to a subsequence of $n\to\infty$, $y_n\to -l$ and $(u_n,v_n,w_n)$ converges to some triplet $(u_{\infty},v_{\infty},w_{\infty})$ locally uniformly in $\left(C^{2,1}(\O_{l}\times\R)\right)^2\times C^{2,1}(\R\times\R)$. Moreover, $\displaystyle \lim_{n \to \infty}u_n=u_{\infty},\lim_{n \to \infty}v_n=v_{\infty}$ in $C^{1}\left(\ov{\O}_l\times \R\right)$ if $l>-\infty$.
Since $t_{n}\to \infty$, for any $(x,y)\in \R^2$, we can find $N(x,y)\in\N$ large enough such that
\[
	\left|\frac{1}{kt_n}(x,y)\right|<\frac{\varepsilon}{4}, \qquad \forall n\geq N(x,y).
\]
Thus for any  $(x,y)\in \R^2$, $t\geq (k-1)t_n$ and $n\geq N(x,y)$, it holds
\[
	\begin{aligned}
	{\rm dist}\left(\frac{1}{t+t_{n}}(x+x_{n},y+y_{n}),\ov{\O}_0\setminus\mathcal{V}\right) 
	\geq &{\rm dist}\left(\frac{1}{kt_{n}}(x+x_{n},y+y_{n}),\ov{\O}_0\setminus\mathcal{V}\right) \\
	\geq&{\rm dist}\left(\frac{1}{kt_n}(x_n,y_n),\ov{\O}_0\setminus\mathcal{V}\right)-\left|\frac{1}{kt_n}(x,y)\right|\\
	>&\frac{\e}{4}>0.
	\end{aligned}
\]
Therefore, Theorems \ref{persistence}, \ref{uniform} and \ref{thm317} imply that there exists $\d>0$ such that
\[
	\d\leq u_{n}(x,y,t),v_{n}(x,y,t),w_{n}(x,t)\leq \max\left\{1, b-1,\frac{\nu}{\mu}(b-1)\right\}
\]
forall $(x,y)\in \R^2,n\geq N(x,y)$ and $t\geq (k-1)t_n$. Thus, it follows that
\[
	\d\leq u_{\infty}(x,y,t),v_{\infty}(x,y,t),w_{\infty}(x,t)\leq \max\left\{1, b-1,\frac{\nu}{\mu}(b-1)\right\},\qquad \forall (x,y)\in\ov{\O}_l,t\in\R.
\]
If $l>-\infty$, then $(u_{\infty},v_{\infty},w_\infty)$ is an entire solution of \eqref{prey-predator-road}, after a translation in $y$. If $l=-\infty$, $(u_{\infty},v_{\infty})$ satisfies
\[
	\left\{
		\begin{aligned}
			&\partial_{t}u_{\infty}-d_1\Delta u_{\infty}=u_{\infty}(1-u_{\infty}-a v_{\infty}),&&\quad(x,y,t)\in \mathbb{R}^{3},&\\
			&\partial_{t}v_{\infty}-d_2\Delta v_{\infty}=v_{\infty}(-1+bu_{\infty}-v_{\infty}),&&\quad(x,y,t)\in \mathbb{R}^3.&
			\end{aligned}
	\right.
\]
For such two cases, it follows from Lemma \ref{lem converge} and \cite[Theorem 1.1]{guo+2021} that $(u_{\infty},v_{\infty},w_{\infty})\equiv (u^*,v^*,w^*)$, contradicts \eqref{converge contra}. Therefore, we have \eqref{plane converge}. Furthermore, we obtain \eqref{w converges} by repeating the above arguments with fixed $0\leq c<\min\left\{c^{*},\vr^{*}_{b-1}\left(\frac{\pi}{2}\right)\right\}$ and $|x_n|\leq ct_n, y_n=0, ~\forall n\in\N$. This completes the proof.
\end{proof}

\section{\texorpdfstring{Spreading Speeds for \eqref{prey-predator-road} with \ref{hy-initial-positive}}{Spreading Speeds for (1.4) with (H1')}}\label{positive}
In this section, we consider the propagation of \eqref{prey-predator-road} with initial condition \ref{hy-initial-positive}. In this case, the prey is initially distributed throughout the entire field. Hence the invasion speed of predators depends only on their own diffusion ability.
We first show that the prey persists throughout the whole field.
\begin{theorem}\label{lem51}
	Let $(u,v,w)$ be the solution of \eqref{prey-predator-road} with $(u,v,w)=(u_0,v_0,w_0)$ at $t=0$.
	Then there exists $\eta>0$ such that 	for any $(u_0,v_0,w_0)$ satisfying {\rm\ref{hy-initial-positive}}, there holds
	\[
		\inf_{\substack{(x,y)\in\ov{\O}_0\\t\geq 0}} u(x,y,t) \geq \eta.
	\]
\end{theorem}
\begin{proof}
	Choose $0<\d<\frac{1}{a}$. By Lemma \ref{predator die without prey}, we can find $T_\d>0$ and $M_{\d}>0$ such that for any $(u_0,v_0,w_0)$ satisfying \ref{hy-initial-positive}, $u(x,y,t+T_\d)$ is a supersolution of
	\begin{equation}\label{52}
		\left\{
			\begin{aligned}
				&\partial_{t} \underline{u}-d_1 \Delta \underline{u}=\left[ 1-a\delta -(1+aM_{\delta})\underline{u}\right]\underline{u},\qquad &&(x,y)\in \O_0, t> 0,&\\
				&\p_y\un{u}(x,0,t)=0, &&x\in\R,t>0,&\\
				&\un{u}(x,y,0)=u(x,y,T_{\d}),&& (x,y)\in \O_0.&
			\end{aligned}
		\right.
	\end{equation}
	Since $0\leq u(x,y,t)\leq 1,~v(x,y,t)\leq b-1$ and $u_0(x,y)\geq \varepsilon_0$, we infer from the comparison principle that $u(x,y,t)\geq \varepsilon_0e^{-a(b-1)t}$. Therefore, $\un{u}(x,y,0)=u(x,y,T_{\d})\geq \varepsilon_0e^{-a(b-1)T_{\d}}>0$. We take $0<\eta_{\d}<\min\left\{\varepsilon_0e^{-a(b-1)T_{\d}},\frac{1-a\delta}{1+aM_{\d}}\right\}$, which is a strict subsolution of \eqref{52}. Invoking the comparison principle again, we find
	\[
		\inf_{\substack{(x,y)\in\ov{\O}_0\\ t\geq T_{\d}}} u(x,y,t)\geq \inf_{\substack{(x,y)\in\ov{\O}_0\\ t\geq 0}} \un{u}(x,y,t)>\eta_{\d}.
	\]
	Since
	\[
		\inf_{\substack{(x,y)\in\ov{\O}_0\\ 0\leq t\leq T_{\d}}} u(x,y,t)\geq \varepsilon_0e^{-a(b-1)T_{\d}},
	\]
	The proof is completed by taking $\eta=\min\left\{ \eta_{\d}, \varepsilon_0e^{-a(b-1)T_{\d}} \right\}$.
\end{proof}
\begin{lemma}\label{lem52}
	 Let $(u,v,w)$ be the solution of \eqref{prey-predator-road} with $(u,v,w)=(u_0,v_0,w_0)$ at $t=0$. Then for any $\s>0$, there exist $M_{\s}>0$ and $T_{\s}>0$ such that for any $(u_0,v_0,w_0)$ satisfying {\rm\ref{hy-initial-positive}}, we have
	\[
		1-u(x,y,t)\leq \s+M_{\s}v(x,y,t),\qquad \forall t\geq T_{\s},(x,y)\in\ov{\O}_0.
	\]
\end{lemma}
\begin{proof}
	Suppose by contradiction that there exist $\s_0>0$ and sequences $\left\{ (u_{0,n},v_{0,n},w_{0,n}) \right\}$  satisfying \ref{hy-initial-positive}, $\{t_n\}$ and $\{(x_n,y_n)\}\subset \ov{\O}_0$ such that
	\[
		\lim_{n\to\infty}t_n=+\infty, \qquad 1-u^{n}(x_n,y_n,t_n)>\sigma_0+n v^{n}(x_n,y_n,t_n),~\forall n\in\N,
	\]
where $(u^n,v^n,w^n)$ is the solution of \eqref{prey-predator-road} with initial value $ (u_{0,n},v_{0,n},w_{0,n})$, $n\in\N$. 	After passing a subsequence of $n\to\infty$, $v^{n}(x+x_n,y+y_n,t+t_n)\to 0$ and $u^{n}(x+x_n,y+y_n,t+t_n)$ converges to some function $u^{\infty}(x,y,t)$. Fix $(x,y)\in\ov{\O}_{-y_n}, t\in\R$.
	Then 
	\[
		u^{n}(x+x_n,y+y_n,t+t_n)\geq \inf_{\substack{(\xi,\zeta)\in\ov{\O}_0\\ \tau\geq 0}} u^{n}(\xi,\zeta,\tau)\geq \eta,\qquad \forall n\geq N
	\]
	for some $N\in\N$ such that $t+t_n\geq 0,~\forall n\geq N$, where $\eta>0$ is defined in Theorem \ref{lem51}. Therefore, one has  
	\[
		\tilde{\eta}:=\inf_{\substack{(x,y)\in\ov{\O}_{-z_0}\\ t\in\R}}u^{\infty}(x,y,t)\geq\eta>0,
	\]
	where $z_0\in[0,+\infty]$ is the limit of a subsequence of $\{y_n\}$ as $n\to\infty$. Notice that  $u^{\infty}$ is a solution of \eqref{44} or \eqref{46}. Applying the comparison principle with the solution of equation
	$\un{u}^{\prime}=\un{u}(1-\un{u})$ in $ t>0$ with $\un{u}(0)=\tilde{\eta}$, 
	we have $u^{\infty}(x,y,t+\tau)\geq\un{u}(t)$ for all $t\geq 0,\tau\in\R,(x,y)\in\ov{\O}_{-z_0}$. Hence $u^{\infty}(0,0,0)\geq \displaystyle \lim_{\tau\to-\infty} u(-\tau)=1$. This is a contradiction to $u^{\infty}(0,0,0)\leq 1-\s_0$. This completes the proof.
\end{proof}
\begin{theorem}\label{thm43}
	 Let $(u,v,w)$ be the solution of \eqref{prey-predator-road} with  initial data $(u,v,w)=(u_0,v_0,w_0)$. Assume that {\rm \ref{hy-initial-positive}} holds. For any $\e>0$, there exists $\varpi>0$ such that
	\[
		\liminf_{ t \to \infty }\inf_{\substack{(x,y)\in\ov{\O}_0,\\ {\rm dist}\left(\frac{1}{t}(x,y),\ov{\O}_0\setminus \mathcal{W}_{b-1}\right)}} v(x,y,t)\ge \varpi.
	\]
\end{theorem}

This result can be proved by using an argument similar to that of Theorem \ref{uniform}, hence we omit the details. In contrast to the proof of Theorem \ref{uniform}, where $(v,w)$ is a supersolution of the truncated system \eqref{351}, the lower control system of $(v,w)$ in the proof of Theorem \ref{thm43} is
 \[
 	\left\{
		\begin{aligned}
			&\p_t\un{v}-d_2\D \un{v}=\big[b(1-\s)-1-(bM_{\s}+1)\un{v}\big]\un{v},&& (x,y)\in\O_{0},t>T_{\sigma}, &\\
			&\p_t\un{w}-D \p_{xx}\un{w}=-\m\un{w}+\n\un{v}(x,0,t),&& x\in\R,t>T_{\sigma}, &\\
			& -d_2\p_y\un{v}(x,0,t)=\m\un{w}(x,t)-\n\un{v}(x,0,t),&& x\in\R,t>T_{\sigma},&\\
			&\underline{v}(x,y,T_{\sigma})=v(x,y,T_{\sigma}),&&(x,y)\in\O_{0},&\\
			&\underline{w}(x,T_{\sigma})=w(x,T_{\sigma}),&&x\in\R.&
		\end{aligned}
	\right.
\]
whose lower bound of spreading speed along the direction $(\sin\vt,\cos\vt)$ is $\vr_{b(1-\sigma)-1}^*(\vt)$. It makes the proof of Theorem \ref{thm43} easier than that of Theorem \ref{uniform}.

Using Lemma \ref{lem313}, Theorem\ref{lem51}, Theorem \ref{thm43} and Lemma \ref{lem converge}, in analogy with the proofs of \eqref{plane converge} and \eqref{w converges}, we obtain \eqref{positive-coexistence} and \eqref{positive-w converge}.
Since $(v,w)$ is a subsolution of \eqref{39}, we have \eqref{v small best} and \eqref{positive-w die}. Using Theorem \ref{lem51}, a similar proof of \eqref{u best} gives that for any $\varepsilon>0$,
\[
	\lim_{t\to\infty}\sup_{\substack{(x,y)\in\ov{\O}_0\\{{\rm dist}\left(\frac{1}{t}(x,y),\mathcal{W}_{b-1}\right)>\varepsilon}}}|u(x,y,t)-1|=0.
\]
Combining this with \eqref{v small best}, we have \eqref{positive-vdie}.


\section*{Conflict of interest}
All authors declare no conflict of interest in this study.

\section*{Data availability statement}
No data was used and no new data was generated in this study.

\section*{Appendix A}\label{Appendix}
\renewcommand{\thetheorem}{A.\arabic{theorem}}
\setcounter{theorem}{0}
\renewcommand{\theequation}{A.\arabic{equation}}
\setcounter{equation}{0}
	Consider a cooperative system
	\begin{equation}\label{cooperate}
		\left\{\begin{array}{ll}
			\partial_{t}u-d_1\Delta u =g_1(u,v),&\qquad(x,y)\in \O_0,t>0,\\
			\partial_{t}v-d_2\Delta v =g_2(u,v),&\qquad(x,y)\in \O_0,t>0,\\
			\partial_{t}w-D\p_{xx}w=-\mu w+ \n v(x,0,t),&\qquad x\in \mathbb{R},t>0,\\
			\p_y u(x,0,t)=0,& \qquad x\in \mathbb{R},t>0,\\
			-d_2\partial_{y}v(x,0,t)=\mu w-\n v(x,0,t),&\qquad x\in \mathbb{R},t>0,\\
		\end{array}\right.
	\end{equation}
where 
$g_1(u,v)$ is nondecreasing in $v$, and $g_2(u,v)$ is nondecreasing in $u$. Moreover, $g_1$ and $g_2$ are Lipschitz continuous with Lipschitz constant $L$, the initial data $(u_0,v_0,w_0)$ satisfies {\rm \ref{h2}}. In this section, we prove that \eqref{cooperate} satisfies the comparison principle, namely, the following lemma.
\begin{definition}
	A nonnegative triplet $(\ov{u},\ov{v},\ov{w})$ is called a supersolution {\rm(resp. subsolution)} of \eqref{cooperate} if
	\begin{itemize}
		\item $\overline{u},\,\overline{v}\in C^{2,1}\left(\O_0\times(0,+\infty) \right)\cap C\left(\overline{\Omega}_0\times[0,+\infty)\right)$, $\overline{w}\in  C^{2,1}\left(\R \times(0,+\infty)\right)\cap C\left(\R\times[0,+\infty)\right)$;
		\item $(\ov{u},\ov{v},\ov{w})$ satisfies \[
		\left\{\begin{array}{ll}
			\partial_{t}\overline{u}-d_1\Delta \overline{u}\ge g_{1}(\overline{u},\overline{v}),&\qquad(x,y)\in \O_0,t>0,\\
			\partial_{t}\overline{v}-d_2\Delta \overline{v}\ge g_{2}(\overline{u},\overline{v}),&\qquad(x,y)\in \O_0,t>0,\\
			\partial_{t}\overline{w}-D\p_{xx}\overline{w}\ge -\mu \overline{w}+\n \overline{v}(x,0,t),&\qquad x\in \mathbb{R},t>0, \\
			-\p_y \ov{u}(x,0,t)\geq 0,&\qquad x\in\R,t>0,\\
			-d_2\p_y\ov{v}(x,0,t)\ge \mu \overline{w}-\n\overline{v}(x,0,t),&\qquad x\in \mathbb{R},t>0.
		\end{array}\right.
		\]
	\end{itemize}
\end{definition}
Analogously, we define $(\un{u},\un{v},\un{w})$ as a subsolution of \eqref{cooperate} by replacing ``$\geq$'' with ``$\leq$'' in the above definition.
\begin{lemma}\label{comparison-cooperate}
	Let $\left(\ov{u},\ov{v},\ov{w}\right)$ and $\left(\un{u},\un{v},\un{w}\right)$ be repectively the supersolution bounded from below and the subsolution bounded from above of \eqref{cooperate}. Assume that $\left.(\ov{u},\ov{v},\ov{w})\right|_{t=0}\geq\left.(\un{u},\un{v},\un{w})\right|_{t=0}$. Then we have
$	\left(\ov{u},\ov{v},\ov{w}\right)\geq \left(\un{u},\un{v},\un{w}\right)$  for  any $ (x,y)\in\ov{\O}_0$ and $t\geq 0$.
Moreover, it holds that

i) either $\ov{u}> \un{u}$ in $\ov{\O}_0\times (0,+\infty)$ or there exists $T>0$ such that $\ov{u}\equiv \un{u}$ on $\ov{\O}_0\times [0,T]$;

ii) either $\left(\ov{v},\ov{w}\right)> \left(\un{v},\un{w}\right)$ in $\ov{\O}_0\times (0,+\infty)$ or there exists $T>0$ such that $\left(\ov{v},\ov{w}\right)\equiv \left(\un{v},\un{w}\right)$ on $\ov{\O}_0\times [0,T]$.
\end{lemma}

\begin{proof}
	Take $\widetilde{L}\geq 2L$. Set
$(\tilde{u},\tilde{v},\tilde{w}):=(\ov{u},\ov{v},\ov{w})e^{-\widetilde{L}t}$ and $\left(\check{u},\check{v},\check{w}\right):=\left(\un{u},\un{v},\un{w}\right)e^{-\widetilde{L}t}$.
	Then $(\tilde{u},\tilde{v},\tilde{w})$ and $\left(\check{u},\check{v},\check{w}\right)$ are respectively super- and subsolutions of the following system
		\[
		\left\{\begin{array}{ll}
			\partial_{t}u-d_1\Delta u+\left(\tilde{L}-L\right)u=\tilde{g}_1(u,v),&\qquad(x,y)\in \O_0,t>0,\\
			\partial_{t}v-d_2\Delta v+\left(\tilde{L}-L\right)v=\tilde{g}_2(u,v),&\qquad(x,y)\in \O_0,t>0,\\
			\partial_{t}w-D\p_{xx}w+\left(\mu+\tilde{L}\right)w=\n v(x,0,t),&\qquad x\in \mathbb{R},t>0,\\
			\p_y u(x,0,t)=0,& \qquad x\in \mathbb{R},t>0,\\
			-d_2\partial_{y}v(x,0,t)=\mu w-\n v(x,0,t),&\qquad x\in \mathbb{R},t>0.
		\end{array}\right.
	\]
	Here $\tilde{g}_1(u,v)=g_1\left(u e^{\tilde{L}t},ve^{\tilde{L}t}\right)e^{-\tilde{L}t}-Lu$, $\tilde{g}_2(u,v)=g_2\left(u e^{\tilde{L}t},ve^{\tilde{L}t}\right)e^{-\tilde{L}t}-Lv$. Up to increasing $L$, we can assume that $\tilde{g}_1(u,v)$ and $\tilde{g}_2(u,v)$ are strictly decreasing in $u$ and $v$, respectively. In addition, $\tilde{g}_1(u,v)$ and $\tilde{g}_2(u,v)$ are nondecreasing in $v$ and $u$, respectively. Let $\Lambda:\mathbb{R}\to [0,+\infty)$ be a smooth function satisfying
	\begin{equation}\label{24}
		 \Lambda(0)=0, \quad \li_{ |\rho| \to +\infty } \Lambda(\rho)=+\infty,\quad
			(2d_1+2d_2+D)\left\| \Lambda^{\prime\prime} \right\|\le 1.
	\end{equation}
	For any $\varepsilon>0$, set
$$\begin{array}{lll}
& \hat{u}(x,y,t):=\tilde{u}+\varepsilon \left( \Lambda(x) +\Lambda(y)+t+1\right), & (x,y)\in\Omega_0,t\geq 0, \\
& \hat{v}(x,y,t):=\tilde{v}+\varepsilon\left( \Lambda(x)+\Lambda(y)+t+1 \right), & (x,y)\in\Omega_0,t\geq 0,\\
& \hat{w}(x,t):=\tilde{w}+\frac{\nu}{\mu }\varepsilon\left( \Lambda(x)+t+1 \right), &x\in\mathbb{R},t\geq 0 .
\end{array}$$
Notice that $(\hat{u},\hat{v},\hat{w})>(\check{u},\check{v},\check{w})$ at $t=0$, we claim that this strict inequality holds for all $t\ge 0$.
If not, let $T$ be the smallest time at which
\[
	\min\left\{ \inf_{\ov{\O}_0}(\hat{u}-\check{u})(\cdot,T),\, \inf_{\ov{\O}_0}(\hat{v}-\check{v})(\cdot,T),\, \inf_{\mathbb{R}}(\hat{w}-\check{w})(\cdot,T)\right\}=0.
\]
Then one has $0<T<+\infty$, since $(\check{u},\check{v},\check{w})$ is bounded from above while
\[
	\li_{t\to\infty}\hat{u}(x,y,t)=\li_{t\to\infty}\hat{v}(x,y,t)=+\infty\qquad \text{uniformly in}~ (x,y)\in \ov{\O}_0
\]
and
\[\li_{t\to+\infty}\hat{w}(x,t)=+\infty \qquad\text{uniformly in}~ x\in\R.\]
Furthermore, $(\hat{u},\hat{v})>(\check{u},\check{v})$ in $\ov{\O}_0\times[0,T)$,  $\hat{w}> \check{w}$ in $\mathbb{R}\times [0,T)$.
Since $\li_{|x|+|y|\to\infty}\hat{u}(x,y,t)=\li_{|x|+|y|\to\infty}\hat{v}(x,y,t)=+\infty$ and $\li_{|x|\to\infty}\hat{w}(x,t)=+\infty$ uniformly in $t\geq 0$, we have that at least one of the following three cases holds.

	\vspace{1.5ex}
	\noindent\textit{Case 1.} 
$(\hat{u}-\check{u})(\xi,T)=0$ for some $\xi=(\xi_1,\xi_2)\in\ov{\O}_0$.
	\vspace{1.5ex}

Note that we have $\hat{v}(x,y,t)\geq \check{v}(x,y,t)$ for all $(x,y)\in\Omega_0$ and $t\in [0,T]$. If $\xi_{2}>0$, there exist $0<r<\frac{\xi_2}{2}$ and $0<t_0<T$ such that $\left(\tilde{u}-\check{u}\right)\leq 0$ on $\ov{B}_r(\xi)\times [t_0,T]$. Then for any $(x,y,t)\in B_{r}(\xi)\times[t_0,T]$, by the monotonicity of $\tilde{g}_1$ and \eqref{24} one has
\[
	\begin{aligned}
	& \p_t(\hat{u}-\check{u})-d_1\D (\hat{u}-\check{u})+\left(\widetilde{L}-L\right)(\hat{u}-\check{u})\\
	\ge & \tilde{g}_1(\tilde{u},\tilde{v})-\tilde{g}_1(\check{u},\check{v})+\e\left[ 1-d_1\left( \left|\Lambda^{\prime\prime}(x)\right|+\left|\Lambda^{\prime\prime}(y)\right| \right) \right]+\left(\widetilde{L}-L\right)\left(\hat{u}-\tilde{u}\right)\\
	\geq & \tilde{g}_1(\tilde{u},\tilde{v})-\tilde{g}_1(\check{u},\check{v})+\left(\widetilde{L}-L\right)\left(\hat{u}-\tilde{u}\right)\\
	=& \tilde{g}_1(\tilde{u},\tilde{v})-\tilde{g}_1(\check{u},\check{v})+\left(\widetilde{L}-L\right)\left(\hat{v}-\tilde{v}\right)\\
	> & \tilde{g}_1(\tilde{u},\hat{v})-\tilde{g}_1(\check{u},\check{v})\geq 0.
\end{aligned}
\]
Since $(\xi,T)$ is the minimum point of $\left(\hat{u}-\check{u}\right)$ in $B_{r}(\xi)\times [t_0,T]$, we have that at $(\xi,T)$, there holds
\[
	\p_t\left(\hat{u}-\check{u}\right)\leq 0,\quad \nabla\left(\hat{u}-\check{u}\right)=0,\quad \D\left(\hat{u}-\check{u}\right)\geq 0.
\]
It infers that
	$0<\p_t(\hat{u}-\check{u})-d_1\D (\hat{u}-\check{u})+\left(\widetilde{L}-L\right)(\hat{u}-\check{u})\leq 0$,
which is impossible.
Therefore, $(\hat{u}-\un{u})(x,y,T)>0$ for all $(x,y)\in\O_0$, and  $\xi_{2}=0$. The Hopf's lemma implies
$	\p_y(\hat{u}-\check{u})(\xi_1,0,T)>0$,
which contradicts the fact that $\p_y(\hat{u}-\check{u})(x,0,t)=\p_y(\ov{u}-\un{u})(x,0,t)\leq 0$ for all $x\in\R,t>0$.

	\vspace{1.5ex}
	\noindent\textit{Case 2.} $(\hat{w}-\check{w})(\xi,T)=0$ for some $\xi\in\R$.
	\vspace{1.5ex}
		
	Notice that
	\begin{equation}\label{25}
		\begin{aligned}
			&\partial_{t}(\hat{w}-\check{w})-D\p_{xx}(\hat{w}-\check{w})+\left(\tilde{L}+\mu\right) (\hat{w}-\check{w})\\
			\geq & \nu\tilde{v}(x,0,t)+\frac{\nu\e}{\mu}\left(1-D\Lambda^{\prime\prime}(x)\right)+\left(\mu+\tilde{L}\right)\frac{\nu}{\mu}\e(\Lambda(x)+t+1)-\nu\check{v}(x,0,t) \\
			> & \n\tilde{v}(x,0,t)+\n\e\left(\Lambda(x)+t+1\right)-\n\check{v}(x,0,t)\\
			= & \n(\hat{v}-\check{v})(x,0,t)\ge 0
		\end{aligned}
	\end{equation}
	for all $x\in \mathbb{R}$ and $t\in(0,T]$.
	However, since $(\xi,T)$ is a minimum point of $\hat{w}-\check{w}$ in $\ov{\O}_0\times [0,T]$, we have
		$\partial_{t}(\hat{w}-\check{w})-D\p_{xx}(\hat{w}-\check{w})+\left(\tilde{L}+\mu\right) (\hat{w}-\check{w})\leq 0$,
	which contradicts \eqref{25}, and thus this case is ruled out too.
	
		\vspace{1.5ex}
	\noindent\textit{Case 3.} $(\hat{v}-\check{v})(\xi,T)=0$ for some $\xi=(\xi_1,\xi_2)\in\ov{\O}_0$ and $(\hat{w}-\check{w})(\cdot,T)>0$ in $\R$.
		\vspace{1.5ex}
	
		Just as the arguments for {\it Case 1}, the minimum of $(\hat{v}-\check{v})$ must be attained on $\p\O_0$.
	Thus $\partial_{y}(\hat{v}-\check{v})(\xi_1,0,T)\geq 0$. Consequently, one obtains that
	\begin{align*}
		0<& \mu(\hat{w}-\check{w})(\xi_1,T)\leq \left[-d_2\p_y \hat{v}(\xi_1,0,t)+\nu \hat{v}(\xi_1,0,t)\right]-\left[-d_2 \p_y \check{v}(\xi_1,0,T)+\nu \check{v}(\xi_1,0,T)\right]\leq 0,
	\end{align*}
	which is impossible.
	This case has therefore to be discarded.
    
    Thus the claim is proven. It follows from the arbitrariness of $\e>0$ that $(\ov{u},\ov{v},\ov{w})\geq (\underline{u},\un{v},\un{w})$ for all $(x,y,t)\in \overline{\Omega}_0\times[0,+\infty)$.

Now, let us focus on the proof of the second result. Suppose that there is $T>0$ such that $\displaystyle \min_{\ov{\O}_0}(\ov{u}-\un{u})(\cdot,T)=0$, we prove that $\ov{u}-\un{u}\equiv 0$ on $\ov{\O}_0\times[0,T]$. Set
$\Sigma=\{\xi\in\ov{\O}_0:(\ov{u}-\un{u}) (\xi,T) =0\}$.
If $\Sigma\setminus\p\O_0\not=\emptyset$, we take $\xi_0:=(\xi_1,\xi_{2})\in\S\setminus\p\O_0$. In this case, there is $\xi_{2}>0$. Then for any $(x,y,t)\in B_{\xi_{2}/2}(\xi_0)\times(0,T]$, one has
\[
	\p_t(\ov{u}-\un{u})-d_1\D (\ov{u}-\un{u})
	\ge g_{1}(\ov{u},\ov{v})-g_{1}(\un{u},\un{v})\geq g_{1}(\ov{u},\un{v})-g_{1}(\un{u},\un{v})
	\ge -L\left(\ov{u}-\un{u}\right)
\]
Applying the strong maximum principle yields that $\ov{u}\equiv\un{u}$ on $\ov{B}_{\xi_{2}/2}(\xi_0)\times[0,T]$. Hence $\overline{u}\equiv\underline{u}$ on $\overline{\Omega}_0\times[0,T]$. If $\S\subset\p\O_0$, we have $\overline{u}(\cdot,T)>\underline{u}(\cdot,T)$ in $\Omega_0$ and thus
$\p_y\left (\ov{u}-\un{u}\right )(\xi_0,T) >0$  for any $\xi_0\in\S$, 
contradicting the boundary conditions $\p_y\un{u}(\xi,t)\geq 0\geq\p_y\ov{u}(\xi,t)$ for any $(\xi,t)\in\p\O_0\times(0,T]$.

If $\displaystyle \min_{\mathbb{R}}(\ov{w}-\un w)(\cdot,T)=0$ at some time $T>0$, using the strong maximum principle again as in {\it Case 2}, we have $\ov{w}\equiv\un w$ in $\R\times[0,T]$. And then, we can use the similar arguments to $\ov u-\un u$ and $\hat v-\check{v}$ to prove that $\ov v\equiv \un v$ in $\ov \O_0\times [0,T]$ if $\displaystyle \min_{\ov{\O}_0}(\ov{v}-\un{v})(\cdot,T)=0$ for some $T>0$. 

Finally, for fixed $T>0$, we claim that the following two statements are equivalent:
 \begin{enumerate}[label={\rm (\roman*)}]
	\item \label{equivalence1} $\ov{v}\equiv\un{v}$ for all $(x,y,t)\in\ov{\O}_0\times[0,T]$,
	\item \label{equivalence2} $\ov{w}\equiv\un{w}$ for all $(x,t)\in \R\times[0,T]$.
\end{enumerate}

If $\ov{w}(x,t) =\un{w}(x,t)$ in $\R\times[0,T]$ for some $T>0$, the third equation of \eqref{cooperate} yields that $\ov v(x,0,t)\equiv\un v(x,0,t)$ in $\R\times[0,T]$ and therefore $\ov v\equiv\un v$ on $\ov \O_0\times [0,T]$. On the other hand, suppose that \ref{equivalence1} holds while there is $t_0\in[0,T)$ such that $\ov{w}(x,t) >\un{w}(x,t)$ for all $(x,t)\in\R\times(t_0,+\infty)$. Then
 \[ \nu(\ov v-\un v)(x,0,t)-d_2\p_y(\ov v-\un v)(x,0,t)\geq \mu(\ov w-\un w)(x,t)>0,\qquad \forall (x,t)\in\R\times(t_0,+\infty). \]
 It ensures that $\ov v>\un v$ in $\p\O_0 \times(t_0,+\infty)$,
which contradicts \ref{equivalence1}. Hence, \ref{equivalence1} is equivalent to \ref{equivalence2}. The proof is now complete.
\end{proof}
\begin{remark}
	If the monotonicity conditions on $g_1(u,\cdot)$ and $g_2(\cdot,v)$ are strict, we further have that either $\left(\ov{u},\ov{v},\ov{w}\right)> \left(\un{u},\un{v},\un{w}\right)$ in $\ov{\O}_0\times (0,+\infty)$ or there is $T>0$ such that $\left(\ov{u},\ov{v},\ov{w}\right)\equiv \left(\un{u},\un{v},\un{w}\right)$ on $\ov{\O}_0\times [0,T]$. Namely, \ref{equivalence1} and \ref{equivalence2} above are equivalent to
	\begin{enumerate}[label={\rm(\roman*)}]
	\setcounter{enumi}{2}
		\item\label{equivalence3} $\ov{u}\equiv\un{u}$ for all $(x,y,t)\in\ov{\O}_0\times[0,T]$.
	\end{enumerate}
	Suppose that \ref{equivalence3} holds while $\ov{v}>\un{v}$ on $\ov{\O}_0\times (t_0,+\infty)$ for some $t_0\in[0,T)$.
Since $\ov{u}\equiv\un{u}$ on $\ov{\O}_0\times [0,T]$, we have
\[
0=\p_t\left(\ov{u}-\un{u}\right)-d_1\D\left(\ov{u}-\un{u}\right)
	\geq g_1(\ov{u},\ov{v})(x,y,t)-g_1(\ov{u},\un{v})(x,y,t)>0
\]
for all $(x,y)\in\O_0,t\in (t_0,T]$, which is impossible. Conversely, if \ref{equivalence1} holds and $\ov{u}>\un{u}$ on $\ov{\O}_0\times[t_0,+\infty)$ for some $t_0\in[0,T)$, one gets
\[
	0=\p_t\left(\ov{v}-\un{v}\right)-d_2\D\left(\ov{v}-\un{v}\right)
	\geq g_2(\ov{u},\ov{v})(x,y,t)-g_2(\un{u},\ov{v})(x,y,t)>0
\]
for all $(x,y)\in\O_0,t\in (t_0,T]$,
which  is also impossible. Therefore, we obtain the equivalence of \ref{equivalence1} and \ref{equivalence3}.
\end{remark}

\end{document}